\documentclass[reqno,reqno]{amsart}
\usepackage[utf8]{inputenc}
\usepackage{xcolor}
\usepackage{textcomp}
\usepackage{url}
\usepackage{enumitem}
\usepackage{amstext}
\usepackage{amsthm}
\usepackage{amssymb}
\usepackage{graphicx}
\usepackage{esint}
\PassOptionsToPackage{normalem}{ulem}
\usepackage{ulem}
\usepackage[unicode=true,pdfusetitle,
 bookmarks=true,bookmarksnumbered=false,bookmarksopen=false,
 breaklinks=false,pdfborder={0 0 0},pdfborderstyle={},backref=false,colorlinks=true]
 {hyperref}
\hypersetup{
 pdfborderstyle={}}

\makeatletter

\providecolor{lyxadded}{rgb}{0,0,1}
\providecolor{lyxdeleted}{rgb}{1,0,0}

\DeclareRobustCommand{\lyxsout}[1]{\ifx\\#1\else\sout{#1}\fi}

\numberwithin{equation}{section}
\numberwithin{figure}{section}
\theoremstyle{plain}
\newtheorem{thm}{\protect\theoremname}
\theoremstyle{definition}
\newtheorem{defn}[thm]{\protect\definitionname}
\theoremstyle{plain}
\newtheorem{lem}[thm]{\protect\lemmaname}
\theoremstyle{definition}
\newtheorem{example}[thm]{\protect\examplename}
\theoremstyle{remark}
\newtheorem{rem}[thm]{\protect\remarkname}
\theoremstyle{plain}
\newtheorem{cor}[thm]{\protect\corollaryname}

\usepackage{pdfsync}
\usepackage{accents}
\usepackage{graphics}
\usepackage{epstopdf}
\usepackage[all]{xy}
\usepackage{amscd}
\usepackage{enumitem}
\usepackage{mathtools}
\usepackage{bbold}
\usepackage[utf8]{inputenc}
\setlist[enumerate]{leftmargin=*,label=(\roman*),align=left}

\makeatletter
\@namedef{subjclassname@2020}{%
  \textup{2020} Mathematics Subject Classification}
\makeatother

\newcommand{\xyR}[1]{ \makeatletter
\xydef@\xymatrixrowsep@{#1} \makeatother} 
\newcommand{\xyC}[1]{ \makeatletter
\xydef@\xymatrixcolsep@{#1} \makeatother} 
\entrymodifiers={++[ ][F]} 

\newcommand{\ra}{\longrightarrow}

\newcommand{\field}[1]{\mathbb{#1}}
\newcommand{\R}{\field{R}} 
\newcommand{\Cc}{\field{C}} 
\newcommand{\N}{\field{N}} 
\newcommand{\Z}{\ensuremath{\mathbb{Z}}} 

\newcommand{\eps}{\varepsilon} 
\renewcommand{\phi}{\varphi}
\newcommand{\diff}[1]{\ifmmode\mathchoice{\hbox{\rm d}#1}  
 {\hbox{\rm d}#1}  
 {\scalebox{0.75}{$\hbox{\rm d}#1$}}  
 {\scalebox{0.35}{$\hbox{\rm d}#1$}}  
 \fi} 

\newcommand{\abs}[2][\empty]{\ifx#1\empty\left|#2\right|%
\else#1\vert #2 #1\vert\fi}

\newcommand{\cinfty}{\mathcal{C}^\infty}
\newcommand{\Coo}{\mbox{\ensuremath{\mathcal{C}}}^{\infty}} 
\DeclareMathOperator{\Set}{{\bf Set}} 
\newcommand{\Gcinf}{\mathcal{G}\cinfty} 

\newcommand{\Rtil}{\widetilde \R} 
\newcommand{\Ctil}{\widetilde \Cc} 
\newcommand{\gs}{\mathcal{G}^s} 
\newcommand{\ns}{\mathcal{N}^s} 

\newcommand{\sint}[1]{\langle#1\rangle} 
\newcommand{\Eball}{B^{{\scriptscriptstyle \text{\rm E}}}} 

\newcommand{\csp}[1]{{\text{\rm c}}({#1})}
\newcommand{\fcmp}{\Subset_{\text{\rm f}}}
\newcommand{\frontRise}[2]{\ifmmode\mathchoice{{\vphantom{#1}}^{\scalebox{0.6}{$#2$}}}  
 {{\vphantom{#1}}^{\scalebox{0.56}{$#2$}}}  
 {{\vphantom{#1}}^{\scalebox{0.47}{$#2$}}}  
 {{\vphantom{#1}}^{\scalebox{0.35}{$#2$}}}\fi} 
\newcommand{\RC}[1]{\frontRise{\R}{#1}\Rtil}
\newcommand{\rcrho}{\RC{\rho}}
\newcommand{\rti}{\RC{\rho}}
\newcommand{\gsf}{\frontRise{\mathcal{G}}{\rho}\mathcal{GC}^{\infty}}
\newcommand{\Gsf}[1]{\frontRise{\mathcal{G}}{#1}\mathcal{GC}^{\infty}}
\newcommand{\GSFud}[2]{\frontRiseDown{\mathcal{G}}{#1}{#2}\mathcal{GC}^{\infty}}
\newcommand{\Dgsf}{\frontRise{\mathcal{D}}{\rho}\mathcal{GD}}
\newcommand{\DGsf}[1]{\frontRise{\mathcal{D}}{#1}\mathcal{GD}}
\newcommand{\hyperN}[1]{	\frontRise{\N}{#1}\widetilde{\N}}
\newcommand{\hyperZ}[1]{	\frontRise{\N}{#1}\widetilde{\Z}}
\newcommand{\hypNr}{\hyperN{\rho}}
\newcommand{\hypNs}{\hyperN{\sigma}}

\newcommand{\hyperlimarg}[3]{\mathchoice{\frontRise{\lim}{\raisebox{-0.05em}{$#1\hspace{-0.67em}$}}\lim_{#3\in \hyperN{#2}\,}}
{\frontRise{\lim}{#1\hspace{-0.25em}}\lim_{#3\in \hyperN{#2}\,}}
{\frontRise{\lim}{#1\hspace{-0.25em}}\lim_{#3\in \hyperN{#2}\,}}
{\frontRise{\lim}{#1\hspace{-0.25em}}\lim_{#3\in \hyperN{#2}\,}}}
\newcommand{\hyperlim}[2]{\hyperlimarg{#1}{#2}{n}}

\newcommand{\hypersumarg}[3]{\mathchoice{\frontRise{\sum}{\raisebox{-0.2em}{$#1\hspace{-0.67em}$}}\sum_{#3\in \hyperN{#2}\,}}
{\frontRise{\sum}{#1\hspace{-0.25em}}\sum_{#3\in \hyperN{#2}\,}}
{\frontRise{\sum}{#1\hspace{-0.25em}}\sum_{#3\in \hyperN{#2}\,}}
{\frontRise{\sum}{#1\hspace{-0.25em}}\sum_{#3\in \hyperN{#2}\,}}}
\newcommand{\hypersumargZ}[3]{\mathchoice{\frontRise{\sum}{\raisebox{-0.2em}{$#1\hspace{-0.67em}$}}\sum_{#3\in \hyperZ{#2}\,}}
{\frontRise{\sum}{#1\hspace{-0.25em}}\sum_{#3\in \hyperZ{#2}\,}}
{\frontRise{\sum}{#1\hspace{-0.25em}}\sum_{#3\in \hyperZ{#2}\,}}
{\frontRise{\sum}{#1\hspace{-0.25em}}\sum_{#3\in \hyperZ{#2}\,}}}

\newcommand{\hypersum}[2]{\hypersumarg{#1}{#2}{n}}
\newcommand{\hypersumZ}[2]{\hypersumargZ{#1}{#2}{n}}
\newcommand{\subzero}{\subseteq_{0}}
\newcommand{\dom}[1]{\text{dom}(#1)}

\newcommand{\sbpt}[1]{#1_{\text{\rm s}}}

\newcommand{\frontRiseDown}[3]{\ifmmode\mathchoice{{\vphantom{#1}}^{\scalebox{0.6}{$#2$}}_{\scalebox{0.6}{$#3$}}}  
 {{\vphantom{#1}}^{\scalebox{0.56}{$#2$}}_{\scalebox{0.56}{$#3$}}}  
 {{\vphantom{#1}}^{\scalebox{0.47}{$#2$}}_{\scalebox{0.47}{$#3$}}}  
 {{\vphantom{#1}}^{\scalebox{0.35}{$#2$}}_{\scalebox{0.35}{$#3$}}}\fi} 
\newcommand{\RCud}[2]{\frontRiseDown{\R}{#1}{#2}\Rtil}

\newcommand{\GI}{\frontRise{\mathcal{G}}{\rho}\mathcal{GI}} 

\newcommand{\ptind}{\displaystyle \mathop {\ldots\ldots\,}} 
\newcommand{\rccrho}{\frontRise{\mathbb{C}}{\rho}\widetilde{\mathbb{C}}}

\makeatother

\providecommand{\corollaryname}{Corollary}
\providecommand{\definitionname}{Definition}
\providecommand{\examplename}{Example}
\providecommand{\lemmaname}{Lemma}
\providecommand{\remarkname}{Remark}
\providecommand{\theoremname}{Theorem}

\begin{document}

\title{A Fourier transform for all generalized functions}
\author{Akbarali Mukhammadiev \and Diksha tiwari \and Paolo Giordano}
\thanks{A. Mukhammadiev has been supported by Grant P30407 and P33538 of the
Austrian Science Fund FWF}
\address{\textsc{University of Vienna, Austria}}
\email{akbarali.mukhammadiev@univie.ac.at}
\thanks{D.~Tiwari has been supported by Grant P30407 and P33538 of the Austrian
Science Fund FWF}
\address{\textsc{University of Vienna, Austria}}
\email{diksha.tiwari@univie.ac.at}
\thanks{P.~Giordano has been supported by grants P30407, P33538 and P34113
of the Austrian Science Fund FWF}
\address{\textsc{University of Vienna, Austria}}
\email{paolo.giordano@univie.ac.at}
\subjclass[2020]{42B10, 46F12, 46F-xx, 46F30}
\keywords{Fourier transforms, Schwartz distributions, Integral transforms in
distribution spaces, generalized functions for nonlinear analysis.}
\begin{abstract}
Using the existence of infinite numbers $k$ in the non-Archimedean
ring of Robinson-Colombeau, we define the hyperfinite Fourier transform
(HFT) by considering integration extended to $[-k,k]^{n}$ instead
of $(-\infty,\infty)^{n}$. In order to realize this idea, the space
of generalized functions we consider is that of \emph{generalized
smooth functions} (GSF), an extension of classical distribution theory
sharing many nonlinear properties with ordinary smooth functions,
like the closure with respect to composition, a good integration theory,
and several classical theorems of calculus. Even if the final transform
depends on $k$, we obtain a new notion that applies to all GSF, in
particular to all Schwartz's distributions and to all Colombeau generalized
functions defined in $[-k,k]^{n}$, without growth restrictions. We
prove that this FT generalizes several classical properties of the
ordinary FT, and in this way we also overcome the difficulties of
FT in Colombeau's settings. Differences in some formulas, such as
in the transform of derivatives, reveal to be meaningful since allow
to obtain also non-tempered global solutions of differential equations.
\end{abstract}

\maketitle
\tableofcontents{}

\section{Introduction: extending the domain of the Fourier transform}

Fourier transform (FT) and generalized functions (GF) are naturally
interwoven, since the former naturally leads to suitable spaces of
the latter. This already occurs even in trivial cases, such as transforming
a simple sound wave $f(t)=A\sin(2\pi\omega_{0}t)$, whose spectrum
must be, in some way, concentrated at the frequencies $\pm\omega_{0}$.
Even the link between constants and delta-like functions was already
conceived by Fourier (see e.g.~\cite{Lau92}). Although different
theories of generalized functions arise for different motivations,
from distribution theory of Sobolev, Schwartz \cite{Sch45,Sob50}
up to Hairer's regularity structures \cite{Hai}, almost all these
theories are usually augmented with a corresponding calculus of FT,
which can be applied to an appropriate subspace of generalized functions.
Since the beginning of distribution theory, it was hence natural to
try to extend the domain of the FT with less or even with no growth
restrictions imposed. In fact, e.g., as a consequence of these restrictions,
the only solution of the trivial ODE $y'=y$ we can achieve using
tempered distributions is the trivial one. We can hence cite in \cite{GeSh1,GeSh2}
the definition of the FT as the limit of a sequence of functions integrated
on a finite domain, or \cite{Zem65} for a two-sided Laplace transform
defined on a space larger than that of tempered distributions, and
similarly in \cite{AtPiSa} for the directional short-time Fourier
transform of exponential-type distributions. In the same direction
we can inscribe the works \cite{AtMaPi,CaKaPi,DiPr,KaPePi,PiRaTeVi,Teo,Smi,EsFu,EsViYa}
on ultradistributions, hyperfunctions and thick distributions.

On the other hand, problems originating from physics, such as singularities
and point-source fields, also suggest us to consider alternative modeling,
ranging from non-smooth functions as test functions in the theory
of distributions (see e.g.~\cite{Yan} and references therein) to
non-Archimedean analysis (i.e.~mathematical analysis over a ring
extending the real field and containing infinitesimal and/or infinite
numbers, see \cite{GKOS,FGBL}). In the interplay between mathematics
and physics, it is well-known that heuristically manipulating non-linear
pointwise equalities such as $H^{2}=H$ ($H$ being the Heaviside
function) can easily lead to contradictions (see e.g.~\cite{Bow,GKOS}).
This can make particularly difficult to realize the strategy of \cite{LiWe},
where the authors search for a metaplectic representation from symplectic
maps to symplectic relations. According to A.~Weinstein (personal
communication, May 2019), this would require an algebra of generalized
functions extending the usual algebra of smooth functions and a FT
acting on them with the usual inversion formula and transforming the
Dirac delta into $1$. As we will see more diffusely in the following
sections, this is not possible in the classical approach to Colombeau's
algebra, see \cite{Col85,Das91,NedPil92,Hor99}. We will only arrive
at a partial solution of this problem where equalities are replaced
by infinitely closed relations or by limits, see Cor.~\ref{cor:Dirichlet-delta}
and Thm.~\ref{thm:FIT}, Cor.\ref{cor:FIT=00003D_rho}.

To overcome this type of problems, we are going to use the category
of \emph{generalized smooth functions} (GSF), see \cite{GiKu15,GiKu16,LeLuGi17,GiKu18,GIO1}.
This theory seems to be a good candidate, since it is an extension
of classical distribution theory which allows to model nonlinear singular
problems, while at the same time sharing many nonlinear properties
with ordinary smooth functions, like the closure with respect to composition
(thereby, they form an algebra extending the algebra of \emph{smooth}
functions with pointwise product) and several non trivial classical
theorems of the calculus. One could describe GSF as a methodological
restoration of Cauchy-Dirac's original conception of generalized function,
see \cite{Dirac,Lau89,KaTa99}. In essence, the idea of Cauchy and
Dirac (but also of Poisson, Kirchhoff, Helmholtz, Kelvin and Heaviside)
was to view generalized functions as suitable types of smooth set-theoretical
maps obtained from ordinary smooth maps depending on suitable infinitesimal
or infinite parameters. For example, the density of a Cauchy-Lorentz
distribution with an infinitesimal scale parameter was used by Cauchy
to obtain classical properties which nowadays are attributed to the
Dirac delta, cf.~\cite{KaTa99}.

The basic idea to define a very general FT in this setting is the
following: Since GSF form a non-Archimedean framework, we can consider
a positive infinite generalized number $k$ (i.e.~$k>r$ for all
$r\in\R_{>0}$) and define the FT with the usual formula, but integrating
over the $n$-dimensional interval $[-k,k]^{n}$. Although $k$ is
an infinite number (hence, $[-k,k]^{n}\supseteq\R^{n}$), this interval
behaves like a compact set for GSF, so that, e.g., on these domains
we always have an extreme value theorem and integrals always exist.
Clearly, this leads to a FT, called \emph{hyperfinite} FT, that depends
on the parameter $k$, but, on the other hand, where we can transform
\emph{all} the GSF defined on this interval and these include all
tempered Schwartz distributions, all tempered Colombeau GF, but also
a large class of non-tempered GF, such as the exponential functions,
or non-linear examples like $\delta^{a}\circ\delta^{b}$, $\delta^{a}\circ H^{b}$,
$a$, $b\in\N$, etc. Not all the properties of the classical FT remain
unchanged for this more general transform, but the final formalism
still retains the useful properties of the FT in dealing with differential
equations. Even more, the new formula for the transform of derivatives
leads to discover also exponential solutions of the aforementioned
ODE $y'=y$. Since \cite{DeHaPiVa} proves that ultradistributions
and periodic hyperfunctions can be embedded in Colombeau type algebra,
this give strong hints to conjecture that the hyperfinite FT is very
general, and it justifies the title of this article.

The structure of the paper is as follows. We start with an introduction
into the setting of GSF and give basic notions concerning GSF and
their calculus that are needed for a first study of the hyperfinite
FT (Sec.~\ref{sec:Basic-notions}). We then define the hyperfinite
FT in Sec.~\ref{sec:Hyperfinite-Fourier-transform} and the convolution
of compactly supported GSF in Sec.~\ref{sec:Convolution}. In Sec.~\ref{sec:Elementary-properties},
we show how the elementary properties of FT change for the hyperfinite
FT. In Sec.~\ref{sec:The-inverse-hyperfinite} and Sec.~\ref{sec:preservation},
we respectively prove the inversion theorem and that the embedding
of a very large class of Sobolev-Schwartz tempered distributions preserves
their FT, i.e.~that the hyperfinite FT commutes with the embedding
of Schwartz functions and tempered distributions. In this section,
we also recall the problems of FT in the Colombeau's setting and how
we overcome them. Finally, in Sec.~\ref{sec:Examples-and-applications}
we give several examples which underscore the new possibility to transform
any generalized functions. Thanks to the developed formalism, which
stresses the similarities with ordinary smooth functions, frequently
the proofs we are going to present are very simple and similar to
those for smooth functions, but replacing the real field $\R$ with
the non-Archimedean ring of Robinson-Colombeau $\rti$.

The paper is self-contained, in the sense that it contains all the
statements required for the proofs we are going to present. If proofs
of preliminaries are omitted, we clearly give references to where
they can be found. Therefore, to understand this paper, only a basic
knowledge of distribution theory is needed.

\section{Basic notions\label{sec:Basic-notions}}

\subsection{The new ring of scalars}

In this work, $I$ denotes the interval $(0,1]\subseteq\R$ and we
will always use the variable $\eps$ for elements of $I$; we also
denote $\eps$-dependent nets $x\in\R^{I}$ simply by $(x_{\eps})$.
By $\N$ we denote the set of natural numbers, including zero.

We start by defining a new simple non-Archimedean ring of scalars
that extends the real field $\R$. The entire theory is constructive
to a high degree, e.g.~neither ultrafilter nor non-standard method
are used. For all the proofs of results in this section, see \cite{GiKu18,GiKu15,GIO1,GiKu16}.
\begin{defn}
\label{def:RCGN}Let $\rho=(\rho_{\eps})\in(0,1]^{I}$ be a net such
that $(\rho_{\eps})\to0$ as $\eps\to0^{+}$ (in the following, such
a net will be called a \emph{gauge}), then
\begin{enumerate}
\item $\mathcal{I}(\rho):=\left\{ (\rho_{\eps}^{-a})\mid a\in\R_{>0}\right\} $
is called the \emph{asymptotic gauge} generated by $\rho$.
\item If $\mathcal{P}(\eps)$ is a property of $\eps\in I$, we use the
notation $\forall^{0}\eps:\,\mathcal{P}(\eps)$ to denote $\exists\eps_{0}\in I\,\forall\eps\in(0,\eps_{0}]:\,\mathcal{P}(\eps)$.
We can read $\forall^{0}\eps$ as \emph{for $\eps$ small}.
\item We say that a net $(x_{\eps})\in\R^{I}$ \emph{is $\rho$-moderate},
and we write $(x_{\eps})\in\R_{\rho}$ if 
\[
\exists(J_{\eps})\in\mathcal{I}(\rho):\ x_{\eps}=O(J_{\eps})\text{ as }\eps\to0^{+},
\]
i.e., if 
\[
\exists N\in\N\,\forall^{0}\eps:\ |x_{\eps}|\le\rho_{\eps}^{-N}.
\]
\item Let $(x_{\eps})$, $(y_{\eps})\in\R^{I}$, then we say that $(x_{\eps})\sim_{\rho}(y_{\eps})$
if 
\[
\forall(J_{\eps})\in\mathcal{I}(\rho):\ x_{\eps}=y_{\eps}+O(J_{\eps}^{-1})\text{ as }\eps\to0^{+},
\]
that is if 
\begin{equation}
\forall n\in\N\,\forall^{0}\eps:\ |x_{\eps}-y_{\eps}|\le\rho_{\eps}^{n}.\label{eq:negligible}
\end{equation}
This is a congruence relation on the ring $\R_{\rho}$ of moderate
nets with respect to pointwise operations, and we can hence define
\[
\RC{\rho}:=\R_{\rho}/\sim_{\rho},
\]
which we call \emph{Robinson-Colombeau ring of generalized numbers}.
This name is justified by \cite{Rob73,C1}: Indeed, in \cite{Rob73}
A.~Robinson introduced the notion of moderate and negligible nets
depending on an arbitrary fixed infinitesimal $\rho$ (in the framework
of nonstandard analysis); independently, J.F.~Colombeau, cf.~e.g.~\cite{C1}
and references therein, studied the same concepts without using nonstandard
analysis, but considering only the particular gauge $\rho_{\eps}=\eps$.
\end{enumerate}
\end{defn}

We will also use other directed sets instead of $I$: e.g.~$J\subseteq I$
such that $0$ is a closure point of $J$, or $I\times\N$. The reader
can easily check that all our constructions can be repeated in these
cases. We can also define an order relation on $\RC{\rho}$ by saying
that $[x_{\eps}]\le[y_{\eps}]$ if there exists $(z_{\eps})\in\R^{I}$
such that $(z_{\eps})\sim_{\rho}0$ (we then say that $(z_{\eps})$
is \emph{$\rho$-negligible}) and $x_{\eps}\le y_{\eps}+z_{\eps}$
for $\eps$ small. Equivalently, we have that $x\le y$ if and only
if there exist representatives $[x_{\eps}]=x$ and $[y_{\eps}]=y$
such that $x_{\eps}\le y_{\eps}$ for all $\eps$. Although the order
$\le$ is not total, we still have the possibility to define the infimum
$[x_{\eps}]\wedge[y_{\eps}]:=[\min(x_{\eps},y_{\eps})]$, the supremum
$[x_{\eps}]\vee[y_{\eps}]:=\left[\max(x_{\eps},y_{\eps})\right]$
of a finite number of generalized numbers. See \cite{MTAG} for a
complete study of supremum and infimum in $\rti$. Henceforth, we
will also use the customary notation $\RC{\rho}^{*}$ for the set
of invertible generalized numbers, and we write $x<y$ to say that
$x\le y$ and $x-y\in\rcrho^{*}$. Our notations for intervals are:
$[a,b]:=\{x\in\RC{\rho}\mid a\le x\le b\}$, $[a,b]_{\R}:=[a,b]\cap\R$,
and analogously for segments $[x,y]:=\left\{ x+r\cdot(y-x)\mid r\in[0,1]\right\} \subseteq\RC{\rho}^{n}$
and $[x,y]_{\R^{n}}=[x,y]\cap\R^{n}$. We also set $\Cc_{\rho}:=\R_{\rho}+i\cdot\R_{\rho}$
and $\rccrho:=\rti+i\cdot\rti$, where $i=\sqrt{-1}$. On the $\RC{\rho}$-module
$\RC{\rho}^{n}$ we can consider the natural extension of the Euclidean
norm, i.e.~$|[x_{\eps}]|:=[|x_{\eps}|]\in\RC{\rho}$, where $[x_{\eps}]\in\RC{\rho}^{n}$.

As in every non-Archimedean ring, we have the following
\begin{defn}
\label{def:nonArchNumbs}Let $x\in\RC{\rho}^{n}$ be a generalized
number, then
\begin{enumerate}
\item $x$ is \emph{infinitesimal} if $|x|\le r$ for all $r\in\R_{>0}$.
If $x=[x_{\eps}]$, this is equivalent to $\lim_{\eps\to0^{+}}\left|x_{\eps}\right|=0$.
We write $x\approx y$ if $x-y$ is infinitesimal.
\item $x$ is \emph{finite} if $|x|\le r$ for some $r\in\R_{>0}$.
\item $x$ is \emph{infinite} if $|x|\ge r$ for all $r\in\R_{>0}$. If
$x=[x_{\eps}]$, this is equivalent to $\lim_{\eps\to0^{+}}\left|x_{\eps}\right|=+\infty$.
\end{enumerate}
\end{defn}

\noindent For example, setting $\diff{\rho}:=[\rho_{\eps}]\in\RC{\rho}$,
we have that $\diff{\rho}^{n}\in\RC{\rho}$, $n\in\N_{>0}$, is an
invertible infinitesimal, whose reciprocal is $\diff{\rho}^{-n}=[\rho_{\eps}^{-n}]$,
which is necessarily a positive infinite number. Of course, in the
ring $\RC{\rho}$ there exist generalized numbers which are not in
any of the three classes of Def.~\ref{def:nonArchNumbs}, like e.g.~$x_{\eps}=\frac{1}{\eps}\sin\left(\frac{1}{\eps}\right)$.
\begin{defn}
\label{def:stronWeak}We say that $x$ is a \emph{strong infinite
number} if $|x|\ge\diff{\rho}^{-r}$ for some $r\in\R_{>0}$, whereas
we say that $x$ is a \emph{weak infinite number} if $|x|\le\diff{\rho}^{-r}$
for all $r\in\R_{>0}$. For example, $x=-N\log\diff{\rho}$, $N\in\N,$
is a weak infinite number, whereas if $x_{\eps}=\rho_{\eps}^{-1}$
for $\eps=\frac{1}{k}$, $k\in\N_{>0}$, and $x_{\eps}=-\log\rho_{\eps}$
otherwise, then $x$ is neither a strong nor a weak infinite number.
\end{defn}

The following result is useful to deal with positive and invertible
generalized numbers. For its proof, see e.g.~\cite{GKOS}.
\begin{lem}
\label{lem:mayer} Let $x\in\RC{\rho}$. Then the following are equivalent:
\begin{enumerate}
\item \label{enu:positiveInvertible}$x$ is invertible and $x\ge0$, i.e.~$x>0$.
\item \label{enu:strictlyPositive}For each representative $(x_{\eps})\in\R_{\rho}$
of $x$ we have $\forall^{0}\eps:\ x_{\eps}>0$.
\item \label{enu:greater-i_epsTom}For each representative $(x_{\eps})\in\R_{\rho}$
of $x$ we have $\exists m\in\N\,\forall^{0}\eps:\ x_{\eps}>\rho_{\eps}^{m}$.
\item \label{enu:There-exists-a}There exists a representative $(x_{\eps})\in\R_{\rho}$
of $x$ such that $\exists m\in\N\,\forall^{0}\eps:\ x_{\eps}>\rho_{\eps}^{m}$.
\end{enumerate}
\end{lem}

\subsection{Topologies on $\RC{\rho}^{n}$}

As we mentioned above, on the $\RC{\rho}$-module $\RC{\rho}^{n}$
we defined $|[x_{\eps}]|:=[|x_{\eps}|]\in\RC{\rho}$, where $[x_{\eps}]\in\RC{\rho}^{n}$.
Even if this generalized norm takes values in $\RC{\rho}$, it shares
some essential properties with classical norms: 
\begin{align*}
 & |x|=x\vee(-x)\\
 & |x|\ge0\\
 & |x|=0\Rightarrow x=0\\
 & |y\cdot x|=|y|\cdot|x|\\
 & |x+y|\le|x|+|y|\\
 & ||x|-|y||\le|x-y|.
\end{align*}
It is therefore natural to consider on $\RC{\rho}^{n}$ a topology
generated by balls defined by this generalized norm and the set of
radii $\RC{\rho}_{>0}$ of positive invertible numbers:
\begin{defn}
\label{def:setOfRadii}Let $c\in\RC{\rho}^{n}$ then:
\begin{enumerate}
\item $B_{r}(c):=\left\{ x\in\RC{\rho}^{n}\mid\left|x-c\right|<r\right\} $
for each $r\in\rcrho_{>0}$.
\item $\Eball_{r}(c):=\{x\in\R^{n}\mid|x-c|<r\}$, for each $r\in\R_{>0}$,
denotes an ordinary Euclidean ball in $\R^{n}$ if $c\in\R^{n}$.
\end{enumerate}
\end{defn}

\noindent The relation $<$ has better topological properties as compared
to the usual strict order relation $a\le b$ and $a\ne b$ (that we
will \emph{never} use) because the set of balls $\left\{ B_{r}(c)\mid r\in\rcrho_{>0},\ c\in\RC{\rho}^{n}\right\} $
is a base for a topology on $\RC{\rho}^{n}$ called \emph{sharp topology}.
We will call \emph{sharply open set} any open set in the sharp topology.
The existence of infinitesimal neighborhoods (e.g.~$r=\diff{\rho}$)
implies that the sharp topology induces the discrete topology on $\R$.
This is a necessary result when one has to deal with continuous generalized
functions which have infinite derivatives. In fact, if $f'(x_{0})$
is infinite, we have $f(x)\approx f(x_{0})$ only for $x\approx x_{0}$
, see \cite{GiKu18}. Also open intervals are defined using the relation
$<$, i.e.~$(a,b):=\{x\in\RC{\rho}\mid a<x<b\}$.

\subsection{\label{subsec:subpoints}The language of subpoints}

The following simple language allows us to simplify some proofs using
steps that recall the classical real field $\R$, see \cite{MTAG}.
We first introduce the notion of \emph{subpoint}:
\begin{defn}
For subsets $J$, $K\subseteq I$ we write $K\subzero J$ if $0$
is an accumulation point of $K$ and $K\subseteq J$ (we read it as:
$K$ \emph{is co-final in $J$}). Note that for any $J\subzero I$,
the constructions introduced so far in Def.~\ref{def:RCGN} can be
repeated using nets $(x_{\eps})_{\eps\in J}$. We indicate the resulting
ring with the symbol $\rcrho^{n}|_{J}$. More generally, no peculiar
property of $I=(0,1]$ will ever be used in the following, and hence
all the presented results can be easily generalized considering any
other directed set. If $K\subzero J$, $x\in\rcrho^{n}|_{J}$ and
$x'\in\rcrho^{n}|_{K}$, then $x'$ is called a \emph{subpoint} of
$x$, denoted as $x'\subseteq x$, if there exist representatives
$(x_{\eps})_{\eps\in J}$, $(x'_{\eps})_{\eps\in K}$ of $x$, $x'$
such that $x'_{\eps}=x_{\eps}$ for all $\eps\in K$. In this case
we write $x'=x|_{K}$, $\dom{x'}:=K$, and the restriction $(-)|_{K}:\rcrho^{n}\ra\rcrho^{n}|_{K}$
is a well defined operation. In general, for $X\subseteq\rcrho^{n}$
we set $X|_{J}:=\{x|_{J}\in\rcrho^{n}|_{J}\mid x\in X\}$.
\end{defn}

In the next definition, we introduce binary relations that hold only
\emph{on subpoints}. Clearly, this idea is inherited from nonstandard
analysis, where co-final subsets are always taken in a fixed ultrafilter.
\begin{defn}
Let $x$, $y\in\rcrho$, $L\subzero I$, then we say
\begin{enumerate}
\item $x<_{L}y\ :\iff\ x|_{L}<y|_{L}$ (the latter inequality has to be
meant in the ordered ring $\rcrho|_{L}$). We read $x<_{L}y$ as ``\emph{$x$
is less than $y$ on $L$}''.
\item $x\sbpt{<}y\ :\iff\ \exists L\subzero I:\ x<_{L}y$. We read $x\sbpt{<}y$
as ``\emph{$x$ is less than $y$ on subpoints''.}
\end{enumerate}
Analogously, we can define other relations holding only on subpoints
such as e.g.: $=_{L}$, $\in_{L}$, $\sbpt{\in}$, $\sbpt{\le}$,
$\sbpt{=}$, $\sbpt{\subseteq}$, etc.
\end{defn}

\noindent For example, we have
\begin{align*}
x\le y\  & \iff\ \forall L\subzero I:\ x\le_{L}y\\
x<y\  & \iff\ \forall L\subzero I:\ x<_{L}y,
\end{align*}
the former following from the definition of $\le$, whereas the latter
following from Lem.~\ref{lem:mayer}. Moreover, if $\mathcal{P}\left\{ x_{\eps}\right\} $
is an arbitrary property of $x_{\eps}$, then
\begin{equation}
\neg\left(\forall^{0}\eps:\ \mathcal{P}\left\{ x_{\eps}\right\} \right)\ \iff\ \exists L\subzero I\,\forall\eps\in L:\ \neg\mathcal{P}\left\{ x_{\eps}\right\} .\label{eq:negation}
\end{equation}

Note explicitly that, generally speaking, relations on subpoints,
such as $\sbpt{\le}$ or $\sbpt{=}$, do not inherit the same properties
of the corresponding relations for points. So, e.g., both $\sbpt{=}$
and $\sbpt{\le}$ are not transitive relations.

The next result clarifies how to equivalently write a negation of
an inequality or of an equality using the language of subpoints.
\begin{lem}
\label{lem:negationsSubpoints}Let $x$, $y\in\rcrho$, then
\begin{enumerate}
\item \label{enu:neg-le}$x\nleq y\quad\Longleftrightarrow\quad x\sbpt{>}y$
\item \label{enu:negStrictlyLess}$x\not<y\quad\Longleftrightarrow\quad x\sbpt{\ge}y$
\item \label{enu:negEqual}$x\ne y\quad\Longleftrightarrow\quad x\sbpt{>}y$
or $x\sbpt{<}y$
\end{enumerate}
\end{lem}

Using the language of subpoints, we can write different forms of dichotomy
or trichotomy laws for inequality.
\begin{lem}
\label{lem:trich1st}Let $x$, $y\in\rcrho$, then
\begin{enumerate}
\item \label{enu:dichotomy}$x\le y$ or $x\sbpt{>}y$
\item \label{enu:strictDichotomy}$\neg(x\sbpt{>}y$ and $x\le y)$
\item \label{enu:trichotomy}$x=y$ or $x\sbpt{<}y$ or $x\sbpt{>}y$
\item \label{enu:leThen}$x\le y\ \Rightarrow\ x\sbpt{<}y$ or $x=y$
\item \label{enu:leSubpointsIff}$x\sbpt{\le}y\ \Longleftrightarrow\ x\sbpt{<}y$
or $x\sbpt{=}y$.
\end{enumerate}
\end{lem}

\noindent As usual, we note that these results can also be trivially
repeated for the ring $\rcrho|_{L}$. So, e.g., we have $x\not\le_{L}y$
if and only if $\exists J\subzero L:\ x>_{J}y$, which is the analog
of Lem.~\ref{lem:negationsSubpoints}.\ref{enu:neg-le} for the ring
$\rcrho|_{L}$.

\subsection{Open, closed and bounded sets generated by nets}

A natural way to obtain sharply open, closed and bounded sets in $\RC{\rho}^{n}$
is by using a net $(A_{\eps})$ of subsets $A_{\eps}\subseteq\R^{n}$.
We have two ways of extending the membership relation $x_{\eps}\in A_{\eps}$
to generalized points $[x_{\eps}]\in\RC{\rho}^{n}$ (cf.\ \cite{ObVe08,GiKu15}).
\begin{defn}
\label{def:internalStronglyInternal}Let $(A_{\eps})$ be a net of
subsets of $\R^{n}$, then
\begin{enumerate}
\item $[A_{\eps}]:=\left\{ [x_{\eps}]\in\RC{\rho}^{n}\mid\forall^{0}\eps:\,x_{\eps}\in A_{\eps}\right\} $
is called the \emph{internal set} generated by the net $(A_{\eps})$.
\item Let $(x_{\eps})$ be a net of points of $\R^{n}$, then we say that
$x_{\eps}\in_{\eps}A_{\eps}$, and we read it as $(x_{\eps})$ \emph{strongly
belongs to $(A_{\eps})$}, if
\begin{enumerate}
\item $\forall^{0}\eps:\ x_{\eps}\in A_{\eps}$.
\item If $(x'_{\eps})\sim_{\rho}(x_{\eps})$, then also $x'_{\eps}\in A_{\eps}$
for $\eps$ small.
\end{enumerate}
\noindent Moreover, we set $\sint{A_{\eps}}:=\left\{ [x_{\eps}]\in\RC{\rho}^{n}\mid x_{\eps}\in_{\eps}A_{\eps}\right\} $,
and we call it the \emph{strongly internal set} generated by the net
$(A_{\eps})$.
\item We say that the internal set $K=[A_{\eps}]$ is \emph{sharply bounded}
if there exists $M\in\RC{\rho}_{>0}$ such that $K\subseteq B_{M}(0)$.
\item Finally, we say that the $(A_{\eps})$ is a \emph{sharply bounded
net} if there exists $N\in\R_{>0}$ such that $\forall^{0}\eps\,\forall x\in A_{\eps}:\ |x|\le\rho_{\eps}^{-N}$.
\end{enumerate}
\end{defn}

\noindent Therefore, $x\in[A_{\eps}]$ if there exists a representative
$[x_{\eps}]=x$ such that $x_{\eps}\in A_{\eps}$ for $\eps$ small,
whereas this membership is independent from the chosen representative
in case of strongly internal sets. An internal set generated by a
constant net $A_{\eps}=A\subseteq\R^{n}$ will simply be denoted by
$[A]$.

The following theorem (cf.~\cite{ObVe08,GiKu15,GIO1}) shows that
internal and strongly internal sets have dual topological properties:
\begin{thm}
\noindent \label{thm:strongMembershipAndDistanceComplement}For $\eps\in I$,
let $A_{\eps}\subseteq\R^{n}$ and let $x_{\eps}\in\R^{n}$. Then
we have
\begin{enumerate}
\item \label{enu:internalSetsDistance}$[x_{\eps}]\in[A_{\eps}]$ if and
only if $\forall q\in\R_{>0}\,\forall^{0}\eps:\ d(x_{\eps},A_{\eps})\le\rho_{\eps}^{q}$.
Therefore $[x_{\eps}]\in[A_{\eps}]$ if and only if $[d(x_{\eps},A_{\eps})]=0\in\RC{\rho}$.
\item \label{enu:stronglyIntSetsDistance}$[x_{\eps}]\in\sint{A_{\eps}}$
if and only if $\exists q\in\R_{>0}\,\forall^{0}\eps:\ d(x_{\eps},A_{\eps}^{\text{c}})>\rho_{\eps}^{q}$,
where $A_{\eps}^{\text{c}}:=\R^{n}\setminus A_{\eps}$. Therefore,
if $(d(x_{\eps},A_{\eps}^{\text{c}}))\in\R_{\rho}$, then $[x_{\eps}]\in\sint{A_{\eps}}$
if and only if $[d(x_{\eps},A_{\eps}^{\text{c}})]>0$.
\item \label{enu:internalAreClosed}$[A_{\eps}]$ is sharply closed.
\item \label{enu:stronglyIntAreOpen}$\sint{A_{\eps}}$ is sharply open.
\item \label{enu:internalGeneratedByClosed}$[A_{\eps}]=\left[\text{\emph{cl}}\left(A_{\eps}\right)\right]$,
where $\text{\emph{cl}}\left(S\right)$ is the closure of $S\subseteq\R^{n}$.
\item \label{enu:stronglyIntGeneratedByOpen}$\sint{A_{\eps}}=\sint{\text{\emph{int}\ensuremath{\left(A_{\eps}\right)}}}$,
where $\emph{int}\left(S\right)$ is the interior of $S\subseteq\R^{n}$.
\end{enumerate}
\end{thm}

\noindent For example, it is not hard to show that the closure in
the sharp topology of a ball of center $c=[c_{\eps}]$ and radius
$r=[r_{\eps}]>0$ is 
\begin{equation}
\overline{B_{r}(c)}=\left\{ x\in\rti^{d}\mid\left|x-c\right|\le r\right\} =\left[\overline{\Eball_{r_{\eps}}(c_{\eps})}\right],\label{eq:closureBall}
\end{equation}
whereas
\[
B_{r}(c)=\left\{ x\in\rti^{d}\mid\left|x-c\right|<r\right\} =\sint{\Eball_{r_{\eps}}(c_{\eps})}.
\]

\subsection{Generalized smooth functions and their calculus}

Using the ring $\rti$, it is easy to consider a Gaussian with an
infinitesimal standard deviation. If we denote this probability density
by $f(x,\sigma)$, and if we set $\sigma=[\sigma_{\eps}]\in\RC{\rho}_{>0}$,
where $\sigma\approx0$, we obtain the net of smooth functions $(f(-,\sigma_{\eps}))_{\eps\in I}$.
This is the basic idea we are going to develop in the following
\begin{defn}
\label{def:netDefMap}Let $\left(\Omega_{\eps}\right)$ be a net of
open subsets of $\R^{n}$. Let $X\subseteq\RC{\rho}^{n}$ and $Y\subseteq\RC{\rho}^{d}$
be arbitrary subsets of generalized points. Then we say that 
\[
f:X\longrightarrow Y\text{ is a \emph{generalized smooth function}}
\]
if there exists a net $f_{\eps}\in\cinfty(\Omega_{\eps},\R^{d})$
defining the map $f:X\ra Y$ in the sense that
\begin{enumerate}
\item $X\subseteq\langle\Omega_{\eps}\rangle$,
\item $f([x_{\eps}])=[f_{\eps}(x_{\eps})]\in Y$ for all $x=[x_{\eps}]\in X$,
\item $(\partial^{\alpha}f_{\eps}(x_{\eps}))\in\R_{{\scriptscriptstyle \rho}}^{d}$
for all $x=[x_{\eps}]\in X$ and all $\alpha\in\N^{n}$.
\end{enumerate}
The space of generalized smooth functions (GSF) from $X$ to $Y$
is denoted by $\gsf(X,Y)$.
\end{defn}

Let us note explicitly that this definition states minimal logical
conditions to obtain a set-theoretical map from $X$ into $Y$ and
defined by a net of smooth functions of which we can take arbitrary
derivatives still remaining in the space of $\rho$-moderate nets.
In particular, the following Thm.~\ref{thm:propGSF} states that
the equality $f([x_{\eps}])=[f_{\eps}(x_{\eps})]$ is meaningful,
i.e.~that we have independence from the representatives for all derivatives
$[x_{\eps}]\in X\mapsto[\partial^{\alpha}f_{\eps}(x_{\eps})]\in\RC{\rho}^{d}$,
$\alpha\in\N^{n}$.
\begin{thm}
\label{thm:propGSF}Let $X\subseteq\RC{\rho}^{n}$ and $Y\subseteq\RC{\rho}^{d}$
be arbitrary subsets of generalized points. Let $f_{\eps}\in\cinfty(\Omega_{\eps},\R^{d})$
be a net of smooth functions that defines a generalized smooth map
of the type $X\longrightarrow Y$, then
\begin{enumerate}
\item $\forall\alpha\in\N^{n}\,\forall(x_{\eps}),(x'_{\eps})\in\R_{\rho}^{n}:\ [x_{\eps}]=[x'_{\eps}]\in X\ \Rightarrow\ (\partial^{\alpha}f_{\eps}(x_{\eps}))\sim_{\rho}(\partial^{\alpha}f_{\eps}(x'_{\eps}))$.
\item \label{enu:GSF-cont}Each $f\in\gsf(X,Y)$ is continuous with respect
to the sharp topologies induced on $X$, $Y$.
\item \label{enu:globallyDefNet}$f:X\longrightarrow Y$ is a GSF if and
only if there exists a net $v_{\eps}\in\cinfty(\R^{n},\R^{d})$ defining
a generalized smooth map of type $X\longrightarrow Y$ such that $f=[v_{\eps}(-)]|_{X}$.
\item \label{enu:category}GSF are closed with respect to composition, i.e.~subsets
$S\subseteq\RC{\rho}^{s}$ with the trace of the sharp topology, and
GSF as arrows form a subcategory of the category of topological spaces.
We will call this category $\gsf$, the \emph{category of GSF}. Therefore,
with pointwise sum and product, any space $\gsf(X,\rcrho)$ is an
algebra.
\end{enumerate}
\end{thm}

The differential calculus for GSF can be introduced by showing existence
and uniqueness of another GSF serving as incremental ratio (sometimes
this is called \emph{derivative á la Carathéodory}, see e.g.~\cite{Kuh91}).
\begin{thm}[Fermat-Reyes theorem for GSF]
\label{thm:FR-forGSF} Let $U\subseteq\RC{\rho}^{n}$ be a sharply
open set, let $v=[v_{\eps}]\in\RC{\rho}^{n}$, and let $f\in\gsf(U,\RC{\rho})$
be a GSF generated by the net of smooth functions $f_{\eps}\in\cinfty(\Omega_{\eps},\R)$.
Then
\begin{enumerate}
\item \label{enu:existenceRatio}There exists a sharp neighborhood $T$
of $U\times\{0\}$ and a generalized smooth map $r\in\gsf(T,\RC{\rho})$,
called the \emph{generalized incremental ratio} of $f$ \emph{along}
$v$, such that 
\[
\forall(x,h)\in T:\ f(x+hv)=f(x)+h\cdot r(x,h).
\]
\item \label{enu:uniquenessRatio}Any two generalized incremental ratios
coincide on a sharp neighborhood of $U\times\{0\}$, so that we can
use the notation $\frac{\partial f}{\partial v}[x;h]:=r(x,h)$ if
$(x,h)$ are sufficiently small.
\item \label{enu:defDer}We have $\frac{\partial f}{\partial v}[x;0]=\left[\frac{\partial f_{\eps}}{\partial v_{\eps}}(x_{\eps})\right]$
for every $x\in U$ and we can thus define $\diff{f}(x)\cdot v:=\frac{\partial f}{\partial v}(x):=\frac{\partial f}{\partial v}[x;0]=\left[\frac{\partial f_{\eps}}{\partial v_{\eps}}(x_{\eps})\right][x;0]$,
so that $\frac{\partial f}{\partial v}\in\gsf(U,\RC{\rho})$.
\end{enumerate}
\end{thm}

Note that this result permits us to consider the partial derivative
of $f$ with respect to an arbitrary generalized vector $v\in\RC{\rho}^{n}$
which can be, e.g., infinitesimal or infinite. Using recursively this
result, we can also define subsequent differentials $\diff{^{j}}f(x)$
as $j-$multilinear maps, and we set $\diff{^{j}}f(x)\cdot h^{j}:=\diff{^{j}}f(x)(h,\ptind^{j},h)$.
The set of all the $j-$multilinear maps $\left(\rti^{n}\right)^{j}\ra\rti^{d}$
over the ring $\rti$ will be denoted by $L^{j}(\rti^{n},\rti^{d})$.
For $A=[A_{\eps}(-)]\in L^{j}(\rti^{n},\rti^{d})$, we set $\Vert{A}\Vert:=[|{A_{\eps}}|]$,
the generalized number defined by the operator norms of the multilinear
maps $A_{\eps}\in L^{j}(\R^{n},\R^{d})$.

The following result follows from the analogous properties for the
nets of smooth functions defining $f$ and $g$.
\begin{thm}
\label{thm:rulesDer} Let $U\subseteq\rcrho^{n}$ be an open subset
in the sharp topology, let $v\in\rcrho^{n}$ and $f$, $g:U\longrightarrow\rcrho$
be generalized smooth maps. Then
\begin{enumerate}
\item $\frac{\partial(f+g)}{\partial v}=\frac{\partial f}{\partial v}+\frac{\partial g}{\partial v}$
\item $\frac{\partial(r\cdot f)}{\partial v}=r\cdot\frac{\partial f}{\partial v}\quad\forall r\in\rcrho$
\item $\frac{\partial(f\cdot g)}{\partial v}=\frac{\partial f}{\partial v}\cdot g+f\cdot\frac{\partial g}{\partial v}$
\item For each $x\in U$, the map $\diff{f}(x).v:=\frac{\partial f}{\partial v}(x)\in\rcrho$
is $\rcrho$-linear in $v\in\rcrho^{n}$
\item Let $U\subseteq\rcrho^{n}$ and $V\subseteq\rcrho^{d}$ be open subsets
in the sharp topology and $g\in{}^{\rho}\Gcinf(V,U)$, $f\in{}^{\rho}\Gcinf(U,\rcrho)$
be generalized smooth maps. Then for all $x\in V$ and all $v\in\rcrho^{d}$,
we have $\frac{\partial\left(f\circ g\right)}{\partial v}(x)=\diff{f}\left(g(x)\right).\frac{\partial g}{\partial v}(x)$.
\end{enumerate}
\end{thm}

One dimensional integral calculus of GSF is based on the following
\begin{thm}
\label{thm:existenceUniquenessPrimitives}Let $f\in{}^{\rho}\Gcinf([a,b],\rcrho)$
be a GSF defined in the interval $[a,b]\subseteq\RC{\rho}$, where
$a<b$. Let $c\in[a,b]$. Then, there exists one and only one GSF
$F\in{}^{\rho}\Gcinf([a,b],\rcrho)$ such that $F(c)=0$ and $F'(x)=f(x)$
for all $x\in[a,b]$. Moreover, if $f$ is defined by the net $f_{\eps}\in\Coo(\R,\R)$
and $c=[c_{\eps}]$, then $F(x)=\left[\int_{c_{\eps}}^{x_{\eps}}f_{\eps}(s)\diff{s}\right]$
for all $x=[x_{\eps}]\in[a,b]$.
\end{thm}

\noindent We can thus define
\begin{defn}
\label{def:integral}Under the assumptions of Theorem \ref{thm:existenceUniquenessPrimitives},
we denote by $\int_{c}^{(-)}f:=\int_{c}^{(-)}f(s)\,\diff{s}\in{}^{\rho}\Gcinf([a,b],\rcrho)$
the unique GSF such that:
\begin{enumerate}
\item $\int_{c}^{c}f=0$
\item $\left(\int_{u}^{(-)}f\right)'(x)=\frac{\diff{}}{\diff{x}}\int_{u}^{x}f(s)\,\diff{s}=f(x)$
for all $x\in[a,b]$.
\end{enumerate}
\end{defn}

\noindent All the classical rules of integral calculus hold in this
setting:
\begin{thm}
\label{thm:intRules}Let $f\in{}^{\rho}\Gcinf(U,\rcrho)$ and $g\in{}^{\rho}\Gcinf(V,\rcrho)$
be two GSF defined on sharply open domains in $\rcrho$. Let $a$,
$b\in\rcrho$ with $a<b$ and $c$, $d\in[a,b]\subseteq U\cap V$,
then
\begin{enumerate}
\item \label{enu:additivityFunction}$\int_{c}^{d}\left(f+g\right)=\int_{c}^{d}f+\int_{c}^{d}g$
\item \label{enu:homog}$\int_{c}^{d}\lambda f=\lambda\int_{c}^{d}f\quad\forall\lambda\in\rcrho$
\item \label{enu:additivityDomain}$\int_{c}^{d}f=\int_{c}^{e}f+\int_{e}^{d}f$
for all $e\in[a,b]$
\item \label{enu:chageOfExtremes}$\int_{c}^{d}f=-\int_{d}^{c}f$
\item \label{enu:foundamental}$\int_{c}^{d}f'=f(d)-f(c)$
\item \label{enu:intByParts}$\int_{c}^{d}f'\cdot g=\left[f\cdot g\right]_{c}^{d}-\int_{c}^{d}f\cdot g'$
\item \label{enu:intMonotone}If $f(x)\le g(x)$ for all $x\in[a,b]$, then
$\int_{a}^{b}f\le\int_{a}^{b}g$.
\item \label{enu:derUnderInt}Let $a$, $b$, $c$, $d\in\rcrho$, with
$a<b$ and $c<d$, and $f\in\gsf([a,b]\times[c,d],\rcrho^{d})$, then
\[
\frac{\diff{}}{\diff{s}}\int_{a}^{b}f(\tau,s)\,\diff{\tau}=\int_{a}^{b}\frac{\partial}{\partial s}f(\tau,s)\,\diff{\tau}\quad\forall s\in[c,d].
\]
\end{enumerate}
\end{thm}

\begin{thm}
\label{thm:changeOfVariablesInt}Let $f\in{}^{\rho}\Gcinf(U,\rcrho)$
and $\phi\in{}^{\rho}\Gcinf(V,U)$ be GSF defined on sharply open
domains in $\rcrho$. Let $a$, $b\in\rcrho$, with $a<b$, such that
$[a,b]\subseteq V$, $\phi(a)<\phi(b)$, $[\phi(a),\phi(b)]\subseteq U$.
Finally, assume that $\phi([a,b])\subseteq[\phi(a),\phi(b)]$. Then
\[
\int_{\phi(a)}^{\phi(b)}f(t)\diff{t}=\int_{a}^{b}f\left[\phi(s)\right]\cdot\phi'(s)\diff{s}.
\]
\end{thm}

We also have a generalization of Taylor formula:
\begin{thm}
\label{thm:Taylor}Let $f\in\gsf(U,\rcrho)$ be a generalized smooth
function defined in the sharply open set $U\subseteq\rcrho^{d}$.
Let $a$, $b\in\rcrho^{d}$ such that the line segment $[a,b]\subseteq U$,
and set $h:=b-a$. Then, for all $n\in\N$ we have
\begin{enumerate}
\item \label{enu:LagrangeRest}$\exists\xi\in[a,b]:\ f(a+h)=\sum_{j=0}^{n}\frac{\diff{^{j}f}(a)}{j!}\cdot h^{j}+\frac{\diff{^{n+1}f}(\xi)}{(n+1)!}\cdot h^{n+1}.$
\item \label{enu:integralRest}$f(a+h)=\sum_{j=0}^{n}\frac{\diff{^{j}f}(a)}{j!}\cdot h^{j}+\frac{1}{n!}\cdot\int_{0}^{1}(1-t)^{n}\,\diff{^{n+1}f}(a+th)\cdot h^{n+1}\,\diff{t}.$
\end{enumerate}
\noindent Moreover, there exists some $R\in\rcrho_{>0}$ such that
\begin{equation}
\forall k\in B_{R}(0)\,\exists\xi\in[a,a+k]:\ f(a+k)=\sum_{j=0}^{n}\frac{\diff{^{j}f}(a)}{j!}\cdot k^{j}+\frac{\diff{^{n+1}f}(\xi)}{(n+1)!}\cdot k^{n+1}\label{eq:LagrangeInfRest}
\end{equation}
\begin{equation}
\frac{\diff{^{n+1}f}(\xi)}{(n+1)!}\cdot k^{n+1}=\frac{1}{n!}\cdot\int_{0}^{1}(1-t)^{n}\,\diff{^{n+1}f}(a+tk)\cdot k^{n+1}\,\diff{t}\approx0.\label{eq:integralInfRest}
\end{equation}
\end{thm}

Formulas \ref{enu:LagrangeRest} and \ref{enu:integralRest} correspond
to a plain generalization of Taylor's theorem for ordinary smooth
functions with Lagrange and integral remainder, respectively. Dealing
with generalized functions, it is important to note that this direct
statement also includes the possibility that the differential $\diff{^{n+1}f}(\xi)$
may be an infinite number at some point. For this reason, in \eqref{eq:LagrangeInfRest}
and \eqref{eq:integralInfRest}, considering a sufficiently small
increment $k$, we get more classical infinitesimal remainders $\diff{^{n+1}f}(\xi)\cdot k^{n+1}\approx0$.
We can also define right and left derivatives as e.g.~$f'(a):=f'_{+}(a):=\lim_{\substack{t\to a\\
a<t
}
}f'(t)$, which always exist if $f\in\gsf([a,b],\RC{\rho}^{d})$.

\subsection{\label{subsec:Embedding}Embedding of Sobolev-Schwartz distributions
and Colombeau functions}

We finally recall two results that give a certain flexibility in constructing
embeddings of Schwartz distributions. Note that both the infinitesimal
$\rho$ and the embedding of Schwartz distributions have to be chosen
depending on the problem we aim to solve. A trivial example in this
direction is the ODE $y'=y/\diff{\eps}$, which cannot be solved for
$\rho=(\eps)$ (in a finite interval), but it has a solution for $t\le N\left[e^{-1/\eps}\frac{1}{\eps}\right]$
if $\rho=(e^{-1/\eps})$. As another simple example, if we need the
property $H(0)=1/2$, where $H$ is the Heaviside function, then we
have to choose the embedding of distributions accordingly. In other
words, both the gauges and the particular embedding we choose have
to be thought of as elements of the mathematical structure we are
considering to deal with the particular problem we want to solve.
See also \cite{GiLu16,LuGi17} for further details in this direction.\\
 If $\phi\in\mathcal{D}(\R^{n})$, $r\in\R_{>0}$ and $x\in\R^{n}$,
we use the notations $r\odot\phi$ for the function $x\in\R^{n}\mapsto\frac{1}{r^{n}}\cdot\phi\left(\frac{x}{r}\right)\in\R$
and $x\oplus\phi$ for the function $y\in\R^{n}\mapsto\phi(y-x)\in\R$.
These notations permit us to highlight that $\odot$ is a free action
of the multiplicative group $(\R_{>0},\cdot,1)$ on $\mathcal{D}(\R^{n})$
and $\oplus$ is a free action of the additive group $(\R_{>0},+,0)$
on $\mathcal{D}(\R^{n})$. We also have the distributive property
$r\odot(x\oplus\phi)=rx\oplus r\odot\phi$.
\begin{lem}
\label{lem:strictDeltaNet}Let $b\in\rti$ be a net such that $\lim_{\eps\to0^{+}}b_{\eps}=+\infty$.
Let $d\in(0,1)_{\R}$, there exists a net $\left(\psi_{\eps}\right)_{\eps\in I}$
of $\mathcal{D}(\R^{n})$ with the properties:
\begin{enumerate}
\item \label{enu:suppStrictDeltaNet}$supp(\psi_{\eps})\subseteq B_{1}(0)$
and $\psi_{\eps}$ is even for all $\eps\in I$.
\item \label{enu:c_n}Let $\omega_{n}$ denote the surface area of $S^{n-1}$
and set $c_{n}:=\frac{2n}{\omega_{n}}$ for $n>1$ and $c_{1}:=1$,
then $\psi_{\eps}(0)=c_{n}$ for all $\eps\in I$.
\item \label{enu:intOneStrictDeltaNet}$\int\psi_{\eps}=1$ for all $\eps\in I$.
\item \label{enu:moderateStrictDeltaNet}$\forall\alpha\in\N^{n}\exists p\in\mathbb{N}:\ \sup_{x\in\R^{n}}\left|\partial^{\alpha}\psi_{\eps}(x)\right|=O(b_{\eps}^{p})$
as $\eps\to0^{+}$.
\item \label{enu:momentsStrictDeltaNet}$\forall j\in\N\,\forall^{0}\eps:\ 1\le|\alpha|\le j\Rightarrow\int x^{\alpha}\cdot\psi_{\eps}(x)\,\diff{x}=0$.
\item \label{enu:smallNegPartStrictDeltaNet}$\forall\eta\in\R_{>0}\,\forall^{0}\eps:\ \int\left|\psi_{\eps}\right|\le1+\eta$.
\item \label{enu:int1Dim}If $n=1$, then the net $(\psi_{\eps})_{\eps\in I}$
can be chosen so that $\int_{-\infty}^{0}\psi_{\eps}=d$.
\end{enumerate}
\noindent In particular $\psi_{\eps}^{b}:=b_{\eps}^{-1}\odot\psi_{\eps}$
satisfies \ref{enu:intOneStrictDeltaNet} - \ref{enu:smallNegPartStrictDeltaNet}.
For example, for $n=1$, the net $\left(\psi_{\eps}\right)_{\eps\in I}$
can even be taken independently from $\eps$ by setting $\psi:=\mathcal{F}^{-1}(\beta)$,
where $\beta\in\Coo(\R)$ is supported e.g. in $[-1,1]$ and identically
equals $1$ in a neighborhood of $0$.
\end{lem}

\noindent Concerning embeddings of Schwartz distributions, we have
the following result, where $\csp{\Omega}:=\{[x_{\eps}]\in[\Omega]\mid\exists K\Subset\Omega\,\forall^{0}\eps:\ x_{\eps}\in K\}$
is called the set of \emph{compactly supported points in }$\Omega\subseteq\R^{n}$.
Note that $\csp{\Omega}=\left\{ x\in[\Omega]\mid x\text{ is finite}\right\} $
(see Def.~\ref{def:nonArchNumbs}).
\begin{thm}
\label{thm:embeddingD'}Under the assumptions of Lemma \ref{lem:strictDeltaNet},
let $\Omega\subseteq\R^{n}$ be an open set and let $(\psi_{\eps}^{b})$
be the net defined in \ref{lem:strictDeltaNet}. Then the mapping
\begin{equation}
\iota_{\Omega}^{b}:T\in\mathcal{E}'(\Omega)\mapsto\left[\left(T\ast\psi_{\eps}^{b}\right)(-)\right]\in\gsf(\csp{\Omega},\rti)\label{eq:embE'}
\end{equation}
uniquely extends to a sheaf morphism of real vector spaces 
\[
\iota^{b}:\mathcal{D}'\ra\gsf(\csp{-},\rti),
\]
and satisfies the following properties:
\begin{enumerate}
\item \label{enu:embSmooth}If $b\in\rti_{>0}$ is a strong infinite number,
then $\iota^{b}|_{\Coo(-)}:\Coo(-)\ra\gsf(\csp{-},\RC{\rho})$ is
a sheaf morphism of algebras and $\iota_{\Omega}^{b}(f)(x)=f(x)$
for all smooth functions $f\in\Coo(\Omega)$ and all $x\in\Omega$;
\item \label{enu:supportDistr}If $T\in\mathcal{E}'(\Omega)$ then $\text{\text{\emph{supp}}}(T)=\text{\emph{\text{stsupp}}}(\iota_{\Omega}^{b}(T))$,
where
\begin{equation}
\text{\emph{stsupp}}(f):=\left(\bigcup\left\{ \Omega'\subseteq\Omega\mid\Omega'\text{ open},\ f|_{\Omega'}=0\right\} \right)^{\text{c}}\label{eq:stSupp}
\end{equation}
for all $f\in\gsf(\csp{\Omega},\rcrho)$.
\item \label{enu:D'}Let $b\in\rti_{>0}$ be a strong infinite number. Then
$\big[\int_{\Omega}\iota_{\Omega}^{b}(T)_{\eps}(x)\cdot\phi(x)\,\diff{x}\big]=\langle T,\phi\rangle$
for all $\phi\in\mathcal{D}(\Omega)$ and all $T\in\mathcal{D}'(\Omega)$;
\item $\iota^{b}$ commutes with partial derivatives, i.e.~$\partial^{\alpha}\left(\iota_{\Omega}^{b}(T)\right)=\iota_{\Omega}^{b}\left(\partial^{\alpha}T\right)$
for each $T\in\mathcal{D}'(\Omega)$ and $\alpha\in\N$.
\item Similar results also hold for the embedding of tempered distributions:
setting
\[
\mathcal{S}'(\Omega):=\left\{ T\in\mathcal{D}'(\Omega)\mid\exists\tilde{T}\in\mathcal{S}'(\R):\ \tilde{T}|_{\Omega}=T\text{ in }\mathcal{D}'(\Omega)\right\} ,
\]
we have
\[
\iota_{\Omega}^{b}:T\in\mathcal{S}'(\Omega)\mapsto\left[\left(\tilde{T}\ast\psi_{\eps}^{b}\right)|_{\Omega}(-)\right]\in\gsf(\csp{\Omega},\rti),
\]
where $\tilde{T}\in\mathcal{S}'(\R)$, $\tilde{T}|_{\Omega}=T$ in
$\mathcal{D}'(\Omega)$, is any extension of $T$.
\end{enumerate}
\end{thm}

Concerning the embedding of Colombeau generalized functions (CGF),
we recall that the special Colombeau algebra on $\Omega$ is defined
as the quotient $\gs(\Omega):=\mathcal{E}_{M}(\Omega)/\ns(\Omega)$
of \emph{moderate nets} over \emph{negligible nets}, where the former
is 
\[
\mathcal{E}_{M}(\Omega):=\{(u_{\eps})\in\cinfty(\Omega)^{I}\mid\forall K\Subset\Omega\,\forall\alpha\in\N^{n}\,\exists N\in\N:\sup_{x\in K}|\partial^{\alpha}u_{\eps}(x)|=O(\eps^{-N})\}
\]
and the latter is 
\[
\ns(\Omega):=\{(u_{\eps})\in\cinfty(\Omega)^{I}\mid\forall K\Subset\Omega\,\forall\alpha\in\N^{n}\,\forall m\in\N:\sup_{x\in K}|\partial^{\alpha}u_{\eps}(x)|=O(\eps^{m})\}.
\]
Using $\rho=(\eps)$, we have the following compatibility result:
\begin{thm}
\label{thm:inclusionCGF}A Colombeau generalized function $u=(u_{\eps})+\ns(\Omega)^{d}\in\gs(\Omega)^{d}$
defines a GSF $u:[x_{\eps}]\in\csp{\Omega}\longrightarrow[u_{\eps}(x_{\eps})]\in\Rtil^{d}$.
This assignment provides a bijection of $\gs(\Omega)^{d}$ onto $\gsf(\csp{\Omega},\rti^{d})$
for every open set $\Omega\subseteq\R^{n}$.
\end{thm}

\begin{example}
\label{exa:deltaCompDelta}~
\begin{enumerate}
\item \label{enu:deltaH}Let $\delta\in\gsf(\csp{\R^{n}},\rcrho)$ and $H\in\gsf(\csp{\R},\rcrho)$
be the $\iota^{b}$-embeddings of the Dirac delta and of the Heaviside
function. Then $\delta(x)=b^{n}\cdot\psi(b\cdot x)$, where $\psi(x):=[\psi_{\eps}(x_{\eps})]$
is called \emph{$n$-dimensional Colombeau mollifier}. Note that $\delta$
is an even function because of Lem.~\ref{lem:strictDeltaNet}.\ref{enu:suppStrictDeltaNet}.
We have that $\delta(0)=c_{n}b^{n}$ is a strong infinite number and
$\delta(x)=0$ if $|x|>r$ for some $r\in\R_{>0}$ because of Lem.~\ref{lem:strictDeltaNet}.\ref{enu:suppStrictDeltaNet}
(see Lem.~\ref{lem:strictDeltaNet}.\ref{enu:c_n} for the definition
of $c_{n}\in\R_{>0}$). If $n=1$, by the intermediate value theorem
(see \cite{GIO1}), $\delta$ takes any value in the interval $[0,b]\subseteq\rcrho$.
Similar properties can be stated e.g.~for $\delta^{2}(x)=b^{2}\cdot\psi(b\cdot x)^{2}$.
Using these formulas, we can simply consider $\delta\in\gsf(\rti^{n},\rcrho)$
and $H\in\gsf(\rti,\rcrho)$.
\item Analogously, we have $H(x)=1$ if $x>r$ for some $r\in\R_{>0}$;
$H(x)=0$ if $x<-r$ for some $r\in\R_{>0}$, and finally $H(0)=\frac{1}{2}$
because of Lem.~\ref{lem:strictDeltaNet}.\ref{enu:suppStrictDeltaNet}.
By the intermediate value theorem, $H$ takes any value in the interval
$[0,1]\subseteq\rcrho$.
\item If $n=1$, The composition $\delta\circ\delta\in\gsf(\rcrho,\rcrho)$
is given by $(\delta\circ\delta)(x)=b\psi\left(b^{2}\psi(bx)\right)$
and is an even function. If $|x|>r$ for some $r\in\R_{>0}$, then
$(\delta\circ\delta)(x)=b$. Since $(\delta\circ\delta)(0)=0$, again
using the intermediate value theorem, we have that $\delta\circ\delta$
takes any value in the interval $[0,b]\subseteq\rcrho$. Suitably
choosing the net $(\psi_{\eps})$ it is possible to have that if $0\le x\le\frac{1}{kb}$
for some $k\in\N_{>1}$ (hence $x$ is infinitesimal), then $(\delta\circ\delta)(x)=0$.
If $x=\frac{k}{b}$ for some $k\in\N_{>0}$, then $x$ is still infinitesimal
but $(\delta\circ\delta)(x)=b$. Analogously, one can deal with compositions
such as $H\circ\delta$ and $\delta\circ H$.
\end{enumerate}
\noindent See Fig.~\ref{fig:MollifierHeaviside} for a graphical
representations of $\delta$ and $H$. The infinitesimal oscillations
shown in this figure can be proved to actually occur as a consequence
of Lem.~\ref{lem:strictDeltaNet}.\ref{enu:momentsStrictDeltaNet}
which is a necessary property to prove Thm.~\ref{thm:embeddingD'}.\ref{enu:embSmooth},
see \cite{GIO1,GiLu16}. It is well-known that the latter property
is one of the core ideas to bypass the Schwartz's impossibility theorem,
see e.g.~\cite{GKOS}.
\end{example}

\noindent \begin{center}
\begin{figure}
\label{fig: Col_mol}
\begin{centering}
\includegraphics[scale=0.15]{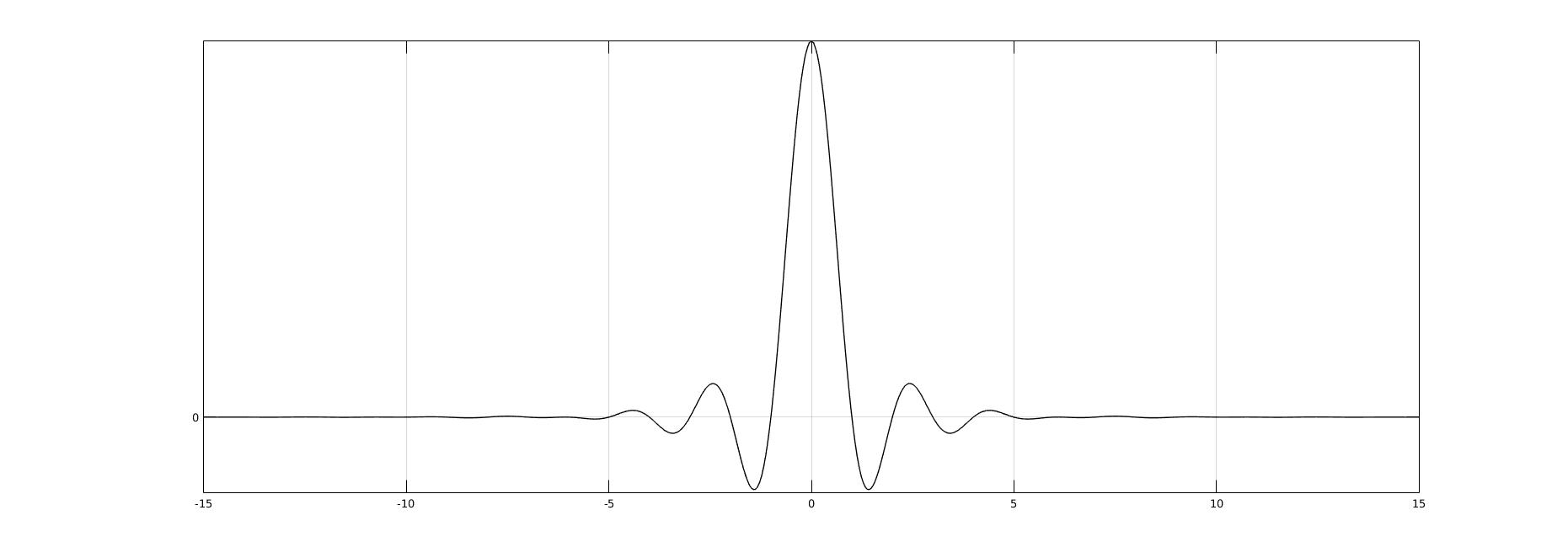}
\par\end{centering}
\begin{centering}
\includegraphics[scale=0.15]{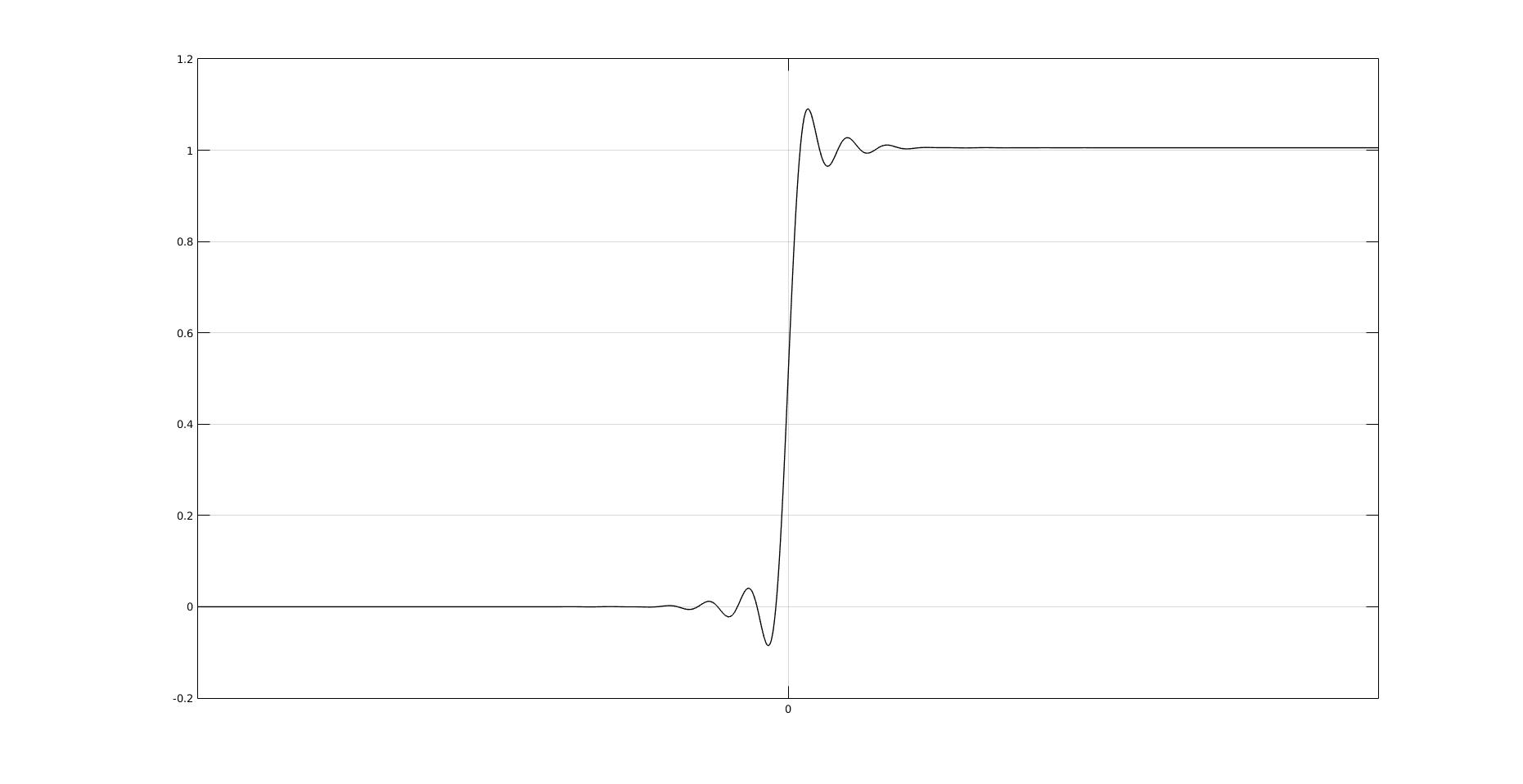}
\par\end{centering}
\caption{\label{fig:MollifierHeaviside}Representations of Dirac delta and
Heaviside function}
\end{figure}
\par\end{center}

\subsection{Functionally compact sets and multidimensional integration}

\subsubsection{\label{subsec:EVTandFcmp}Extreme value theorem and functionally
compact sets}

For GSF, suitable generalizations of many classical theorems of differential
and integral calculus hold: intermediate value theorem, mean value
theorems, suitable sheaf properties, local and global inverse function
theorems, Banach fixed point theorem and a corresponding Picard-Lindelöf
theorem both for ODE and PDE, see \cite{GiKu15,GiKu16,GIO1,LuGi17,GiLu16}.

Even though the intervals $[a,b]\subseteq\rcrho$, $a$, $b\in\R$,
are not compact in the sharp topology (see \cite{GiKu18}), analogously
to the case of smooth functions, a GSF satisfies an extreme value
theorem on such sets. In fact, we have:
\begin{thm}
\label{thm:extremeValues}Let $f\in\Gcinf(X,\rcrho)$ be a GSF defined
on the subset $X$ of $\rcrho^{n}$. Let $\emptyset\ne K=[K_{\eps}]\subseteq X$
be an internal set generated by a sharply bounded net $(K_{\eps})$
of compact sets $K_{\eps}\Subset\R^{n}$ , then 
\begin{equation}
\exists m,M\in K\,\forall x\in K:\ f(m)\le f(x)\le f(M).\label{eq:epsExtreme}
\end{equation}
\end{thm}

We shall use the assumptions on $K$ and $(K_{\eps})$ given in this
theorem to introduce a notion of ``compact subset'' which behaves
better than the usual classical notion of compactness in the sharp
topology.
\begin{defn}
\label{def:functCmpt-1} A subset $K$ of $\rcrho^{n}$ is called
\emph{functionally compact}, denoted by $K\fcmp\rcrho^{n}$, if there
exists a net $(K_{\eps})$ such that
\begin{enumerate}
\item \label{enu:defFunctCmpt-internal-1}$K=[K_{\eps}]\subseteq\rcrho^{n}$.
\item \label{enu:defFunctCmpt-sharpBound-1}$\exists R\in\rcrho_{>0}:\ K\subseteq B_{R}(0)$,
i.e.~$K$ is sharply bounded.
\item \label{enu:defFunctCmpt-cmpt-1}$\forall\eps\in I:\ K_{\eps}\Subset\R^{n}$.
\end{enumerate}
If, in addition, $K\subseteq U\subseteq\rcrho^{n}$ then we write
$K\fcmp U$. Finally, we write $[K_{\eps}]\fcmp U$ if \ref{enu:defFunctCmpt-sharpBound-1},
\ref{enu:defFunctCmpt-cmpt-1} and $[K_{\eps}]\subseteq U$ hold.
Any net $(K_{\eps})$ such that $[K_{\eps}]=K$ is called a \emph{representative}
of $K$.
\end{defn}

\noindent We motivate the name \emph{functionally compact subset}
by noting that on this type of subsets, GSF have properties very similar
to those that ordinary smooth functions have on standard compact sets.
\begin{rem}
\noindent \label{rem:defFunctCmpt}\ 
\begin{enumerate}
\item \label{enu:rem-defFunctCmpt-closed}By Thm.~\ref{thm:strongMembershipAndDistanceComplement}.\ref{enu:internalAreClosed},
any internal set $K=[K_{\eps}]$ is closed in the sharp topology and
hence functionally compact sets are always closed. In particular,
the open interval $(0,1)\subseteq\rcrho$ is not functionally compact
since it is not closed.
\item \label{enu:rem-defFunctCmpt-ordinaryCmpt}If $H\Subset\R^{n}$ is
a non-empty ordinary compact set, then the internal set $[H]$ is
functionally compact. In particular, $[0,1]=\left[[0,1]_{\R}\right]$
is functionally compact.
\item \label{enu:rem-defFunctCmpt-empty}The empty set $\emptyset=\widetilde{\emptyset}\fcmp\rcrho$.
\item \label{enu:rem-defFunctCmpt-equivDef}$\rcrho^{n}$ is not functionally
compact since it is not sharply bounded.
\item \label{enu:rem-defFunctCmpt-cmptlySuppPoints}The set of compactly
supported points $\csp{\R}$ is not functionally compact because the
GSF $f(x)=x$ does not satisfy the conclusion \eqref{eq:epsExtreme}
of Thm.~\ref{thm:extremeValues}.
\end{enumerate}
\end{rem}

\noindent In the present paper, we need the following properties of
functionally compact sets.
\begin{thm}
\label{thm:image}~
\begin{enumerate}
\item Let $K\subseteq X\subseteq\rcrho^{n}$, $f\in\Gcinf(X,\rcrho^{d})$.
Then $K\fcmp\rcrho^{n}$ implies $f(K)\fcmp\rcrho^{d}$.
\item \label{enu:fcmpIntCup}Let $K$, $H\fcmp\rti^{n}$. If $K\cup H$
is an internal set, then it is a functionally compact set. If $K\cap H$
is an internal set, then it is a functionally compact set.
\item \label{enu:fcmpSub}Let $H\subseteq K\fcmp\rti^{n}$, then if $H$
is an internal set, then $H\fcmp\rti^{n}$.
\end{enumerate}
\end{thm}

\noindent As a corollary of this theorem and Rem.\ \eqref{rem:defFunctCmpt}.\ref{enu:rem-defFunctCmpt-ordinaryCmpt}
we get
\begin{cor}
\label{cor:intervalsFunctCmpt}If $a$, $b\in\rcrho$ and $a\le b$,
then $[a,b]\fcmp\rcrho$.
\end{cor}

\noindent Let us note that $a$, $b\in\rcrho$ can also be infinite
numbers, e.g.~$a=\diff{\rho}^{-N}$, $b=\diff{\rho}^{-M}$ or $a=-\diff{\rho}^{-N}$,
$b=\diff{\rho}^{-M}$ with $M>N$, so that e.g.~$[-\diff{\rho}^{-N},\diff{\rho}^{-M}]\supseteq\R$.
Finally, in the following result we consider the product of functionally
compact sets:
\begin{thm}
\noindent \label{thm:product}Let $K\fcmp\rti^{n}$ and $H\fcmp\rti^{d}$,
then $K\times H\fcmp\rti^{n+d}$. In particular, if $a_{i}\le b_{i}$
for $i=1,\ldots,n$, then $\prod_{i=1}^{n}[a_{i},b_{i}]\fcmp\rti^{n}$.
\end{thm}

Applying the extreme value theorem Thm.~\ref{thm:extremeValues}
to the first derivative, we also have the following
\begin{thm}
\label{thm:meanValue}Let $a$, $b\in\rcrho^{n}$, $a<b$, $f\in{}^{\rho}\Gcinf([a,b],\rcrho)$
be a GSF. Then
\begin{enumerate}
\item \label{enu:meanValue}$\exists c\in[a,b]:\ f(b)-f(a)=(b-a)\cdot f'(c)$.
\item \label{enu:finiteIncr}Setting $M:=\max_{c\in[a,b]}\left|f'(c)\right|\in\rti$,
we hence have $\forall x,y\in[a,b]:\ \left|f(x)-f(y)\right|\le M\cdot\left|x-y\right|$.
\end{enumerate}
\end{thm}

A theory of compactly supported GSF has been developed in \cite{GiKu18},
and it closely resembles the classical theory of LF-spaces of compactly
supported smooth functions.

\subsubsection{\label{subsec:Multidimensional-integration}Multidimensional integration}

Finally, to define FT of multivariable GSF we have to introduce multidimensional
integration on suitable subsets of $\rcrho^{n}$ (see \cite{GIO1}).
\begin{defn}
\label{def:intOverCompact}Let $\mu$ be a measure on $\R^{n}$ and
let $K$ be a functionally compact subset of $\RC{\rho}^{n}$. Then,
we call $K$ $\mu$-\emph{measurable} if the limit 
\begin{equation}
\mu(K):=\lim_{m\to\infty}[\mu(\overline{\Eball}_{\rho_{\eps}^{m}}(K_{\eps}))]\label{eq:muMeasurable}
\end{equation}
exists for some representative $(K_{\eps})$ of $K$. Here $m\in\N$,
the limit is taken in the sharp topology on $\RC{\rho}$, and $\overline{\Eball}_{r}(A):=\{x\in\R^{n}:d(x,A)\le r\}$.

\label{def:integrableMap}Let $K\fcmp\RC{\rho}^{n}$. Let $(\Omega_{\eps})$
be a net of open subsets of $\R^{n}$, and $(f_{\eps})$ be a net
of continuous maps $f_{\eps}$: $\Omega_{\eps}\longrightarrow\R$.
Then we say that 
\[
(f_{\eps})\textit{ defines a generalized integrable map}:K\longrightarrow\RC{\rho}
\]
if
\begin{enumerate}
\item $K\subseteq\sint{\Omega_{\eps}}$ and $[f_{\eps}(x_{\eps})]\in\RC{\rho}$
for all $[x_{\eps}]\in K$.
\item $\forall(x_{\eps}),(x'_{\eps})\in\R_{\rho}^{n}:\ [x_{\eps}]=[x'_{\eps}]\in K\ \Rightarrow\ (f_{\eps}(x_{\eps}))\sim_{\rho}(f_{\eps}(x'_{\eps}))$.
\end{enumerate}
\noindent If $f\in\Set(K,\RC{\rho})$ is such that 
\begin{equation}
\forall[x_{\eps}]\in K:\ f\left([x_{\eps}]\right)=\left[f_{\eps}(x_{\eps})\right]
\end{equation}
we say that $f:K\longrightarrow\RC{\rho}$ is a \emph{generalized
integrable function}.

\noindent We will again say that $f$ \emph{is defined by the net}
$(f_{\eps})$ or that the net $(f_{\eps})$ \emph{represents} $f$.
The set of all these generalized integrable functions will be denoted
by $\GI(K,\RC{\rho})$.
\end{defn}

\noindent E.g., if $f=[f_{\eps}(-)]|_{K}\in\gsf(K,\RC{\rho})$, then
both $f$ and $|f|=[|f_{\eps}(-)|]|_{K}$ are integrable on $K$ (but
note that, in general, $|f|$ is not a GSF).

\noindent In the following result, we show that this definition generates
a correct notion of multidimensional integration for GSF.
\begin{thm}
\label{thm:muMeasurableAndIntegral}Let $K\subseteq\RC{\rho}^{n}$
be $\mu$-measurable.
\begin{enumerate}
\item \label{enu:indepRepr}The definition of $\mu(K)$ is independent of
the representative $(K_{\eps})$.
\item \label{enu:existsRepre}There exists a representative $(K_{\eps})$
of $K$ such that $\mu(K)=[\mu(K_{\eps})]$.
\item \label{enu:epsWiseDefInt}Let $(K_{\eps})$ be any representative
of $K$ and let $f=[f_{\eps}(-)]|_{K}\in\GI(K,\RC{\rho})$. Then 
\[
\int_{K}f\,\diff{\mu}:=\lim_{m\to\infty}\biggl[\int_{\overline{\Eball}_{\rho_{\eps}^{m}}(K_{\eps})}f_{\eps}\,\diff{\mu}\biggr]\in\rcrho
\]
exists and its value is independent of the representative $(K_{\eps})$.
\item \label{enu:existsReprDefInt}There exists a representative $(K_{\eps})$
of $K$ such that 
\begin{equation}
\int_{K}f\,\diff{\mu}=\biggl[\int_{K_{\eps}}f_{\eps}\,\diff{\mu}\biggr]\in\rcrho\label{eq:measurable}
\end{equation}
for each $f=[f_{\eps}(-)]|_{K}\in\GI(K,\RC{\rho})$. From \eqref{eq:measurable},
it also follows that $\left|\int_{K}f\,\diff{\mu}\right|\le\int_{K}\left|f\right|\,\diff{\mu}$.
\item \label{enu:int-ndimInt}If $K=\prod_{i=1}^{n}[a_{i},b_{i}]$, then
$K$ is $\lambda$-measurable ($\lambda$ being the Lebesgue measure
on $\R^{n}$) and for all for each $f=[f_{\eps}(-)]|_{K}\in\GI(K,\RC{\rho})$
we have
\begin{equation}
\int_{K}f\,\diff{\lambda}=\left[\int_{a_{1,\eps}}^{b_{1,\eps}}\,dx_{1}\dots\int_{a_{n,\eps}}^{b_{n,\eps}}f_{\eps}(x_{1},\dots,x_{n})\,\diff{x_{n}}\right]\in\rcrho\label{eq:intInt}
\end{equation}
for any representatives $(a_{i,\eps})$, $(b_{i,\eps})$ of $a_{i}$
and $b_{i}$ respectively. Therefore, if $n=1$, this notion of integral
coincides with that of Thm.~\ref{thm:existenceUniquenessPrimitives}
and Def.~\ref{def:integral}. Note that \eqref{eq:intInt} also directly
implies Fubini's theorem for this type of integrals.
\item Let $K\subseteq\RC{\rho}^{n}$ be $\lambda$-measurable, where $\lambda$
is the Lebesgue measure, and let $\phi\in\gsf(K,\RC{\rho}^{d})$ be
such that $\phi^{-1}\in\gsf(\phi(K),\RC{\rho}^{n})$. Then $\phi(K)$
is $\lambda$-measurable and 
\[
\int_{\phi(K)}f\,\diff{\lambda}=\int_{K}(f\circ\phi)\left|\det(\diff{\phi})\right|\,\diff{\lambda}
\]
for each $f\in\GI(\phi(K),\RC{\rho})$.
\end{enumerate}
\end{thm}

\noindent In order to state a continuity property for this notion
of integration, we have to introduce \emph{hypernatural numbers and
hyperlimits} as follows
\begin{defn}
\label{def:hyperfiniteN}~
\begin{enumerate}
\item \label{enu:hypernatural}$\hypNr:=\left\{ [n_{\eps}]\in\rcrho\mid n_{\eps}\in\N\ \forall\eps\right\} $.
Elements of $\hypNr$ are called \emph{hypernatural numbers} or \emph{hyperfinite
numbers}. We clearly have $\N\subseteq\hypNr$, but among hypernatural
numbers we also have infinite numbers.
\item $\N_{\rho}:=\left\{ (n_{\eps})\in\R_{\rho}\mid n_{\eps}\in\N\ \forall\eps\right\} $.
\item \label{enu:hyperlimit}A map $x:\hyperN{\sigma}\ra\rcrho$, whose
domain is the set of hyperfinite numbers $\hyperN{\sigma}$ is called
a ($\sigma-$) \emph{hypersequence} (of elements of $\rcrho$) and
denoted by $(x_{n})_{n\in\hyperN{\sigma}}$, or simply $(x_{n})_{n}$
if the gauge on the domain is clear from the context. Let $\sigma$,
$\rho$ be two gauges, $x:\hyperN{\sigma}\ra\rcrho$ be a hypersequence
and $l\in\rcrho$. We say that $l$ is the \emph{hyperlimit} of $(x_{n})_{n}$
as $n$$\rightarrow\infty$ and $n$$\in\hyperN{\sigma}$, if
\[
\forall q\in N\,\exists M\in\hyperN{\sigma}\,\forall n\in\hyperN{\sigma}_{\geq M}:\ |x_{n}-l|<\diff{\rho}^{q}.
\]
It can be easily proved that there exists at most one hyperlimit,
and in this case it is denoted by $\hyperlim{\rho}{\sigma}x_{n}=l$.
Note that $\diff{\rho}<\frac{1}{n}$ if $n\in\N_{>0}$ so that $\frac{1}{n}\not\to0$
in the sharp topology. On the contrary $\hyperlim{\rho}{\rho}\frac{1}{n}=0$
because $\hypNr$ contains arbitrarily large infinite hypernatural
numbers.
\end{enumerate}
\end{defn}

The following continuity result once again underscores that functionally
compact sets (even if they can be unbounded from a classical point
of view) behaves as compact sets for GSF.
\begin{thm}
\label{thm:contResult}Let $K\subseteq\RC{\rho}^{n}$ be a $\mu$-measurable
functionally compact set and $f_{n}\in\gsf(K,\rcrho^{d})$ for all
$n\in\hypNs$. Then, if the hyperlimit $\hyperlim{\rho}{\sigma}f_{n}(x)$
exists for each $x\in K$, then the convergence is uniform on $K$
and $\hyperlim{\rho}{\sigma}f_{n}\in\gsf(K,\rcrho^{d})$. Finally
\begin{equation}
\hyperlim{\rho}{\sigma}\int_{K}f_{n}\,\diff{x_{n}}=\int_{K}\hyperlim{\rho}{\sigma}f_{n}\,\diff{x_{n}}.\label{eq:ContinuityProperty}
\end{equation}
\end{thm}

\noindent For the proof of this theorem see \cite{GIO1}, and for
the notion of hyperlimit see \cite{MTAG}.

\section{Convolution on $\rcrho^{n}$\label{sec:Convolution}}

In this section, we define and study convolution $f*g$ of two GSF,
where $f$ or $g$ is compactly supported. Compactly supported GSF
were introduced in \cite{GiKu18} for the gauge $\rho_{\eps}=\eps$.
For an arbitrary gauge, we here define and study the notions needed
for the HFT as well as for the study of convolution of GSF.
\begin{defn}
Assume that $X\subseteq\rti^{n}$, $Y\subseteq\rti^{d}$ and $f\in\gsf\left(X,Y\right)$,
then
\begin{enumerate}
\item \label{enu:support}$\text{supp}\left(f\right):=\overline{\left\{ x\in X\mid\left|f\left(x\right)\right|>0\right\} }$,
where $\overline{\left(\cdot\right)}$ denotes the relative closure
in $X$ with respect to the sharp topology, is called the \emph{support}
of $f$. We recall (see just after Def.~\ref{def:RCGN} and Lem.~\ref{lem:mayer})
that $x>0$ means that $x\in\rti_{\ge0}$ is positive and invertible.
\item \label{enu:exterior}For $A\subseteq\rcrho$ we call the set $\text{ext}\left(A\right):=\left\{ x\in\rcrho\mid\forall a\in A:\ \left|x-a\right|>0\right\} $
the \emph{strong exterior of }$A$. Recalling Lem.~\ref{lem:mayer},
if $x\in\text{ext}(A)$, then $|x-a|\ge\diff{\rho}^{q}$ for all $a\in A$
and for some $q=q(a)\in\N$.
\item \label{enu:CSgsf}Let $H\fcmp\rti^{n}$, we say that $f\in\Dgsf\left(H,Y\right)$
if $f\in\gsf(\rti^{n},Y)$ and $\text{supp}\left(f\right)\subseteq H$.
We say that $f\in\Dgsf(\rti^{n},Y)$ if $f\in\Dgsf\left(H,Y\right)$
for some $H\fcmp\rti^{n}$. Such an $f$ is called \emph{compactly
supported}; for simplicity we set $\Dgsf(H):=\Dgsf(H,\rccrho)$. Note
that $\text{supp}(f)$ is clearly always closed, and if $f\in\Dgsf\left(H,Y\right)$
then it is also sharply bounded. However, in general it is not an
internal set so it is not a functionally compact set. Accordingly,
the theory of multidimensional integration of Sec.~\ref{subsec:Multidimensional-integration}
does not allow us to consider $\int_{\text{supp}(f)}f$ even if $f$
is compactly supported.
\end{enumerate}
\end{defn}

\begin{rem}
~
\begin{enumerate}
\item Note that the notion of \emph{standard support} $\text{stsupp}\left(f\right)$
as defined in Thm. \ref{thm:embeddingD'} and the present notion $\text{supp}\left(f\right)$
of support, as defined above, are different. The main distinction
is that $\text{stsupp}\left(f\right)\subseteq\mathbb{R}^{n}$ while
$\text{supp}\left(f\right)\subseteq\rcrho^{n}$. Moreover if we consider
a CGF $f\in\gsf(\csp{\Omega},\rti^{d})$, then $\text{supp}\left(f\right)\cap\Omega\subseteq\text{stsupp}\left(f\right)$.
\item Since $\delta\left(0\right)>0$ then $\delta|_{B_{r}\left(0\right)}>0$
for some $r\in\rcrho_{>0}$ by the sharp continuity of $\delta$,
i.e.~Thm\@.~\ref{thm:propGSF}.\ref{enu:GSF-cont}, hence $B_{r}\left(0\right)\subseteq\text{supp}\left(\delta\right)$,
whereas $\text{stsupp}\left(\delta\right)=\left\{ 0\right\} $. Example
\ref{exa:deltaCompDelta}.\ref{enu:deltaH} also yields that $\text{supp}(\delta)\subseteq[-r,r]^{n}$
for all $r\in\R_{>0}$.
\item \label{enu:rapDecrComptSupp}Any rapidly decreasing function $f\in\mathcal{S}(\R^{n})$
satisfies the inequality $0\leq f\left(x\right)\leq\left|x\right|^{-q}$,
$\forall q\in\mathbb{N}$, for $\left|x\right|$ finite sufficiently
large. Therefore, for all strongly infinite $x$, we have $f\left(x\right)=0$
i.e., $f\in\Dgsf\left(\rcrho^{n}\right)$.
\end{enumerate}
\end{rem}

\begin{lem}
\label{lem:extOpen}Let $\emptyset\ne H\Subset_{f}\rcrho^{n}$. Then
$\text{\emph{ext}}\left(H\right)$ is sharply open.
\end{lem}

\begin{proof}
If $x=\left[x_{\eps}\right]\in\text{ext}\left(H\right)$, we set $d_{\eps}:=d\left(x_{\eps},H_{\eps}\right)$
where $H=\left[H_{\eps}\right]$ and $\emptyset\ne H_{\eps}\Subset\mathbb{R}^{n}$
for all $\eps$ (because $H\ne\emptyset$). Then $\exists h_{\eps}\in H_{\eps}:\ d:=d\left(x_{\eps},h_{\eps}\right)$,
we set $h:=\left[h_{\eps}\right]\in H$ and $\left|x-h\right|=\left[d_{\eps}\right]=:d>0$
because $x\in\text{ext}(H)$ and $h\in H$. Now, by taking $r:=\frac{d}{2}>0$,
we prove that $B_{r}\left(x\right)\subseteq\text{ext}\left(H\right)$.
Pick $y\in B_{r}\left(x\right)$, then for all $a\in H$, we have
$\left|y-a\right|\ge|x-a|-|y-x|\ge d-\frac{d}{2}>0$.
\end{proof}
\begin{thm}
\label{thm:DerivativeIsZero}Let $H\fcmp\rti^{n}$ and $f\in\gsf(\rti^{n},\rccrho)$,
then the following properties hold:
\begin{enumerate}
\item \label{enu:equivCmptSupp}$f\in\Dgsf\left(H\right)$ if and only if
$f|_{\text{\emph{ext}}(H)}=0$.
\end{enumerate}
If $f\in\Dgsf\left(H\right)$, $x\in\rti^{n}$ and $\alpha\in\mathbb{N}^{n}$,
then:
\begin{enumerate}[resume]
\item \label{enu:derZeroExt}$\partial^{\alpha}f\left(x\right)=0$ for all
$x\in\text{\emph{ext}}(H)$.
\item \label{enu:derZeroBound}If $H\subseteq[-h,h]^{n}$ then $\partial^{\alpha}f(x)=0$
whenever $x_{p}\ge h$ or $x_{p}\le-h$ for some $p=1,\ldots,n$.
\item \label{enu:intBound}If $H\subseteq[-h,h]^{n}\subseteq\prod_{p=1}^{n}[a_{p},b_{p}]$,
then
\[
\intop_{a_{1}}^{b_{1}}\,\diff{x_{1}}\ldots\intop_{a_{n}}^{b_{n}}f\left(x\right)\,\diff{x_{n}}=\intop_{-h}^{h}\,\diff{x_{1}}\ldots\intop_{-h}^{h}f\left(x\right)\,\diff{x_{n}}
\]
\end{enumerate}
\end{thm}

\begin{proof}
\ref{enu:equivCmptSupp}: Assume that $\text{supp}(f)\subseteq H$
and $x=[x_{\eps}]\in\text{ext}(H)$, but $f(x)\ne0$. This implies
that $|f(x)|\not\le0$ because always $|f(x)|\ge0$. Thereby, Lem.~\ref{lem:negationsSubpoints}
yields $|f(x)|>_{L}0$ for some $L\subzero I$. Applying Lem.~\ref{lem:mayer}
for the ring $\rti|_{L}$ we get $|f(x)|>_{L}\diff{\rho}^{q}$ for
some $q\in\R_{>0}$, i.e.~$\left|f_{\eps}(x_{\eps})\right|>\rho_{\eps}^{q}$
for all $\eps\in L_{\le\eps_{0}}$. Define $\bar{x}_{\eps}:=x_{\eps}$
for all $\eps\in L$ and $\bar{x}_{\eps}:=x_{\eps_{0}}$ otherwise,
so that $\bar{x}:=[\bar{x}_{\eps}]\in\rti^{n}$ and $|f(\bar{x})|>\diff{\rho}^{q}$.
This yields $\bar{x}\in\text{supp}(f)\subseteq H$, and hence $|x-\bar{x}|>0$,
which is impossible by construction because $\bar{x}|_{L}=x|_{L}$
and because of Lem.~\ref{lem:mayer}.

Vice versa, assume that $f|_{\text{ext}(H)}=0$ and take $x=[x_{\eps}]\in\text{supp}(f)\setminus H$.
The property
\[
\forall q\in\R_{>0}\,\forall^{0}\eps:\ d(x_{\eps},H_{\eps})\le\rho_{\eps}^{q}
\]
cannot hold, because for $q\to+\infty$ Thm\@.~\ref{thm:strongMembershipAndDistanceComplement}.\ref{enu:internalSetsDistance}
would imply $x\in H=[H_{\eps}]$. Therefore, for some $q\in\R_{>0}$
and some $L\subzero I$, we have $d(x_{\eps},H_{\eps})\ge\rho_{\eps}^{q}$
for all $\eps\in L$. Thereby, if $a=[a_{\eps}]\in H$ where $a_{\eps}\in H_{\eps}$
for all $\eps$, we get $d(x_{\eps},a_{\eps})\ge d(x_{\eps},H_{\eps})\ge\rho_{\eps}^{q}$
for all $\eps\in L$, i.e.~$x|_{L}\in\text{ext}(H)|_{L}$. Applying
Lem.~\ref{lem:extOpen} for the ring $\rti|_{L}$ we get
\begin{equation}
B_{r}(x)|_{L}\subseteq\text{ext}(H)|_{L}\label{eq:extLBall}
\end{equation}
for some $r\in\rti_{>0}$. From $x\in\text{supp}(f)$, we get the
existence of a sequence $(x_{p})_{p\in\N}$ of points of $\left\{ x\in\rti^{n}\mid|f(x)|>0\right\} $
such that $x_{p}\to x$ as $p\to+\infty$ in the sharp topology. Therefore,
$x_{p}\in B_{r}(x)$ for $p\in\N$ sufficiently large. Thereby, $x_{p}|_{L}\in\text{ext}(H)|_{L}$
from \eqref{eq:extLBall} and hence $f(x_{p})|_{L}=\left[\left(f_{\eps}(x_{p\eps})\right)_{\eps\in L}\right]=0$,
which contradicts $|f(x_{p})|>0$.

Property \ref{enu:derZeroExt} follows by induction on $|\alpha|\in\N$
using Thm.~\ref{thm:FR-forGSF}. We prove property \ref{enu:derZeroBound}
for the case $x_{p}\ge h$, the other case being similar. We consider
\[
\bar{x}_{q}:=(x_{1},\ptind^{p-1},x_{p-1},x_{p}+\diff{\rho}^{q},x_{p+1},\ldots,x_{n})\quad\forall q\in\N.
\]
Then $|\bar{x}_{q}-a|\ge|x_{p}+\diff{\rho}^{q}-a_{p}|\ge\diff{\rho}^{q}$
for all $a\in[-h,h]^{n}\supseteq H$ because $x_{p}\ge h\ge a_{p}$.
Therefore, $\bar{x}_{q}\in\text{ext}(H)$ and hence $\partial^{\alpha}f(\bar{x}_{q})=0$
from the previous \ref{enu:derZeroExt}. The conclusion now follows
from the sharp continuity of the GSF $\partial^{\alpha}f$ (Thm.~\ref{thm:propGSF}.\ref{enu:GSF-cont}).

\ref{enu:intBound}: The inclusion $\pm(h,\ldots,h)\in[-h,h]^{n}\subseteq\prod_{p=1}^{n}[a_{p},b_{p}]$
implies $a_{p}\le-h$ and $b_{p}\ge h$ for all $p=1,\ldots,n$. Using
Thm.~\ref{thm:muMeasurableAndIntegral}.\ref{enu:int-ndimInt}, we
can write
\begin{align*}
\intop_{a_{1}}^{b_{1}}\,\diff{x_{1}}\ldots\intop_{a_{n}}^{b_{n}}f\left(x\right)\,\diff{x_{n}} & =\intop_{a_{1}}^{b_{1}}\,\diff{x_{1}}\ldots\intop_{a_{n-1}}^{b_{n-1}}\,\diff{x_{n-1}}\intop_{a_{n}}^{-h}f\left(x\right)\,\diff{x_{n}}+\\
 & \phantom{=}\intop_{a_{1}}^{b_{1}}\,\diff{x_{1}}\ldots\intop_{a_{n-1}}^{b_{n-1}}\,\diff{x_{n-1}}\intop_{-h}^{+h}f\left(x\right)\,\diff{x_{n}}+\\
 & \phantom{=}\intop_{a_{1}}^{b_{1}}\,\diff{x_{1}}\ldots\intop_{a_{n-1}}^{b_{n-1}}\,\diff{x_{n-1}}\intop_{h}^{b_{n}}f\left(x\right)\,\diff{x_{n}}.
\end{align*}
But if $x_{n}\in[a_{n},-h]$ or $x_{n}\in[h,b_{n}]$, then property
\ref{enu:derZeroBound} yields $f(x)=0$ and we obtain
\[
\intop_{a_{1}}^{b_{1}}\,\diff{x_{1}}\ldots\intop_{a_{n}}^{b_{n}}f\left(x\right)\,\diff{x_{n}}=\intop_{a_{1}}^{b_{1}}\,\diff{x_{1}}\ldots\intop_{a_{n-1}}^{b_{n-1}}\,\diff{x_{n-1}}\intop_{-h}^{h}f\left(x\right)\,\diff{x_{n}}.
\]
Proceeding in the same way with all the other integrals we get the
claim.
\end{proof}
\noindent In particular, if $T\in\mathcal{E}'(\Omega)$, then Thm.~\ref{thm:DerivativeIsZero}.\ref{enu:equivCmptSupp}
implies that $\iota_{\Omega}^{b}(T)\in\Dgsf\left(\rti^{n}\right)$.
Also observe that $f(x)=e^{-x^{2}}$, $x\in\left\{ x\in\rti\mid\exists N\in\N:\ x^{2}\ge N\log\diff{\rho}\right\} $,
satisfies $f(x)\le x^{-q}$ for all infinite $x$ and all $q\in\N$.
Therefore
\[
\forall Q\in\N:\ f\in\Dgsf\left([-\diff{\rho}^{-Q},\diff{\rho}^{-Q}]\right).
\]

\noindent Based on these results, we can define
\begin{defn}
\label{def:intCmpSupp}Let $f\in\Dgsf(\rti^{n})$, then
\begin{equation}
\int f:=\int_{\rti^{n}}f:=\intop_{a_{1}}^{b_{1}}\,\diff{x_{1}}\ldots\intop_{a_{n}}^{b_{n}}f\left(x\right)\,\diff{x_{n}}\label{eq:intCmpSupp}
\end{equation}
where $\text{supp}(f)\subseteq\prod_{p=1}^{n}[a_{p},b_{p}]$. This
equality does not depend on $a_{p}$, $b_{p}$ because of Thm.~\ref{thm:DerivativeIsZero}.\ref{enu:intBound}.
\end{defn}

\noindent Note that we can also write \eqref{eq:intCmpSupp} as
\begin{equation}
\int f=\lim_{\substack{a_{p}\to-\infty\\
b_{p}\to+\infty\\
p=1,\ldots,n
}
}\ \intop_{a_{1}}^{b_{1}}\,\diff{x_{1}}\ldots\intop_{a_{n}}^{b_{n}}f\left(x\right)\,\diff{x_{n}}=\lim_{h\to+\infty}\intop_{-h}^{h}\,\diff{x_{1}}\ldots\intop_{-h}^{h}f\left(x\right)\,\diff{x_{n}}\label{eq:intCmpSuppLimit}
\end{equation}
even if we are actually considering limits of eventually constant
functions. Using this notion of integral of a compactly supported
GSF, we can also write the value of a distribution $\langle T,\phi\rangle$
as an integral: let $b\in\rti_{>0}$ be a strong infinite number,
$\Omega\subseteq\R^{n}$ be an open set, $T\in\mathcal{D}'(\Omega)$
and $\phi\in\mathcal{D}(\Omega$), with $\text{supp}(\phi)\subseteq\prod_{i=1}^{n}[a_{i},b_{i}]_{\R}=:J$.
Then from Thm.~\ref{thm:embeddingD'}.\ref{enu:D'} and Thm.~\ref{thm:muMeasurableAndIntegral}.\ref{enu:int-ndimInt}
we get
\begin{equation}
\langle T,\phi\rangle=\int_{[J]}\iota_{\Omega}^{b}(T)(x)\cdot\phi(x)\,\diff{x}=\int\iota_{\Omega}^{b}(T)(x)\cdot\phi(x)\,\diff{x},\label{eq:pairTphiAsInt}
\end{equation}
where the equalities are in $\rti$.
\begin{defn}
\label{def:convolution}Let $f$, $g\in\gsf(\rti^{n})$, with $f\in\Dgsf(\rti^{n})$
or $g\in\Dgsf(\rti^{n})$. In the former case, by Thm.~\ref{thm:propGSF}.\ref{enu:category}
and Thm.~\ref{thm:DerivativeIsZero}.\ref{enu:equivCmptSupp}, for
all $x\in\rti^{n}$, $f\cdot g(x-\cdot)\in\Dgsf(\rti^{n})$ with $\text{supp}\left(f\cdot g(x-\cdot)\right)\subseteq\text{supp}(f)\fcmp\rti^{n}$.
Moreover, $\text{supp}\left(f(x-\cdot)\cdot g\right)\subseteq x-\text{supp}(f)\fcmp\rti^{n}$.
Similarly, we can argue in the latter case, and we can hence define
\begin{equation}
\left(f\ast g\right)\left(x\right):=\intop f\left(y\right)g\left(x-y\right)\,\diff{y}=\intop f\left(x-y\right)g\left(y\right)\,\diff{y}\quad\forall x\in\rti^{n}.\label{eq:defConvolution}
\end{equation}

\noindent Note that directly from Thm\@.~\ref{thm:existenceUniquenessPrimitives}
and Def.~\ref{def:intCmpSupp}, it follows that $f*g\in\gsf(\rti^{n})$.
The next theorems provide the usual basic properties of convolution
suitably formulated in our framework. We start by studying how the
convolution is in relation to the supports of its factors:
\end{defn}

\begin{thm}
\label{thm:convSupp}Let $f$, $g$, $h\in\Dgsf(\rti^{n})$. Then
the following properties hold:
\begin{enumerate}
\item \label{enu:supp1}Let $\text{\emph{supp}}(f)\subseteq[-a,a]^{n}$,
$\text{\emph{supp}}(g)\subseteq[-b,b]^{n}$, $a$, $b\in\rti_{>0}$,
and $x\in\rti^{n}$. Set $L_{x}:=[-a,a]^{n}\cap\left(x-[-b,b]^{n}\right)$,
then
\begin{align}
\text{\emph{supp}}\left(f\cdot g(x-\cdot)\right) & \subseteq L_{x}=\prod_{p=1}^{n}[\max(-a,x_{p}-b),\min(a,x_{p}+b)]\label{eq:supp0}\\
\left(f\ast g\right)\left(x\right) & =\intop_{L_{x}}f\left(y\right)g\left(x-y\right)\,\diff{y}.\label{eq:supp1}
\end{align}
\item \label{enu:supp2}$\text{\emph{supp}}(f*g)\subseteq\overline{\text{\emph{supp}}(f)+\text{\emph{supp}}(g)}$,
therefore $f*g\in\Dgsf(\rti^{n})$.
\end{enumerate}
\end{thm}

\begin{proof}
\ref{enu:supp1}: If $|f(t)g(x-t)|>0$, then $t\in\text{supp}(f)$
and $x-t\in\text{supp}(g)$. Therefore, $\text{supp}\left(f\cdot g(x-\cdot)\right)\subseteq[-a,a]^{n}\cap\left(x-[-b,b]^{n}\right)$.
As in the case of real numbers, we can say that if $t\in[-a,a]^{n}\cap\left(x-[-b,b]^{n}\right)$,
then $-a\le t_{p}\le a$ and $-b\le x_{p}-t_{p}\le b$ for all $p=1,\ldots,n$.
Therefore, $t_{p}\in[\max(-a,x_{p}-b),\min(a,x_{p}+b)]$. Similarly,
we can prove that also $L_{x}\subseteq[-a,a]^{n}\cap\left(x-[-b,b]^{n}\right)$.
The conclusion \eqref{eq:supp0} now follows from Def.~\ref{def:intCmpSupp}.
For completeness, recall that in general $\text{supp}(f)$ and $\text{supp}(g)$
are not functionally compact sets and our integration theory allows
to integrate only over the latter kind of sets. This justifies our
formulation of the present property using intervals.

\ref{enu:supp2}: Since $f$ and $g$ are compactly supported, we
have $\text{supp}(f)\subseteq H$ and $\text{supp}(g)\subseteq L$
for some $H$, $L\fcmp\rti^{n}$. Assume that $\left|(f*g)(x)\right|>0$.
Then, by Thm.~\ref{thm:intRules}.\ref{enu:intMonotone}, Thm.~\ref{thm:muMeasurableAndIntegral}.\ref{enu:int-ndimInt}
and the extreme value Thm.~\ref{thm:extremeValues}, we get
\[
0<\left|(f*g)(x)\right|\le\lambda(H)\cdot\max_{y\in H}|f(y)g(x-y)|,
\]
where $\lambda$ is the extension of the Lebesgue measure given by
Def.~\ref{def:intOverCompact}. Therefore, there exists $y\in H$
such that $0<\lambda(H)\cdot\left|f(y)g(x-y)\right|$. This implies
that $y\in\text{supp}(f)$ and $x-y\in\text{supp}(g)$. Thereby, $x=y+(x-y)\in\text{supp}(f)+\text{supp}(g)$.
Taking the sharp closure we get the conclusion. Finally, $\overline{\text{supp}(f)+\text{supp}(g)}\subseteq\overline{H+L}=H+L$
and $H+L\fcmp\rti^{n}$ because it is the image under the sum $+$
of $H\times L$ (see Thm.~\ref{thm:product} and Thm.~\ref{thm:image}).
\end{proof}
Now, we consider algebraic properties of convolution and its relations
with derivations and integration:
\begin{thm}
\label{thm:convAlgDiffInt}Let $f$, $g$, $h\in\gsf(\rti^{n})$ and
assume that at least two of them are compactly supported. Then the
following properties hold:
\begin{enumerate}
\item \label{enu:Commutative:conv}$f\ast g=g*f$.
\item \label{enu:Associative:conv}$\left(f\ast g\right)\ast h=f\ast\left(g\ast h\right)$.
\item \label{enu:Distributive:conc}$f\ast\left(h+g\right)=f\ast h+f\ast g$.
\item \label{enu:complex_conjugate}$\overline{f*g}=\overline{f}*\overline{g}$
\item \label{enu:translation_conv}${\displaystyle t\oplus\left(f*g\right)=\left(t\oplus f\right)*g=f*\left(t\oplus g\right)}$
where $t\oplus f$ is the translation of the function $f$ by $t$
defined by $\left(t\oplus f\right)\left(x\right)=f\left(x-t\right)$
(see Sec.~\ref{subsec:Embedding}).
\item \label{enu:differen_conv}$\frac{\partial}{\partial x_{p}}\left(f*g\right)=\frac{\partial f}{\partial x_{p}}*g=f*\frac{\partial g}{\partial x_{p}}$
for all $p=1,\ldots,n$.
\item \label{enu:integration_conv}If both $f$ and $g$ are compactly supported,
then
\[
\intop\left(f\ast g\right)\left(x\right)\,\diff{x}=\left(\intop f\left(x\right)\,\diff{x}\right)\left(\intop g\left(x\right)\,\diff{x}\right).
\]
\end{enumerate}
\end{thm}

\begin{proof}
\ref{enu:Commutative:conv}: We assume, e.g., that $f\in\Dgsf(\rti^{n})$.
Take $h\in\rti_{>0}$ such that $\text{supp}(f)\subseteq[-h,h]^{n}$.
By \eqref{eq:supp1} and Def.~\ref{def:intCmpSupp}, we can write
\[
\left(f\ast g\right)\left(x\right)=\intop_{-h}^{h}\,\diff{y_{1}}\ldots\intop_{-h}^{h}f\left(y\right)g\left(x-y\right)\,\diff{y_{n}}.
\]
We can now proceed as in the classical case, i.e.~considering the
change of variable $z=x-y$ (Thm.~\ref{thm:changeOfVariablesInt}).
We get
\[
\left(f\ast g\right)\left(x\right)=\intop_{x_{1}-h}^{x_{1}+h}\,\diff{z_{1}}\ldots\intop_{x_{n}-h}^{x_{n}+h}f\left(x-z\right)g\left(z\right)\,\diff{z_{n}}.
\]
Taking the limit $h\to+\infty$ (see \eqref{eq:intCmpSuppLimit}),
we obtain the desired equality. Similarly, we can also prove \ref{enu:Associative:conv}
and \ref{enu:Distributive:conc}.

As usual, \ref{enu:complex_conjugate} is a straightforward consequence
of the definition of complex conjugate.

\ref{enu:translation_conv}: The usual proof applies, in fact
\begin{align}
t\oplus\left(f*g\right)\left(x\right) & =\left(f*g\right)\left(x-t\right)=\intop f\left(y\right)g\left(x-t-y\right)\,\mathrm{d}y=\nonumber \\
 & =\intop f\left(y\right)\left(t\oplus g\right)\left(x-y\right)\,\mathrm{d}y=\left(f\ast\left(t\oplus g\right)\right)\left(x\right).\label{eq:tInside}
\end{align}
Finally, the commutativity property \ref{enu:Commutative:conv} yields
$\left(t\oplus f\right)*g=g*(t\oplus f)$ and applying \eqref{eq:tInside}
$g*(t\oplus f)=t\oplus\left(g*f\right)=t\oplus\left(f*g\right)$.

\ref{enu:differen_conv}: Set $h:=f*g$ and take $x\in\rti^{n}$.
Using differentiation under the integral sign (Thm.~\ref{thm:intRules}.\ref{enu:derUnderInt})
and Def.~\ref{def:intCmpSupp} we get
\[
\frac{\partial}{\partial x_{p}}h\left(x\right)=\intop_{\rcrho^{n}}f\left(y\right)\frac{\partial g}{\partial x_{p}}\left(x-y\right)\,\diff{y}=\left(f\ast\frac{\partial g}{\partial x_{p}}\right)\left(x\right).
\]
Using \ref{enu:Commutative:conv}, we also have $\frac{\partial}{\partial x_{p}}h=\frac{\partial f}{\partial x_{p}}\ast g$.

To prove \ref{enu:integration_conv} we show the case $n=1$, even
if the general one is similar. Let $a$, $b\in\rti_{>0}$ be such
that $\text{supp}(f*g)\subseteq[-a,a]$ (Thm.~\ref{thm:convSupp})
and $\text{supp}(f)\subseteq[-b,b]$. Then
\[
\int(f*g)(x)\,\diff{x}=\int_{-a}^{a}\diff{x}\int_{-b}^{b}f(y)g(x-y)\,\diff{y}.
\]
Using Fubini's Thm.~\ref{thm:muMeasurableAndIntegral}.\ref{enu:int-ndimInt},
we can write
\begin{align*}
\int(f*g)(x)\,\diff{x} & =\int_{-b}^{b}f(y)\int_{-a}^{a}g(x-y)\,\diff{x}\,\diff{y}=\\
 & =\int_{-b}^{b}f(y)\int_{-a-y}^{a-y}g(z)\,\diff{z}\,\diff{y}=\\
 & =\int_{-b}^{b}f(y)\,\diff{y}\int_{-c}^{c}g(z)\,\diff{z},
\end{align*}
where we have taken $a\to+\infty$ or equivalently, considered any
$c\ge a+b$.
\end{proof}
Young's inequality for convolution is based on the generalized Hölder's
inequality, on the inequality $\left|\int_{K}f\,\diff{\mu}\right|\le\int_{K}\left|f\right|\,\diff{\mu}$
(see Thm.~\ref{thm:muMeasurableAndIntegral}.\ref{enu:existsReprDefInt}),
monotonicity of integral (see Thm.~\ref{thm:intRules}.\ref{enu:intMonotone})
and Fubini's theorem (see Thm.~\ref{thm:muMeasurableAndIntegral}.\ref{enu:int-ndimInt}).
Therefore, the usual proofs can be repeated in our setting if we take
sufficient care of terms such as $|f(x)|^{p}$ if $p\in\rti_{\ge1}$:
\begin{defn}
\label{def:pNorm}Let $f\in\Dgsf(\rti^{n})$ and $p\in\rti_{\ge1}$
be a finite number. Then, we set
\[
\Vert f\Vert_{p}:=\left(\int|f(x)|^{p}\,\diff{x}\right)^{1/p}\in\rti_{\ge0}.
\]
Note that $|f|^{p}$ is a generalized integrable function (Def.~\ref{def:intOverCompact})
because $p$ is a finite number (in general the power $x^{y}$ is
not well-defined, e.g.~$\left(\frac{1}{\rho_{\eps}}\right)^{1/\rho_{\eps}}=\rho_{\eps}^{-1/\rho_{\eps}}$
is not $\rho$-moderate).
\end{defn}

\noindent On the other hand, Hölder's inequality, if $\Vert f\Vert_{p}>0$
and $\Vert g\Vert_{q}>0$, is simply based on monotonicity of integral,
Fubini's theorem and Young's inequality for products. The latter holds
also in $\rti_{\ge0}$ because it holds in the entire $\R_{\ge0}$,
see e.g.~\cite{Sch05}.
\begin{thm}[Hölder]
\label{thm:Holder}Let $f_{k}\in\Dgsf(\rti^{n})$ and $p_{k}\in\rti_{\ge1}$
for all $k=1,\ldots,m$ be such that $\sum_{k=1}^{m}\frac{1}{p_{k}}=1$
and $\Vert f_{k}\Vert_{p_{k}}>0$. Then
\[
\left\Vert \prod_{k=1}^{m}f_{k}\right\Vert _{1}\le\prod_{k=1}^{m}\Vert f_{k}\Vert_{p_{k}}.
\]
\end{thm}

\ 
\begin{thm}[Young]
\label{thm:convYoung}~Let $f$, $g\in\Dgsf(\rti^{n})$ and $p$,
$q$, $r\in\rti_{\ge1}$ be such that ${\displaystyle {\textstyle {\frac{1}{p}}+{\frac{1}{q}}=1+{\frac{1}{r}}}}$
and $\Vert f\Vert_{p}$, $\Vert g\Vert_{q}>0$, then $\Vert f*g\Vert_{r}\le\Vert f\Vert_{p}\cdot\Vert g\Vert_{q}$.
\end{thm}

\noindent In the following theorem, we consider when the equality
$\left(\delta*f\right)(x)=f(x)$ holds. As we will see later in Sec.~\ref{sec:The-Riemann-Lebesgue-lemma},
as a consequence of the Riemann-Lebesgue lemma we necessarily have
a limitation concerning the validity of this equality.
\begin{thm}
\label{thm:convIdentity}Let $\delta$ be the $\iota_{\R^{n}}^{b}$-embedding
of the $n$-dimensional Dirac delta (see Thm.~\ref{thm:embeddingD'}).
Assume that $f\in\gsf\left(\rcrho^{n}\right)$ satisfies, at the point
$x\in\rcrho^{n}$, the condition
\begin{equation}
\exists r\in\R_{>0}\,\exists M,c\in\rcrho\,\forall y\in\overline{B}_{r}(x)\,\forall j\in\mathbb{N}:\ \left|\diff{}^{j}f\left(y\right)\right|\leq Mc^{j},\label{eq:regular}
\end{equation}
\[
\frac{b}{c}\text{ is a large infinite number}
\]
i.e. in a \emph{finite }neighborhood of $x$ all its differentials
$\diff{^{j}}f(y)$ are bounded by a suitably small polynomial $Mc^{j}$
(such a function $f$ will be called\emph{ bounded by a tame polynomial
at} $x$). Then $\left(\delta\ast f\right)(x)=f(x)$.
\end{thm}

\begin{proof}
Considering that $\delta(y)=b^{n}\psi(by)$, where $\psi$ is the
considered $n$-dimensional Colombeau mollifier and $b$ is a strong
infinite number. (see Example \ref{exa:deltaCompDelta}.\ref{enu:deltaH}),
we have:
\begin{align*}
\left(\delta\ast f\right)\left(x\right)-f\left(x\right) & =\int f\left(x-y\right)\delta\left(y\right)\,\diff{y}-f\left(x\right)\int\delta\left(y\right)\,\diff{y}=\\
 & =\int\left(f\left(x-y\right)-f\left(x\right)\right)\delta\left(y\right)\,\diff{y}=\\
 & =\int_{\left[-\frac{r}{\sqrt{n}},\frac{r}{\sqrt{n}}\right]^{n}}\left(f\left(x-y\right)-f\left(x\right)\right)\delta\left(y\right)\,\diff{y}=\\
 & =\int_{\left[-\frac{r}{\sqrt{n}},\frac{r}{\sqrt{n}}\right]^{n}}\left(f\left(x-y\right)-f\left(x\right)\right)b^{n}\psi\left(by\right)\,\diff{y},
\end{align*}
 where $r\in\rti_{>0}$ is the radius from \eqref{eq:regular}, so
that $\text{supp}(\delta)\subseteq\left[-\frac{r}{\sqrt{n}},\frac{r}{\sqrt{n}}\right]^{n}$
since $r\in\R_{>0}$. By changing the variable $by=t$, and setting
$H:=\left[-\frac{br}{\sqrt{n}},\frac{br}{\sqrt{n}}\right]^{n}$we
have
\[
\left(f\ast\delta\right)\left(x\right)-f\left(x\right)=\int_{H}\left(f\left(x-\frac{t}{b}\right)-f\left(x\right)\right)\psi\left(t\right)\,\diff{t}.
\]
Using Taylor's formula (Thm.~\ref{thm:Taylor}.\ref{enu:integralRest})
up to an arbitrary order $q\in\N$, we get
\begin{multline}
\int_{H}\left(f\left(x-\frac{t}{b}\right)-f\left(x\right)\right)\psi\left(t\right)\,\diff{t}=\int_{H}\sum_{0<\left|\alpha\right|\leq q}\frac{1}{\alpha!}\left(-\frac{t}{b}\right)^{\alpha}\partial^{\alpha}f\left(x\right)\psi\left(t\right)\,\diff{t}+\\
+\int_{H}\frac{1}{(q+1)!}\int_{0}^{1}\left(1-z\right)^{q}\diff{^{q+1}f}\left(x-z\frac{t}{b}\right)\left(-\frac{t}{b}\right)^{q+1}\psi\left(t\right)\,\diff{z}\,\diff{t}.\label{eq:tay}
\end{multline}
But \ref{enu:suppStrictDeltaNet} and \ref{enu:momentsStrictDeltaNet}
of Lem.~\ref{lem:strictDeltaNet} yield:
\[
\int_{H}t^{\alpha}\psi(t)\,\diff{t}=\left[\int_{\left[-\frac{b_{\eps}r}{\sqrt{n}},\frac{b_{\eps}r}{\sqrt{n}}\right]^{n}}t^{\alpha}\psi_{\eps}(t)\,\diff{t}\right]=\left[\int t^{\alpha}\psi_{\eps}(t)\,\diff{t}\right]=0\quad\forall|\alpha|\le q,
\]
where we also used that $\frac{b_{\eps}r}{\sqrt{n}}>1$ for $\eps$
sufficiently small because $b>0$ is an infinite number and $r\in\R_{>0}$.
Thereby, in \eqref{eq:tay} we only have to consider the remainder
\begin{multline*}
R_{q}\left(x\right):=\int_{H}\frac{1}{(q+1)!}\int_{0}^{1}\left(1-z\right)^{q}\diff{^{q+1}f}\left(x-z\frac{t}{b}\right)\left(-\frac{t}{b}\right)^{q+1}\psi\left(t\right)\,\diff{z}\,\diff{t}=\\
=\frac{(-1)^{q+1}}{b^{q+1}(q+1)!}\int_{H}\int_{0}^{1}\left(1-z\right)^{q}\diff{^{q+1}f}\left(x-z\frac{t}{b}\right)t^{q+1}\psi\left(t\right)\,\diff{z}\,\diff{t}.
\end{multline*}
For all $z\in(0,1)$ and $t\in H=\left[-\frac{r}{\sqrt{n}},\frac{r}{\sqrt{n}}\right]^{n}$,
we have $\left|\frac{zt}{b}\right|\leq\left|\frac{t}{b}\right|\leq\frac{\sqrt{n}|t|_{\infty}}{b}\le\frac{rb}{b}=r$
and hence $x-z\frac{t}{b}\in\overline{B}_{r}(x)$. Thereby, assumption
\eqref{eq:regular} yields $\diff{^{q+1}f}\left(x-z\frac{t}{b}\right)\le Mc^{q+1}$,
and hence 
\begin{align*}
\left|R_{q}\left(x\right)\right| & \leq b^{-q-1}\frac{Mc^{q+1}}{(q+1)!}\intop_{H}\left|t^{q+1}\psi\left(t\right)\right|\,\diff{t}=\\
 & =\left(\frac{b}{c}\right)^{-q-1}\frac{M}{(q+1)!}\intop_{[-1,1]^{n}}\left|t^{q+1}\psi\left(t\right)\right|\,\diff{t}\le\\
 & \le\left(\frac{b}{c}\right)^{-q-1}\frac{M}{(q+1)!}\intop_{[-1,1]^{n}}\left|\psi\left(t\right)\right|\,\diff{t}\le\\
 & \le\left(\frac{b}{c}\right)^{-q-1}\frac{2M}{(q+1)!},
\end{align*}
where we used \ref{enu:suppStrictDeltaNet} and \ref{enu:smallNegPartStrictDeltaNet}
of Lem.~\ref{lem:strictDeltaNet} and $\frac{br}{\sqrt{n}}>1$. We
can now let $q\rightarrow+\infty$ considering that $\frac{b}{c}>\diff{\rho}^{-s}$
for some $s\in\R_{>0}$, so that $\left|R_{q}\left(x\right)\right|\rightarrow0$
and hence $\left(\delta*f\right)(x)=f(x)$.
\end{proof}
\begin{example}
\label{exa:tamePol}~
\begin{enumerate}
\item If $f_{\omega}(x)=e^{-ix\omega}$, $b\ge\diff{\rho}^{-r}$ and $\omega\in\rti$
satisfies $|\omega|\le\diff{\rho}^{-s}$ with $s<r$ (e.g.~if $\omega$
is a weak infinite number, see Def.~\ref{def:stronWeak}), then $\frac{b}{|\omega|}\ge\diff{\rho}^{-(r-s)}$
and $f_{\omega}$ is bounded by a tame polynomial at each point $x\in\rti$.
On the contrary, e.g.~if $b=\diff{\rho}^{-r}$ and $\left|\omega\right|\ge\diff{\rho}^{-r}$,
then $\frac{b}{|\omega|}\le1$ and $f_{\omega}$ is not bounded by
a tame polynomial at any $x\in\rti$.
\item If $f\in\gsf(\rti^{n})$ has always finite derivatives at a finite
point $x\in\rti^{n}$ (e.g.~it originates from the embedding of an
ordinary smooth function), then it suffices to take $c=\diff{\rho}^{-r-1}$
to prove that $f$ is bounded by a tame polynomial at $x$. Similarly,
we can argue if $f$ is polynomially bounded for $x\to\infty$ and
$x\in\rti^{n}$ is not finite.
\item The Dirac delta $\delta(x)=b^{n}\psi(bx)$ is not bounded by a tame
polynomial at $x=0$. This also shows that, generally speaking, the
embedding of a compactly supported distribution is not bounded by
a tame polynomial. Below we will show that indeed $\delta*\delta\ne\delta$,
even if we clearly have $\left(\delta*\delta\right)(x)=\delta(x)=0$
for all $x\in\rti^{n}$ such that $|x|\ge r\in\R_{>0}$.
\item \label{enu:intAgainstDelta}If $f\in\gsf\left(\rcrho^{n}\right)$
is bounded by a tame polynomial at $0$, then since $\delta$ is an
even function (see Example \ref{exa:deltaCompDelta}.\ref{enu:deltaH}),
we have:
\begin{equation}
\int\delta(x)\cdot f(x)\,\diff{x}=\int\delta(0-x)\cdot f(x)\,\diff{x}=\left(f*\delta\right)(0)=f(0).\label{eq:deltaAgainstsFncn}
\end{equation}
\end{enumerate}
\end{example}

Finally, the following theorem considers the relations between convolution
of distributions and their embedding as GSF:
\begin{thm}
\label{thm:compatConv}Let $S\in\mathcal{E}'(\R^{n})$, $T\in\mathcal{D}'(\R^{n})$
and $b\in\rti_{>0}$ be a strong positive infinite number, then for
all $\phi\in\mathcal{D}(\R^{n})$:
\begin{enumerate}
\item \label{enu:convDist1}$\langle S*T,\phi\rangle=\int\iota_{\R^{n}}^{b}(S)(x)\cdot\iota_{\R^{n}}^{b}(T)(y)\cdot\phi(x+y)\,\diff{x}\,\diff{y}=\int\left(\iota_{\R^{n}}^{b}(S)*\iota_{\R^{n}}^{b}(T)\right)(z)\cdot\phi(z)\,\diff{z}.$
\item \label{enu:convDist2}$T*\phi=\iota_{\R^{n}}^{b}(T)*\phi$.
\end{enumerate}
\end{thm}

\begin{proof}
\ref{enu:convDist1}: Using \eqref{eq:pairTphiAsInt}, we have
\begin{align*}
\langle S*T,\phi\rangle & =\langle S(x),\langle T(y),\phi(x+y)\rangle\rangle=\langle S(x),\int\iota_{\R^{n}}^{b}(T)(y)\phi(x+y)\,\diff{y}\rangle=\\
 & =\int\iota_{\R^{n}}^{b}(S)(x)\int\iota_{\R^{n}}^{b}(T)(y)\phi(x+y)\,\diff{y}\,\diff{x}=\\
 & =\int\left(\iota_{\R^{n}}^{b}(S)*\iota_{\R^{n}}^{b}(T)\right)(z)\phi(z)\,\diff{z},
\end{align*}
where, in the last step, we used the change of variables $x=z-y$
and Fubini's theorem.

\ref{enu:convDist2}: For all $x\in\csp{\R^{n}}$, using again \eqref{eq:pairTphiAsInt},
we have $\left(T*\phi\right)(x)=\langle T(y),\phi(x-y)\rangle=\int\iota_{\R^{n}}^{b}(T)(y)\phi(x-y)\,\diff{y}=\left(\iota_{\R^{n}}^{b}(T)*\phi\right)(x)$.
\end{proof}
\noindent We note that an equality of the type $\iota_{\R^{n}}^{b}(S*T)=\iota_{\R^{n}}^{b}(S)*\iota_{\R^{n}}^{b}(T)$
cannot hold because from Thm.~\ref{thm:convAlgDiffInt}.\ref{enu:Associative:conv}
it would imply $1*(\delta'*H)=(1*\delta')*H$ as distributions. Considering
their embeddings, we have $\iota_{\R^{n}}^{b}(1)*\left(\iota_{\R^{n}}^{b}(\delta')*\iota_{\R^{n}}^{b}(H)\right)=\iota_{\R^{n}}^{b}(1)*\left(\iota_{\R^{n}}^{b}(\delta)*\iota_{\R^{n}}^{b}(\delta)\right)=\left(\iota_{\R^{n}}^{b}(1)*\iota_{\R^{n}}^{b}(\delta')\right)*\iota_{\R^{n}}^{b}(H)=\left(\iota_{\R^{n}}^{b}(1')*\iota_{\R^{n}}^{b}(\delta)\right)*\iota_{\R^{n}}^{b}(H)=0$.
In particular, at the term $\iota_{\R^{n}}^{b}(\delta)*\iota_{\R^{n}}^{b}(\delta)$
we cannot apply Thm.~\ref{thm:convIdentity} because $\delta^{(j)}(x)=b^{j+1}\psi^{(j)}(bx)$.
This also implies that $\iota_{\R^{n}}^{b}(\delta)*\iota_{\R^{n}}^{b}(\delta)\ne\iota_{\R^{n}}^{b}(\delta)$
because otherwise we would have $0=\iota_{\R^{n}}^{b}(1)*\left(\iota_{\R^{n}}^{b}(\delta)*\iota_{\R^{n}}^{b}(\delta)\right)=\iota_{\R^{n}}^{b}(1)*\iota_{\R^{n}}^{b}(\delta)=\int\delta=1$.

\section{Hyperfinite Fourier transform\label{sec:Hyperfinite-Fourier-transform}}
\begin{defn}
\label{def:HyperfiniteFT}Let $k\in\rti_{>0}$ be a positive infinite
number. Let $f\in\gsf(K,\rccrho)$, we define the $n$-dimensional
\emph{hyperfinite Fourier transform (HFT) $\mathcal{F}_{k}(f)$ of
$f$} \emph{on }$K:=\left[-k,k\right]^{n}$ as follows: 
\begin{equation}
\mathcal{F}_{k}\left(f\right)\left(\omega\right):=\intop_{K}f\left(x\right)e^{-ix\cdotp\omega}\,\diff{x}=\intop_{-k}^{k}\,\diff{x_{1}}\ldots\intop_{-k}^{k}f\left(x_{1},\ldots,x_{n}\right)e^{-ix\cdotp\omega}\,\diff{x_{n}},\label{eq:Hpfinite_FT}
\end{equation}

\noindent where $x=\left(x_{1}\ldots x_{n}\right)\in K$ and $\omega=\left(\omega_{1}\ldots\omega_{n}\right)\in\rcrho^{n}$.
As usual, the product $x\cdotp\omega$ on $\rcrho^{n}$ denotes the
dot product $x\cdotp\omega=\sum_{j=1}^{n}x_{j}\omega_{j}\in\rti$.
For simplicity, in the following we will also use the notation $\gsf(X):=\gsf(X,\rccrho)$.
If $f\in\Dgsf(X)$ and $\text{supp}(f)\subseteq K=[-k,k]^{n}$, based
on Def.~\ref{def:intCmpSupp}, we can use the simplified notation
$\mathcal{F}(f):=\mathcal{F}_{k}(f)$.
\end{defn}

In the following, $k=[k_{\eps}]\in\rcrho_{>0}$ will always denote
a positive infinite number, and we set $K:=\left[-k,k\right]^{n}\fcmp\rcrho^{n}$.

The adjective \emph{hyperfinite} can be motivated as follows: on the
one hand, $k\in\rti$ is an infinite number, but on the other hand
we already mentioned that GSF behave on a functionally compact set
like $K$ as if it were a compact set. Similarly to the case of hyperfinite
numbers $\hypNr$ (see Def.~\ref{def:hyperfiniteN}), the adjective
\emph{hyperfinite }is frequently used to denote mathematical objects
which are in some sense infinite but behave, from several points of
view, as bounded ones.
\begin{thm}
\label{thm:base}Let $f\in\gsf\left(K\right)$, then the following
properties hold:
\begin{enumerate}
\item \label{enu:epsIntFT}Let $\omega=[\omega_{\eps}]\in\rcrho^{n}$ and
let $f$ be defined by the net $(f_{\eps})$. Then we have: 
\[
\mathcal{F}_{k}\left(f\right)\left(\omega\right)=\left[\intop_{-k_{\eps}}^{k_{\eps}}\,\diff{x_{1}}\ldots\intop_{-k_{\eps}}^{k_{\eps}}f_{\eps}\left(x_{1},\ldots,x_{n}\right)e^{-ix\cdotp\omega_{\eps}}\,\diff{x_{n}}\right]=\left[\hat{\mathcal{F}}(\chi_{K_{\eps}}f_{\eps})(\omega_{\eps})\right]\in\rccrho,
\]
where $\hat{\mathcal{F}}:\mathcal{S}(\R^{n})\ra\mathcal{S}(\R^{n})$
is the classical FT, and $\chi_{K_{\eps}}$ is the characteristic
function of $K_{\eps}$.
\item \label{enu:bound}$\forall\omega\in\rti^{n}:\ \left|\mathcal{F}_{k}(f)(\omega)\right|\le\int_{K}\left|f(x)\right|\,\diff{x}=\Vert f\Vert_{1}$,
so that the HFT is always sharply bounded.
\item \label{enu:Fourier_Map}$\mathcal{F}_{k}:\gsf\left(K\right)\longrightarrow\gsf\left(\rcrho^{n}\right)$.
\end{enumerate}
\end{thm}

\begin{proof}
\ref{enu:epsIntFT}: For all $\omega\in\rcrho^{n}$ fixed, the map
$x\in K\mapsto f\left(x\right)e^{-ix\cdotp\omega}$ is a GSF by the
closure with respect to composition, i.e.~Thm.~\ref{thm:propGSF}.\ref{enu:category}.
Therefore, we can apply Thm.~\ref{thm:muMeasurableAndIntegral}.\ref{enu:int-ndimInt}.

To prove \ref{enu:Fourier_Map}, we have to show that $\mathcal{F}_{k}(f):\rcrho^{n}\ra\rccrho$
is defined by a net $\left(\mathcal{F}_{k}\right)_{\eps}\in\cinfty\left(\mathbb{R}^{n},\Cc\right)$
(see Def.~\ref{def:netDefMap}). We can naturally define such a net
as
\[
\left(\mathcal{F}_{k}\right)_{\eps}\left(y\right):=\intop_{-k_{\eps}}^{k_{\eps}}\,\diff{x_{1}}\ldots\intop_{-k_{\eps}}^{k_{\eps}}f_{\eps}\left(x_{1},\ldots,x_{n}\right)e^{-ix\cdotp y}\,\diff{x_{n}}\quad\forall y\in\R^{n},
\]
and we claim it satisfies the following properties:
\begin{enumerate}[label=(\alph*)]
\item \label{enu:hypotheis1}$\left[\left(\mathcal{F}_{k}\right)_{\eps}\left(\omega_{\eps}\right)\right]\in\rccrho$,
$\forall\left[\omega_{\eps}\right]\in\rcrho^{n}$.
\item \label{enu:hypothesis2}$\forall\left[\omega_{\eps}\right]\in\rcrho^{n}\,\forall\alpha\in\mathbb{N}^{n}:\ \left(\partial^{\alpha}\left(\mathcal{F}_{k}\right)_{\eps}\left(\omega_{\eps}\right)\right)\in\Cc_{\rho}$.
\end{enumerate}
Claim \ref{enu:hypotheis1} is justified by \ref{enu:epsIntFT} above.
From \ref{enu:epsIntFT} it directly follows \ref{enu:bound}. In
order to prove \ref{enu:hypothesis2}, we use the standard derivation
under the integral sign to have
\[
\partial^{\alpha}\left(\mathcal{F}_{k}\right)_{\eps}\left(\omega_{\eps}\right)=\intop_{-k_{\eps}}^{k_{\eps}}\,\diff{x_{1}}\ldots\intop_{-k_{\eps}}^{k_{\eps}}f_{\eps}\left(x_{1},\ldots,x_{n}\right)e^{-ix\cdotp\omega_{\eps}}(-ix^{\alpha})\,\diff{x_{n}}.
\]
We can now proceed as above to prove \ref{enu:hypothesis2} and hence
the claim \ref{enu:Fourier_Map}.
\end{proof}

\subsection{\label{subsec:The-heuristic-motivation}The heuristic motivation
of the FT in a non-Archimedean setting}

Frequently, the formula for the definition of the FT (e.g.~for rapidly
decreasing functions) is informally motivated using its relations
with Fourier series. In order to replicate a similar argument for
GSF, we need the notion of \emph{hyperseries}. In fact, exactly as
the ordinary limit $\lim_{n\in\N}a_{n}$ is not well suited for the
sharp topology (because of its infinitesimal neighbourhoods) and we
have to consider hyperlimits $\hyperlim{\rho}{\sigma}a_{n}$ (see
Def.~\ref{def:hyperfiniteN}.\ref{enu:hyperlimit}), likewise to
study series of $a_{n}\in\rccrho$, $n\in\N$, we have to consider
\begin{align*}
\hypersum{\rho}{\sigma}a_{n} & :=\hyperlimarg{\rho}{\sigma}{N}\sum_{n=0}^{N}a_{n}\in\rccrho,\\
\hypersumZ{\rho}{\sigma}a_{n} & :=\hyperlimarg{\rho}{\sigma}{N}\sum_{n=-N}^{N}a_{n}\in\rccrho,
\end{align*}
where $\hyperZ{\sigma}:=\hypNs\cup\left(-\hypNs\right)\subseteq\RC{\sigma}$.
The main problem in this definition is how to define the \emph{hyperfinite
sums} $\sum_{n=M}^{N}a_{n}\in\rccrho$ for arbitrary hypernatural
numbers $N$, $M\in\hypNs$ and starting from suitable \emph{ordinary}
sequences $\left(a_{n}\right)_{n\in\N}$ of $\rccrho$. However, this
can be done, and the resulting notion extends several classical theorems,
see \cite{TiwGio21}.

Only for this section, we hence assume that $f\in\Dgsf([-T,T])$,
$T\in\rti_{>0}$, can be written as a Fourier hyperseries
\[
f(t)=\hypersumZ{\rho}{\sigma}c_{n}e^{2\pi i\frac{n}{T}t}\quad\forall t\in(-T,T),
\]
where $\sigma$ is another gauge such that $\sigma_{\eps}\le\rho_{\eps}^{q}$
for all $q\in\N$ and for $\eps$ small (so that $\R_{\rho}\subseteq\R_{\sigma}$,
see Def.~\ref{def:RCGN}). Using Thm.~\ref{thm:contResult} to exchange
hyperseries and integration, for each $h\in\hyperZ{\sigma}$, we have
\[
\int_{-T}^{T}f(t)e^{-2\pi i\frac{h}{T}t}\,\diff{t}=\hypersumZ{\rho}{\sigma}c_{n}\int_{-T}^{T}e^{2\pi i\frac{t}{T}(n-h)}\,\diff{t}=2T\cdot c_{h}.
\]
That is $c_{h}=\frac{1}{2T}\mathcal{F}(f)\left(2\pi\frac{h}{T}\right)$.

It is also well-known that, informally, if $T$ is ``sufficiently
large'', then the Fourier coefficients $c_{n}$ ``approximate''
the FT scaled by $\frac{1}{2T}$ and dilated by $2\pi$. Using our
non-Archimedean language, this can be formalized as follows: Let $\omega=[\omega_{\eps}]\in\rti$,
and assume that $T=[T_{\eps}]$ is an infinite number, then setting
$h_{\omega}:=\left[\text{int}\left(\omega_{\eps}\cdot T_{\eps}\right)\right]\in\hyperZ{\rho}$
(here we use $\R_{\rho}\subseteq\R_{\sigma}$), we have $\omega_{\eps}\le\frac{h_{\omega\eps}}{T_{\eps}}\le\omega_{\eps}+\frac{1}{T_{\eps}}$,
so that $\frac{h_{\omega}}{T}\approx\omega$ because $T$ is an infinite
number. By Thm.~\ref{thm:base}, $\mathcal{F}(f)$ is a GSF. Let
$a$, $b$, $c$, $d\in$$\rti$, with $a<c<d<b$, and set $M:=\max_{\omega\in[2\pi a,2\pi b]}\mathcal{F}(f)'(\omega)$.
Using Lem.~\ref{lem:mayer}, we can find $q\in\N$ such that $c-a\ge\diff{\rho}^{q}$
and $b-d\ge\diff{\rho}^{q}$. Assume that $T$ is sufficiently large
so that the following conditions hold
\[
\frac{1}{T}\le\diff{\rho}^{q},\quad\frac{M}{T}\approx0.
\]
Then, for all $\omega\in[c,d]$, we have $\frac{h_{\omega}}{T}\le\omega+\frac{1}{T}\le d+\diff{\rho}^{q}\le b$,
and $\frac{h_{\omega}}{T}\ge\omega\ge c>a$, so that $\frac{h_{\omega}}{T}$,
$\omega\in[a,b]$. From the mean value theorem Thm.~\ref{thm:meanValue},
we hence have
\[
\left|\mathcal{F}(f)\left(2\pi\frac{h_{\omega}}{T}\right)-\mathcal{F}(f)\left(2\pi\omega\right)\right|\le2\pi M\left|\frac{h_{\omega}}{T}-\omega\right|\le2\pi\frac{M}{T}\approx0.
\]
We hence proved that
\[
\exists Q\in\N\,\forall T\ge\diff{\rho}^{-Q}:\ c_{h_{\omega}}\approx\frac{1}{2T}\mathcal{F}(f)(2\pi\omega).
\]
Finally, note that since $T$ is an infinite number, if $h_{\omega}\in\Z$,
then necessarily $\omega$ must be infinitesimal; on the contrary,
if $\omega\ge r\in\R_{\ne0}$, then necessarily $h_{\omega}\in\hyperZ{\sigma}\setminus\Z$
is an infinite integer number.

Therefore, with the precise meaning given above, the heuristic relations
between Fourier coefficients and HFT holds also for GSF.

\section{\label{sec:The-Riemann-Lebesgue-lemma}The Riemann-Lebesgue lemma
in a non-linear setting}

The following result represents the Riemann-Lebesgue lemma in our
framework. It immediately highlights an important difference with
respect to the classical approach since it states that the HFT of
a very large class of compactly supported GSF is still compactly supported
(see also Thm.~\ref{thm:uncertainty} for a classical formulation
of the uncertainty inequality for GSF).
\begin{lem}
\label{lem:Rieman-Lebesgue}Let $H\fcmp\rti^{n}$ and $f\in\Dgsf\left(H\right)$
be a compactly supported GSF. Assume that
\begin{equation}
\exists C,b\in\rti_{>0}\,\forall x\in H\,\forall j\in\N:\ \left|\diff{^{j}}f(x)\right|\le C\cdot b^{j}.\label{eq:derPol}
\end{equation}
For all $N_{1},\ldots,N_{n}\in\N$ and $\omega\in\rti^{n}$, if $\omega_{1}^{N_{1}}\cdot\ldots\cdot\omega_{n}^{N_{n}}$
is invertible, then
\begin{equation}
\left|\mathcal{F}(f)(\omega)\right|\le\frac{1}{\left|\omega_{1}^{N_{1}}\cdot\ldots\cdot\omega_{n}^{N_{n}}\right|}\cdot\int_{H}\left|\partial_{1}^{N_{1}}\ldots\partial_{n}^{N_{n}}f(x)\right|\,\diff{x}.\label{eq:R-Lineq}
\end{equation}
Therefore
\begin{equation}
\lim_{\omega\to\infty}\left|\mathcal{F}(f)(\omega)\right|=0.\label{eq:R-Llim}
\end{equation}
Actually, \eqref{eq:R-Lineq} yields the stronger result:
\begin{equation}
\exists Q\in\N:\ \mathcal{F}(f)\in\Dgsf\left(\overline{B_{\diff{\rho}^{-Q}}(0)}\right).\label{eq:HFTCmptSupp}
\end{equation}
\end{lem}

\begin{proof}
Let us apply integration by parts Thm.~\ref{thm:intRules}.\ref{enu:intByParts}
at the $p$-th integral in \eqref{eq:Hpfinite_FT} (assuming that
$N_{p}>0$):
\begin{align*}
\intop_{-k}^{k}f\left(x\right)e^{-i\omega\cdot x}\,\diff{x_{p}} & =\left.-\frac{f\left(x\right)}{i\omega_{p}}e^{-i\omega\cdot x}\right|_{x_{p}=-k}^{x_{p}=k}+\frac{1}{i\omega_{p}}\intop_{-k}^{k}\partial_{p}f\left(x\right)e^{-i\omega\cdot x}\,\diff{x_{p}}=\\
 & =\frac{1}{i\omega_{p}}\intop_{-k}^{k}\partial_{p}f\left(x\right)e^{-i\omega\cdot x}\,\diff{x_{p}}.
\end{align*}
because Thm.~\ref{thm:DerivativeIsZero}.\ref{enu:derZeroBound}
yields $f(x)=0$ if $x_{p}=\pm k$. Applying the same idea with $N_{p}\in\N$
repeated integrations by parts for each integral in \eqref{eq:Hpfinite_FT},
and using Thm.~\ref{thm:DerivativeIsZero}.\ref{enu:derZeroBound},
we obtain
\[
\mathcal{F}(f)(\omega)=\frac{1}{\omega_{1}^{N_{1}}\cdot\ldots\cdot\omega_{n}^{N_{n}}i^{N_{1}+\ldots+N_{n}}}\int_{K}\partial_{1}^{N_{1}}\ldots\partial_{n}^{N_{n}}f(x)e^{-ix\cdot\omega}\,\diff{\rho}.
\]
Claims \eqref{eq:R-Lineq} and \eqref{eq:R-Llim} both follows from
Thm.~\ref{thm:muMeasurableAndIntegral}.\ref{enu:existsReprDefInt}
and from the closure of GSF with respect to differentiation, i.e.~Thm.~\ref{thm:FR-forGSF}.

To prove \eqref{eq:HFTCmptSupp}, we first recall \eqref{eq:closureBall},
so that $\overline{B_{\diff{\rho}^{-Q}}(0)}\fcmp\rti^{n}$. Let $C$,
$b\in\rti_{>0}$ from \eqref{eq:derPol} and $\lambda(H)\in\rti$,
where $\lambda$ is the Lebesgue measure. Therefore, $b\le\diff{\rho}^{-R}$
for some $R\in\N$, and we can set $Q:=R+1$. We want to prove the
claim using Thm.~\ref{thm:DerivativeIsZero}.\ref{enu:equivCmptSupp},
so that we take $\omega=(\omega_{1},\ldots,\omega_{n})\in\text{ext}\left(\overline{B_{\diff{\rho}^{-Q}}(0)}\right)$.
It cannot be $|\omega|\sbpt{<}\diff{\rho}^{-Q}$ because this would
yield $|\omega-a|\sbpt{=}0$ for some $a\in\overline{B_{\diff{\rho}^{-Q}}(0)}$;
thereby, $|\omega|\ge\diff{\rho}^{-Q}$ by Lem.~\ref{lem:trich1st}.
It always holds $\max_{l=1,\ldots,n}|\omega_{l}|\ge\frac{1}{n}|\omega|$,
i.e.~$\left[\max_{l=1,\ldots,n}|\omega_{l\eps}|\right]\ge\frac{1}{n}\left[|\omega_{\eps}|\right]$,
where $\omega_{l}=[\omega_{l\eps}]$ and $\omega_{\eps}:=\left|(\omega_{1\eps},\ldots,\omega_{n\eps})\right|$.
In general, we cannot say that $|\omega_{p}|=\max_{l=1,\ldots,n}|\omega_{l}|$
for some $p=1,\ldots,n$ because at most this equality holds only
for subpoints. In fact, set $L_{p}:=\left\{ \eps\in I\mid\max_{l=1,\ldots,n}|\omega_{l\eps}|=\left|\omega_{p\eps}\right|\right\} $
and let $P\subseteq\{1,\ldots,n\}$ be the non empty set of all the
indices $p=1,\ldots,n$ such that $L_{p}\subzero I$. We hence have
$|\omega_{p}|=_{L_{p}}\max_{l=1,\ldots,n}|\omega_{l}|\ge\frac{1}{n}|\omega|\ge\frac{1}{n}\diff{\rho}^{-Q}$
for all $p\in P$, and
\begin{equation}
\forall^{0}\eps\,\exists p\in P:\ \eps\in L_{p}.\label{eq:L_p}
\end{equation}
We apply assumption \eqref{eq:derPol} and inequality \eqref{eq:R-Lineq}
with an arbitrary $N_{p}=N\in\N$, $p\in P$, and with $N_{j}=0$
for all $j\ne p$ to get
\begin{align*}
\left|\mathcal{F}(f)(\omega)\right| & \le\frac{1}{|\omega_{p}|^{N}}\cdot\int_{H}\left|\partial_{p}^{N}f(x)\right|\,\diff{x}\le_{L_{p}}n^{N}\cdot\diff{\rho}^{NQ}Cb^{N}\lambda(H)\le\\
 & \le\diff{\rho}^{-1}\cdot\diff{\rho}^{N(Q-R)}C\lambda(H)=\diff{\rho}^{N-1}C\lambda(H).
\end{align*}
For $N\to+\infty$ (in the ring $\rti|_{L_{p}})$, we hence have that
$\mathcal{F}(f)(\omega)=_{L_{p}}0$. From \eqref{eq:L_p} we hence
finally get $\mathcal{F}(f)(\omega)=0$.
\end{proof}
\begin{rem}
\label{rem:R-L}~
\begin{enumerate}
\item Considering that $\delta(t)=b^{n}\psi(bt)$ and that $\psi$ is an
even function (Lem.~\ref{lem:strictDeltaNet}.\ref{enu:suppStrictDeltaNet}),
we have
\begin{equation}
\mathcal{F}(\delta)(\omega)=\int\delta(t)e^{-it\omega}\,\diff{t}=\int\delta(0-t)e^{-it\omega}\,\diff{t}=\left(\delta*e^{-i(-)\omega}\right)(0).\label{eq:deltaConvFT}
\end{equation}
We already know that if $b/|\omega|$ is a strong infinite number,
then the function $f_{\omega}(t)=e^{-it\omega}$ is bounded by a tame
polynomial. Thereby, using Thm.~\ref{thm:convIdentity}, we have
$\mathcal{F}(\delta)(\omega)=f_{\omega}(0)=1$; in particular, $\mathcal{F}(\delta)|_{\R}=1$.
\item On the other hand (taking for simplicity $\psi:=\mathcal{F}^{-1}(\beta)$,
where $\beta\in\Coo(\R)$ is supported e.g.~in $[-1,1]$ and identically
equals $1$ in a neighborhood of $0$, see Thm.~\ref{thm:embeddingD'}),
$\delta^{(j)}(t)=b^{j+1}\psi^{(j)}(bt)$ if $n=1$, and hence for
all $t\in\rti$, we have
\begin{align*}
\delta^{(j)}(t) & =b^{j+1}\psi^{(j)}(bt)=b^{j+1}\cdot\left[\frac{\diff{}^{j}}{\diff{t^{j}}}\left(\frac{1}{2\pi}\int\beta(x)e^{ib_{\eps}tx}\,\diff x\right)\right]=\\
 & =\frac{b^{j+1}}{2\pi}\left[\int(ib_{\eps}x)^{j}\beta(x)e^{ib_{\eps}tx}\ \diff x\right]\\
\left|\delta^{(j)}(t)\right| & \le\frac{b^{2j+1}}{2\pi}\int_{-1}^{1}|x|^{j}\beta(x)\,\diff x=:C(b^{2})^{j}.
\end{align*}
Thus, Dirac's delta satisfies condition \eqref{eq:derPol} and hence
\begin{equation}
\exists Q\in\N:\ \mathcal{F}(\delta)\in\Dgsf(\overline{B_{\diff{\rho}^{-Q}}(0)}).\label{eq:FTDeltaCmptSupp}
\end{equation}
 In the following, we will use the notation $\mathbb{1}:=\mathcal{F}(\delta)$.
\item The previous result also yields that $f*\delta=f$ cannot hold in
general since otherwise, we can argue as in \eqref{eq:deltaConvFT}
to prove that $\mathcal{F}(\delta)(\omega)=1$ for all $\omega\in\rti$,
in contradiction with \eqref{eq:FTDeltaCmptSupp}.
\end{enumerate}
\end{rem}

Inequality \eqref{eq:R-Lineq} can also be stated as a general impossibility
theorem (where we intuitively think $n=1$).
\begin{thm}
\label{thm:R-Limp}Let $(R,\le)$ be an ordered ring and $G$ be an
$R$-module. Assume that we have the following maps (for which we
use notations aiming to draw the interpretation where $G$ is a space
of GF)
\begin{align*}
(-)' & :G\ra G\\
\int & :G\ra R\\
(-)\cdot\exp_{\omega} & :G\ra G\quad\forall\omega\in R\\
|-| & :R\ra R.
\end{align*}
These maps satisfy the following integration by parts formula
\begin{equation}
\int f\cdot\exp_{\omega}=\frac{1}{\omega}\int f'\cdot\exp_{\omega}\label{eq:absIntParts}
\end{equation}
for all invertible $\omega\in R^{*}$, $f\in G$, and
\begin{equation}
|rs|=|r||s|\quad\forall r,s\in R\label{eq:absProd}
\end{equation}
\begin{equation}
\forall f\in G\,\exists C\in R\,\forall\omega\in R^{*}:\ \left|\int f\cdot\exp_{\omega}\right|\le C.\label{eq:boundInt}
\end{equation}
Then for all $f\in G$ and all $N\in\N_{>0}$ there exists $C=C(f,N)\in R$
such that
\begin{equation}
\forall\omega\in R^{*}:\left|\int f\cdot\exp_{\omega}\right|\le\frac{C}{|\omega|^{N}}.\label{eq:absRL}
\end{equation}
Therefore, if $\delta\in G$ satisfies $\frac{C(\delta,N)}{|\omega|^{N}}<1$
for some $\omega\in R$ and some $N\in\N$, then
\[
\left|\int\delta\cdot\exp_{\omega}\right|<1.
\]
\end{thm}

\begin{proof}
For $f\in G$, in the usual way we recursively define $f^{(p)}\in G$
using the map $(-)':G\ra G$. Taking formula \eqref{eq:absIntParts}
for $N\in\N_{>0}$ times we get $\int f\cdot\exp_{\omega}=\frac{1}{\omega^{N}}\int f^{(N)}\cdot\exp_{\omega}$.
Applying $|-|$ and using \eqref{eq:absProd} and \eqref{eq:boundInt}
we get the conclusion \eqref{eq:absRL}.
\end{proof}
\noindent Note that we can take $R=\left\{ i\cdot r\mid r\in\rccrho\right\} $
to apply this abstract result to the case of Lem\@.~\ref{lem:Rieman-Lebesgue}.
This result also underscore that in the case $G=\mathcal{D}'(\R)$,
$R=\R$ we cannot have an integration by parts formula such as \eqref{eq:absIntParts}.
Once more, it also underscores that, since \eqref{eq:absIntParts}
holds in our setting, we cannot have $f*\delta=f$ without limitations
because this would imply $\mathcal{F}(\delta)(\omega)=1$ for all
$\omega\in\rti$.
\begin{example}
\label{exa:exp}Let $f\left(x\right)=e^{x}$ for all $\left|x\right|\leq k$,
where $k:=-\log\left(\diff\rho\right)$. The hyperfinite Fourier transform
$\mathcal{F}_{k}$ of $f$ is
\begin{align*}
\mathcal{F}_{k}\left(f\right)\left(\omega\right) & =\frac{e^{k\left(1-i\omega\right)}-e^{-k\left(1-i\omega\right)}}{1-i\omega}=\frac{\diff{\rho^{\left(i\omega-1\right)}}-\diff{\rho^{\left(1-i\omega\right)}}}{1-i\omega}=\\
 & =\frac{1}{1-i\omega}\left(\frac{\diff{\rho}^{i\omega}}{\diff{\rho}}-\frac{\diff{\rho}}{\diff{\rho}^{i\omega}}\right)\quad\forall\omega\in\rti.
\end{align*}
Note that $1-i\omega$, $\omega\in\rti$, is always invertible with
the usual inverse $\frac{1+i\omega}{1+\omega^{2}}$, moreover, $\diff{\rho}^{i\omega}=e^{i\omega\log\diff{\rho}}$
and hence $|\diff{\rho}^{i\omega}|=1$. Therefore, $\mathcal{F}_{k}(f)(\omega)$
is always an infinite complex number for all finite numbers $\omega$.
If $\omega\ge\diff{\rho}^{-1-r}$, $r\in\R_{>0}$, then $\mathcal{F}_{k}(f)(\omega)$
is infinitesimal but not zero. Clearly, $f\notin\Dgsf(K)$.

considering Robinson-Colombeau generalized numbers, the Gaussian
is compactly supported:
\begin{lem}
\label{lem:example}Let $f\left(x\right)=e^{-\frac{\left|x\right|^{2}}{2}}$
for all $x\in\rti^{n}$. Then $f\in\Dgsf\left(\overline{B_{h}(0)}\right)$
for all strong infinite number $h\in\rti_{>0}$. Moreover, $\mathcal{F}\left(f\right)=\left(2\pi\right)^{\frac{n}{2}}f$.
\begin{proof}
The function $f$ satisfies the inequality $0\leq f\left(x\right)\leq\left|x\right|^{-q}$,
$\forall q\in\mathbb{N}$, for $\left|x\right|$ finite sufficiently
large. Therefore, for all strongly infinite $x$, we have $f\left(x\right)=0$
i.e., $f\in\Dgsf\left(\rcrho^{n}\right)$. We first prove the second
claim in dimension $n=1$; denoting by $\hat{\mathcal{F}}$ the classical
Fourier we have
\begin{align*}
\mathcal{F}(f)(\omega) & =\mathcal{F}_{\diff\rho^{-1}}(f)(\omega)=\int_{-\diff\rho^{-1}}^{\diff\rho^{-1}}e^{-x^{2}/2}e^{-i\omega x}\,\diff x=\\
 & =\left[\int_{-\rho_{\eps}^{-1}}^{\rho_{\eps}^{-1}}e^{-x^{2}/2}e^{-i\omega_{\eps}x}\,\diff x\right]\\
 & =\left[\int_{-\rho_{\eps}^{-1}}^{-\infty}e^{-x^{2}/2}e^{-i\omega_{\eps}x}\,\diff x+\hat{\mathcal{F}}\left(e^{-x^{2}/2}\right)(\omega_{\eps})+\int_{+\infty}^{\rho_{\eps}^{-1}}e^{-x^{2}/2}e^{-i\omega_{\eps}x}\,\diff x\right]\\
 & =\left[\sqrt{2\pi}e^{-\omega_{\eps}^{2}/2}-2\int_{\rho_{\eps}^{-1}}^{+\infty}e^{-x^{2}/2}e^{-i\omega_{\eps}x}\,\diff x\right]\\
 & =\sqrt{2\pi}f(\omega)-2\cdot\left[\int_{\rho_{\eps}^{-1}}^{+\infty}e^{-x^{2}/2}e^{-i\omega_{\eps}x}\,\diff x\right].
\end{align*}

\noindent Using L'Hôpital rule we can prove that $\lim_{y\to0^{+}}\frac{\intop_{1/y}^{\pm\infty}e^{-\frac{x^{2}}{2}}\,\diff{x}}{y^{q}}=0$
for all $q\in\N$, thereby $\left[\int_{\rho_{\eps}^{-1}}^{+\infty}e^{-x^{2}/2}e^{-i\omega_{\eps}x}\,\diff x\right]=0$
in $\rti$. In dimension $n>1$, we directly calculate using Fubini's
theorem: 
\begin{alignat*}{1}
\mathcal{F}\left(e^{-\frac{\left|x\right|^{2}}{2}}\right)\left(\omega\right) & =\prod_{j=1}^{n}\intop e^{-ix_{j}\cdot\omega_{j}}e^{-\frac{x_{j}^{2}}{2}}\,\diff{x_{j}}\\
 & =\prod_{j=1}^{n}\mathcal{F}\left(e^{-\frac{x_{j}^{2}}{2}}\right)\left(\omega_{j}\right)=\prod_{j=1}^{n}\left(2\pi\right)^{\frac{1}{2}}e^{-\frac{\omega_{j}^{2}}{2}}=\left(2\pi\right)^{\frac{n}{2}}e^{-\frac{\left|\omega\right|^{2}}{2}}.
\end{alignat*}
\end{proof}
\end{lem}

\end{example}

\section{Elementary properties of the hyperfinite Fourier transform\label{sec:Elementary-properties}}

In this section, we list and prove elementary properties of the HFT.
\begin{thm}
\label{thm:thmProperties}(see Sec.~\ref{subsec:Embedding} for the
notations $\odot$ and $\oplus$) Let $f\in\gsf\left(K\right)$ and
$g:\rti^{n}\ra\rccrho$, then
\begin{enumerate}
\item \label{enu:prop1}$\mathcal{F}_{k}\left(f+g\right)=\mathcal{F}_{k}\left(f\right)+\mathcal{F}_{k}\left(g\right)$
if $g\in\gsf(K)$.
\item \label{enu:prop2}$\mathcal{F}_{k}\left(bf\right)=b\mathcal{F}_{k}\left(f\right)$
for all $b\in\rccrho$.
\item \label{enu:prop3}$\mathcal{F}_{k}\left(\overline{f}\right)=\overline{-1\diamond\mathcal{F}_{k}(f)}$,
where $-1\diamond f$ is the \emph{reflection} of $f$, i.e.~$\left(-1\diamond f\right)\left(x\right):=f\left(-x\right)$.
\item \label{enu:prop4}$\mathcal{F}_{k}\left(-1\diamond f\right)=-1\diamond\mathcal{F}_{k}(f)$
\item \label{enu:prop5}$\mathcal{F}_{k}\left(t\diamond g\right)=t\odot\mathcal{F}_{tk}\left(g\right)$
for all $t\in\rti_{>0}$ such that $tk$ is still infinite and $g|_{K}\in\gsf(K)$,
$g|_{tK}\in\gsf(tK)$. Here, $t\diamond g$ is the \emph{dilation
}of $f,$ i.e.~$\left(t\diamond g\right)\left(x\right):=g\left(tx\right)$.
\item \label{enu:prop6}Let $k>h>0$ be infinite numbers, $s\in[-(k-h),k-h]^{n}$,
$f\in\Dgsf([-h,h]^{n})$. Then
\[
\mathcal{F}_{k}\left(s\oplus f\right)=e^{-is\cdot\left(-\right)}\mathcal{F}_{k}\left(f\right)=e^{-is\cdot\left(-\right)}\mathcal{F}_{h}\left(f\right)=e^{-is\cdot\left(-\right)}\mathcal{F}\left(f\right).
\]
In particular, if $h\ge\diff{\rho}^{-p}$, $k\ge\diff{\rho}^{-q}$,
$p$, $q\in\R_{>0}$, $q>p$, and $s\in\csp{\R^{n}}$, then $s\in[-(k-h),k-h]^{n}$.
In particular, $\R^{n}\subseteq[-(k-h),k-h]^{n}$.
\item \label{enu:prop7}$\mathcal{F}_{k}\left(e^{is\cdot\left(-\right)}f\right)=s\oplus\mathcal{F}_{k}\left(f\right)$
for all $s\in\rti^{n}$.
\item \label{enu:prop8}Let $\omega\in\rti^{n}$ and $\alpha\in\N^{n}\setminus\{0\}$.
For $p=1,\ldots,|\alpha|$, define $\beta_{p}=\left(\beta_{p,q}\right)_{q=1,\ldots,n}\in\N^{n}$
with
\begin{align*}
\beta_{0} & :=\alpha\\
\beta_{p+1} & :=(0,\ptind^{j_{p}-1},0,\beta_{p,j_{p}}-1,\beta_{p,j_{p}+1},\ldots,\beta_{p,n})\text{ if }j_{p}:=\min\left\{ q\mid\beta_{p,q}>0\right\} .
\end{align*}
Finally, for all $\bar{f}\in\gsf(K)$ and $j=1,\ldots,n$, set
\begin{align*}
\Delta_{1k}\bar{f}(\omega):= & \left[\bar{f}(x)e^{-ix\cdot\omega}\right]_{x_{1}=-k}^{x_{1}=k}\\
\Delta_{jk}\bar{f}(\omega):= & \intop_{-k}^{k}\,\diff{x_{1}}\ldots\intop_{-k}^{k}\,\diff{x_{j-1}}\intop_{-k}^{k}\,\diff{x_{j+1}}\ldots\intop_{-k}^{k}\left[\bar{f}(x)e^{-ix\cdot\omega}\right]_{x_{j}=-k}^{x_{j}=k}\,\diff{x_{n}}.
\end{align*}
Then, we have
\begin{align}
\mathcal{F}_{k}\left(\partial_{j}f\right) & =i\omega_{j}\mathcal{F}_{k}\left(f\right)+\Delta_{jk}f\quad\forall j=1,\ldots,n\label{eq:DerRule1}\\
\mathcal{F}_{k}\left(\partial^{\alpha}f\right) & =\left(i\omega\right)^{\alpha}\mathcal{F}_{k}\left(f\right)+\sum_{p=0}^{|\alpha|-1}(i\omega)^{\alpha-\beta_{p}}\Delta_{j_{p}k}(\partial^{\beta_{p+1}}f).\label{eq:DerRule}
\end{align}
In particular, if
\begin{equation}
f\left(x_{1},\ldots,x_{j-1},k,x_{j+1}\right)=f\left(x_{1},\ldots,x_{j-1},-k,x_{j+1}\right)=0\quad\forall x\in K,\label{eq:Hpk-k0}
\end{equation}
then 
\[
\mathcal{F}_{k}\left(\partial_{j}f\right)=i\omega_{j}\mathcal{F}_{k}\left(f\right).
\]
\item \label{enu:prop9}$\frac{\partial}{\partial\omega_{j}}\mathcal{F}_{k}\left(f\right)=-i\mathcal{F}_{k}\left(x_{j}f\right)$
for all $j=1,\ldots,n$.
\item \label{enu:prop10}If $f\in\Dgsf(K)$ or $g\in\Dgsf(K)$, then $\mathcal{F}_{k}\left(f\ast g\right)=\mathcal{F}_{k}\left(f\right)\mathcal{F}_{k}\left(g\right)$.
Therefore, if $f\in\Dgsf(\rti^{n})$ and $g\in\Dgsf(\rti^{n})$, then
$\mathcal{F}\left(f\ast g\right)=\mathcal{F}\left(f\right)\mathcal{F}\left(g\right)$.
\item \label{enu:prop11}$\mathcal{F}_{k}\left(s\odot g\right)=s\diamond\mathcal{F}_{\frac{k}{s}}\left(g\right)$
for all invertible $s\in\rti_{>0}$ such that $\frac{k}{s}$ is infinite,
$g|_{K}\in\gsf(K)$ and $g|_{K/s}\in\gsf(K/s)$.
\end{enumerate}
\end{thm}

\begin{proof}
Properties \ref{enu:prop1}-\ref{enu:prop5} can be proved like in
the case of rapidly decreasing smooth functions. For \ref{enu:prop6},
we have
\begin{align*}
\mathcal{F}_{k}\left(s\oplus f\right)\left(\omega\right) & =\mathcal{F}_{k}\left(f\left(x-s\right)\right)\left(\omega\right)=\intop_{K}f\left(x-s\right)e^{-ix\cdot\omega}\,\diff{x}=\\
 & =\intop_{-k}^{k}\,\diff{x_{1}}\ldots\intop_{-k}^{k}f\left(x-s\right)e^{-ix\cdotp\omega}\,\diff{x_{n}}.
\end{align*}
 Considering the change of variable $x-s=u$ we have
\[
\mathcal{F}_{k}\left(s\oplus f\right)\left(\omega\right)=e^{-is\cdot\omega}\intop_{-k-s_{1}}^{k-s_{1}}\,\diff{u_{1}}\ldots\intop_{-k-s_{n}}^{k-s_{n}}f\left(u\right)e^{-iu\cdotp\omega}\,\diff{u_{n}}.
\]
Finally, considering that $k>h$ and $s\in[-k+h,k-h]^{n}$ we have
$k-s_{i}\ge h$, $-h\ge-k-s_{i}$ and $k+s_{i}\ge h$ for all $i=1,\ldots,n$,
so that
\begin{align*}
\intop_{-k-s_{1}}^{k-s_{1}}\,\diff{u_{1}}\ldots\intop_{-k-s_{n}}^{k-s_{n}}f\left(u\right)e^{-iu\cdotp\omega}\,\diff{u_{n}} & =\intop_{-h}^{h}\,\diff{u_{1}}\ldots\intop_{-h}^{h}f\left(u\right)e^{-iu\cdotp\omega}\,\diff{u_{n}}=\\
 & =\intop_{-k}^{k}\,\diff{u_{1}}\ldots\intop_{-k}^{k}f\left(u\right)e^{-iu\cdotp\omega}\,\diff{u_{n}}
\end{align*}
from Def.~\ref{def:intCmpSupp} since $f\in\Dgsf([-h,h]^{n})$.

\ref{enu:prop7} is immediate from the Def. \ref{def:HyperfiniteFT}.

To prove \ref{enu:prop8}, using integration by parts formula, we
have
\begin{align*}
\mathcal{F}_{k}\left(\partial_{j}f\right)\left(\omega\right) & =\intop_{K}\partial_{j}f\left(x\right)e^{-ix\cdot\omega}\,\diff{x}=\intop_{-k}^{k}\,\diff{x_{1}}\ldots\intop_{-k}^{k}\partial_{j}f\left(x\right)e^{-ix\cdotp\omega}\,\diff{x_{n}}=\\
 & =-\intop_{-k}^{k}\,\diff{x_{1}}\ldots\intop_{-k}^{k}f\left(x\right)\left(-i\omega_{j}\right)e^{-ix\cdotp\omega}\,\diff{x_{n}}+\\
 & \phantom{=}+\intop_{-k}^{k}\,\diff{x_{1}}\ldots\intop_{-k}^{k}\,\diff{x_{j-1}}\intop_{-k}^{k}\,\diff{x_{j+1}}\ldots\intop_{-k}^{k}\left[f(x)e^{-ix\cdot\omega}\right]_{x_{j}=-k}^{x_{j}=k}\,\diff{x_{n}}=\\
 & =i\omega_{j}\mathcal{F}_{k}\left(f\right)\left(\omega\right)+\Delta_{jk}f(\omega).
\end{align*}
Therefore, by applying this formula with $\partial_{p}f$ instead
of $f$, we obtain
\[
\mathcal{F}_{k}\left(\partial_{j}\partial_{p}f\right)(\omega)=-\omega_{j}\omega_{p}\mathcal{F}_{k}(f)(\omega)+i\omega_{j}\Delta_{pk}\left(f\right)(\omega)+\Delta_{jk}\left(\partial_{p}f\right)(\omega).
\]
Proceeding similarly by induction on $|\alpha|$, we can prove the
general claim.

To prove \ref{enu:prop9}, we use Thm. \ref{thm:intRules}.\ref{enu:derUnderInt},
i.e.~derivation under the integral sign:
\begin{align*}
\frac{\partial}{\partial\omega_{j}}\mathcal{F}_{k}\left(f\right)\left(\omega\right) & =\frac{\partial}{\partial\omega_{j}}\left(\intop_{-k}^{k}\,\diff{x_{1}}\ldots\intop_{-k}^{k}f\left(x\right)e^{-ix\cdotp\omega}\,\diff{x_{n}}\right)=\\
 & =\intop_{-k}^{k}\,\diff{x_{1}}\ldots\intop_{-k}^{k}\frac{\partial}{\partial\omega_{j}}\left(f\left(x\right)e^{-ix\cdotp\omega}\right)\,\diff{x_{n}}=\\
 & =\intop_{-k}^{k}\,\diff{x_{1}}\ldots\intop_{-k}^{k}-ix_{j}f\left(x\right)e^{-ix\cdotp\omega}\,\diff{x_{n}}=\\
 & =-i\mathcal{F}_{k}\left(x_{j}f\right)\left(\omega\right).
\end{align*}

\ref{enu:prop10}: 
\begin{align*}
\mathcal{F}_{k}\left(\left(f\ast g\right)\right)\left(\omega\right) & =\intop_{K}e^{-ix\omega}\left(f\ast g\right)\left(x\right)\,\diff{x}=\\
 & =\intop_{K}e^{-ix\omega}\intop_{K}f\left(y\right)g\left(x-y\right)\,\diff{y}\,\diff{x}.
\end{align*}

Considering the change of variable $x-y=t$ and using Fubini's theorem,
we have
\begin{align*}
\intop_{K}e^{-i\left(t+y\right)\omega}\intop_{K}f\left(y\right)g\left(t\right)\,\diff{y}\,\diff{t} & =\intop_{K}e^{-iy\omega}f\left(y\right)\,\diff{y}\intop_{K}e^{-it\omega}g\left(t\right)\,\diff{t}=\\
 & =\mathcal{F}_{k}\left(f\right)\left(\omega\right)\mathcal{F}_{k}\left(g\right)\left(\omega\right).
\end{align*}
 Finally, we prove \ref{enu:prop11}:
\[
\mathcal{F}_{k}\left(s\odot g\right)\left(\omega\right)=\mathcal{F}_{k}\left(\frac{1}{s^{n}}g\left(\frac{x}{s}\right)\right)\left(\omega\right)=\intop_{K}e^{-ix\cdot\omega}g\left(\frac{x}{s}\right)\,\frac{\diff{x}}{s^{n}}.
\]

\noindent Considering the change of variable $\frac{x}{s}=y$ we have
\begin{align*}
\intop_{K}e^{-ix\cdot\omega}g\left(\frac{x}{s}\right)\,\frac{\diff{x}}{s^{n}} & =\intop_{-k/s}^{k/s}\,\diff{y_{1}}\ldots\intop_{-k/s}^{k/s}g\left(y\right)e^{-isy\cdotp\omega}\,\diff{y_{n}}=\\
 & =\intop_{K/s}g\left(y\right)e^{-iy\cdotp s\omega}\,\diff{y}=\mathcal{F}_{k/s}\left(g\right)\left(s\omega\right)=\\
 & =\left[s\diamond\mathcal{F}_{k/s}\left(g\right)\right](\omega).
\end{align*}
\end{proof}
We will see in Sec.~\ref{sec:Examples-and-applications} that the
additional term in \eqref{eq:DerRule} plays an important role in
finding \emph{non tempered} solutions of differential equations (like
the exponentials of the trivial ODE $y'=y$). We also note that condition
\eqref{eq:Hpk-k0} is clearly weaker than asking $f$ compactly supported.
For example, setting
\[
l_{j}(x):=\frac{1}{2k}\left[f\left(x\right)|_{x_{j}=k}-f\left(x\right)|_{x_{j}=-k}\right]\cdot(x_{j}+k)+f\left(x\right)|_{x_{j}=-k},
\]
then $\bar{f}:=f-l_{j}$ satisfies \eqref{eq:Hpk-k0}.

\section{\label{sec:The-inverse-hyperfinite}The inverse hyperfinite Fourier
transform}

We naturally define the inverse HFT as follows:
\begin{defn}
Let $f\in\gsf\left(K\right)$, the \emph{inverse} HFT is 
\begin{equation}
\mathcal{F}_{k}^{-1}\left(f\right)(x):=\frac{1}{\left(2\pi\right)^{n}}\intop_{K}f\left(\omega\right)e^{ix\cdot\omega}\,\diff{\omega}\label{eq:InverseFT}
\end{equation}

\noindent for all $x\in\rcrho$. As we proved in Thm.~\ref{thm:base},
we have $\mathcal{F}_{k}^{-1}:\gsf\left(K\right)\longrightarrow\gsf\left(\rcrho^{n}\right)$.
We immediately note that the notation of the inverse function $\mathcal{F}_{k}^{-1}$
is an abuse of language because the codomain of $\mathcal{F}_{k}$
is larger than the domain of $\mathcal{F}_{k}^{-1}$ (and vice versa).
When dealing with inversion properties, it is hence better to think
at
\begin{align*}
\mathcal{F}_{k}|_{K} & :=(-)|_{K}\circ\mathcal{F}_{k}:\gsf\left(K\right)\longrightarrow\gsf\left(K\right)\\
\mathcal{F}_{k}^{-1}|_{K} & :=(-)|_{K}\circ\mathcal{F}_{k}^{-1}:\gsf\left(K\right)\longrightarrow\gsf\left(K\right).
\end{align*}
We will see in Sec.~\ref{sec:Examples-and-applications} that lacking
this precision can easily lead to inconsistencies.

Note that
\begin{equation}
\left(2\pi\right)^{n}\mathcal{F}_{k}^{-1}(f)=\mathcal{F}_{k}\left(-1\diamond f\right)=-1\diamond\mathcal{F}_{k}(f),\label{eq:reflection}
\end{equation}
where $-1\diamond$ denotes the reflection $\left(-1\diamond g\right)(x):=g(-x)$.
\end{defn}

\subsection{The Fourier inversion theorem}

Our main goal is clearly to investigate the relationship between HFT
and its inverse HFT, i.e.~to prove the Fourier inversion theorem
for the HFT. Three important results used in the classical proof of
the Fourier inversion theorem are: the application of approximate
identities for convolution defined by Gaussian like functions (see
\cite[Lem.~4.3]{LeLuGi17} for a similar result), Lebesgue dominated
converge theorem (we can replace it with Thm\@.~\ref{thm:contResult}),
and the translation property of FT. In our setting, the last property
corresponds to Thm.~\ref{thm:thmProperties}.\ref{enu:prop6}, which
works only for compactly supported GSF. A first idea could hence to
avoid proving the inversion theorem firstly at the origin and then
employing the translation property, but to prove it directly at an
arbitrary interior point $y\in\mathring{K}$ using approximate identities
obtained by mollification of a Gaussian function. Unfortunately, this
idea does not work: in fact, if $g(z):=(2\pi)^{-n/2}e^{-\frac{z^{2}}{2}}$,
then our approximate identity would be the mollification $G_{p}:=\frac{1}{p}\odot g$,
where we think $p\in\hypNr$, $p\to+\infty$. We would also need a
function $g_{p}$ such that $\mathcal{F}\left[g_{p}(-,y)\right](x)=(2\pi)^{n/2}G_{p}(y-x)$,
i.e.~$g_{p}(-,y):=e^{iy\cdot\left(-\right)}\cdot\left(t_{p}\diamond g\right)$.
The first problem is that $\text{supp}\left(g_{p}(-,y)\right)\subseteq\overline{B_{p\diff\rho^{-1}}(0)}\uparrow\rti$
as $p\to+\infty$. In an integral of the type $\int_{K}\mathcal{F}_{k}(f)(\omega)g_{p}(\omega,y)\,\diff\omega$
we would therefore need $k$ \emph{non} $\rho$-moderate (see below,
Def.~\ref{def:RCud}) to contain all the support of $g_{p}(-,y)$.
On the other hand, we would also need $\hyperlimarg{\rho}{\rho}{p}g_{p}(\omega,y)=\frac{1}{(2\pi)^{n}}e^{iy\cdot\omega}$
, and it is not hard to prove that $\left|g_{p}(\omega,y)-\frac{1}{(2\pi)^{n}}e^{iy\cdot\omega}\right|\le\diff\rho^{q}$
if $p\ge C_{2}k\diff\rho^{-q/2}$ for some $C_{2}\in\R_{>0}$, and
this implies that $k$ must be $\rho$-moderate.

The idea for a different proof starts from the following calculations
(for $n=1$):
\begin{align*}
\mathcal{F}_{k}(f)(\omega) & =\int_{K}f(x)e^{-ix\omega}\,\diff x=\left[\int_{K_{\eps}}f_{\eps}(x)e^{-ix\omega_{\eps}}\,\diff x\right]\\
 & =\left[\int_{\R}\chi_{K_{\eps}}(x)f_{\eps}(x)e^{-ix\omega_{\eps}}\,\diff x\right]=\left[\hat{\mathcal{F}}\left(\chi_{K_{\eps}}f_{\eps}\right)(\omega_{\eps})\right],
\end{align*}
where $\chi_{K_{\eps}}$ is the characteristic function of $K_{\eps}:=[-k_{\eps},k_{\eps}]$,
and $\hat{\mathcal{F}}$ is the classical Fourier transform. Thereby,
if we take \emph{another} positive infinite number $h=[h_{\eps}]\in\rti$
and set $H:=[-h,h]$, $H_{\eps}:=[-h_{\eps},h_{\eps}]$, then
\begin{align*}
\mathcal{F}_{h}^{-1}\left(\mathcal{F}_{k}(f)\right)(y) & =\frac{1}{2\pi}\int_{H}e^{iy\omega}\left[\hat{\mathcal{F}}\left(\chi_{K_{\eps}}f_{\eps}\right)(\omega)\right]\,\diff\omega\\
 & =\frac{1}{2\pi}\left[\int_{\R}e^{iy_{\eps}\omega}\chi_{H_{\eps}}(\omega)\hat{\mathcal{F}}\left(\chi_{K_{\eps}}f_{\eps}\right)(\omega)\right]\\
 & =\left[\hat{\mathcal{F}}^{-1}\left(\chi_{H_{\eps}}\cdot\hat{\mathcal{F}}\left(\chi_{K_{\eps}}f_{\eps}\right)\right)(y_{\eps})\right]\\
 & =\left[\left(\hat{\mathcal{F}}^{-1}(\chi_{H_{\eps}})*\chi_{K_{\eps}}f_{\eps}\right)(y_{\eps})\right].
\end{align*}
We now compute
\begin{align*}
\hat{\mathcal{F}}^{-1}(\chi_{H_{\eps}})(z) & =\frac{1}{2\pi}\int_{\R}e^{iz\omega}\chi_{H_{\eps}}(\omega)\,\diff\omega=\frac{1}{2\pi}\int_{-h_{\eps}}^{h_{\eps}}e^{iz\omega}\,\diff\omega=\frac{1}{\pi}h_{\eps}S(h_{\eps}z),
\end{align*}
where $S(x)=\frac{1}{2}\int_{-1}^{1}\cos(xt)\,\diff t$ is the smooth
extension of $\frac{\sin(x)}{x}$ at $x=0$. Therefore, we can write
\begin{equation}
\mathcal{F}_{h}^{-1}\left(\mathcal{F}_{k}(f)\right)(y)=\int_{K}\frac{h}{\pi}S(h(y-x))f(x)\,\diff x.\label{eq:beforeSifting}
\end{equation}
For $n\ge1$, we similarly have
\begin{equation}
\hat{\mathcal{F}}^{-1}(\chi_{H_{\eps}})(z)=\frac{1}{\pi^{n}}h_{1\eps}S(z_{1}h_{1\eps})\cdot\ptind^{n}\cdot h_{n\eps}S(z_{n}h_{n\eps})=:\delta_{h_{\eps}}^{n}(z),\label{eq:delta1}
\end{equation}
and
\begin{equation}
\mathcal{F}_{h}^{-1}\left(\mathcal{F}_{k}(f)\right)(y)=\int_{K}\delta_{h}^{n}(y-x)\cdot f(x)\,\diff x.\label{eq:FI-convDirichletDelta}
\end{equation}
We call the GSF ($h$ is an infinite number) $\delta_{h}^{n}$ \emph{Dirichlet
delta function} (recall that the delta sequence $\left(\delta_{n}^{n}\right)_{h\in\N}$
converges to $\delta$ in $\mathcal{D}'$, see e.g.~\cite{BlBr15}).
In fact, these calculations lead us to consider the so-called Dirichlet
sifting theorem
\begin{equation}
\lim_{h\to+\infty}\int_{-\infty}^{+\infty}\delta_{h}^{1}(x)f(x)=f(0),\label{eq:DirichletSifting}
\end{equation}
which holds for $f\in\mathcal{C}^{1}(\R)$ such that $f'$ is bounded
(see e.g.~\cite{Kan98,BlBr15}). Formula \eqref{eq:DirichletSifting}
also justifies why we considered another infinite number $h$; moreover,
in the following proof, we will see that the use of the functionally
compact set $K$ instead of $\int_{-\infty}^{+\infty}$ allows us
to avoid any limitation on the GSF $f$.

\noindent We first need the following results:
\begin{lem}
\label{lem:DirichletMass1}For all sharply interior point $y\in\mathring{K}$,
we have
\[
\lim_{h\to+\infty}\int_{K}\delta_{h}^{n}(y-x)\,\diff x=1.
\]

\noindent Here, the limit is in the sharp topology, i.e.
\[
\forall q\in\N\,\exists\bar{h}\in\rti\,\forall h\in\rti_{\ge\bar{h}}:\ \left|\int_{K}\delta_{h}^{n}(y-x)\,\diff x-1\right|\le\diff\rho^{q}.
\]
\end{lem}

\begin{proof}
We actually prove the case $n=1$, since $n>1$ is similar. From $y=[y_{\eps}]\in\mathring{K}=(-k,k)$,
we can take the representative $(y_{\eps})$ so that $-k_{\eps}<y_{\eps}<k_{\eps}$
for all $\eps$. We have $\int_{K}\delta_{h}^{1}(y-x)\,\diff x=\left[\frac{h_{\eps}}{\pi}\int_{-k_{\eps}}^{k_{\eps}}S\left(h_{\eps}(y_{\eps}-x)\right)\,\diff x\right]$.
With the change of variables $x'=h_{\eps}(y_{\eps}-x)$, we get
\begin{align*}
\frac{h_{\eps}}{\pi}\int_{-k_{\eps}}^{k_{\eps}}S\left(h_{\eps}(y_{\eps}-x)\right)\,\diff x & =\frac{1}{\pi}\int_{h_{\eps}y_{\eps}-h_{\eps}k_{\eps}}^{h_{\eps}y_{\eps}+h_{\eps}k_{\eps}}S\\
 & =\frac{1}{\pi}\left(\int_{h_{\eps}y_{\eps}-h_{\eps}k_{\eps}}^{-\infty}S+\int_{-\infty}^{+\infty}S+\int_{+\infty}^{h_{\eps}y_{\eps}+h_{\eps}k_{\eps}}S\right)\\
 & =\frac{1}{\pi}\left(\int_{h_{\eps}y_{\eps}-h_{\eps}k_{\eps}}^{-\infty}S+\pi+\int_{+\infty}^{h_{\eps}y_{\eps}+h_{\eps}k_{\eps}}S\right).
\end{align*}
Note that $h_{\eps}y_{\eps}+h_{\eps}k_{\eps}\ge0$ and $h_{\eps}y_{\eps}-h_{\eps}k_{\eps}\le0$
because $-k_{\eps}\le y_{\eps}\le k_{\eps}$. In general, if $0<a\le b$
or $a\le b<0$, we have
\begin{align*}
\left|\int_{a}^{b}S\right| & =\left|\int_{a}^{b}\frac{\sin x}{x}\,\diff x\right|=\left|-\frac{\cos y}{y}\right|_{a}^{b}+\left.\int_{a}^{b}\frac{\cos y}{y^{2}}\,\diff y\right|\\
 & \le\frac{1}{|b|}+\frac{1}{|a|}+\int_{a}^{b}\frac{\diff y}{y^{2}}=\frac{2}{|a|}+\frac{2}{|b|}.
\end{align*}
In our cases, this yields
\begin{align*}
\left|\int_{h_{\eps}y_{\eps}+h_{\eps}k_{\eps}}^{+\infty}S\right| & \le\frac{2}{\left|h_{\eps}y_{\eps}+h_{\eps}k_{\eps}\right|}\\
\left|\int_{-\infty}^{h_{\eps}y_{\eps}-h_{\eps}k_{\eps}}S\right| & \le\frac{2}{\left|h_{\eps}y_{\eps}-h_{\eps}k_{\eps}\right|}
\end{align*}
(recall that $-k_{\eps}\le y_{\eps}\le k_{\eps}$). Therefore
\[
\left|\int_{K}\delta_{h}^{1}(y-x)\,\diff x\right|\le1+\frac{2}{\left|hy+hk\right|}+\frac{2}{\left|hy-hk\right|}\to1
\]
as $h\to+\infty$ because $-k<y<k$.
\end{proof}
\noindent Now, we have to deal with estimations of the convolution
\eqref{eq:FI-convDirichletDelta} on the ``tails'', i.e.~arbitrarily
near $y$:
\begin{lem}
\label{lem:DirichletTails}Let $K\subseteq X\subseteq\rti^{n}$ and
$f\in\gsf(X).$Then for all $\delta\in\rti_{>0}$ such that $B_{\delta}(y)\subseteq K$,
we have:
\begin{enumerate}
\item \label{enu:DirichletTails}$\lim_{h\to+\infty}\int_{-k}^{y-\delta}\delta_{h}^{n}(y-x)\cdot f(x)\,\diff x=0=\lim_{h\to+\infty}\int_{y+\delta}^{k}\delta_{h}^{n}(y-x)\cdot f(x)\,\diff x$.
\item \label{enu:DirchletConvAround_y}$\lim_{h\to+\infty}\left(\mathcal{F}_{h}^{-1}\left(\mathcal{F}_{k}(f)\right)(y)-\int_{y-\delta}^{y+\delta}\delta_{h}^{n}(y-x)\cdot f(x)\,\diff x\right)=0$.
\end{enumerate}
As above, the limits are in the sharp topology.

\end{lem}

\begin{proof}
For $y<a\le b\le k$ or $-k\le a\le b<y$, we first consider
\begin{align}
\int_{a}^{b}\delta_{h}^{1}(y & -x)f(x)\,\diff x=\int_{a}^{b}\frac{1}{\pi}\frac{\sin(h(y-x))}{y-x}f(x)\,\diff x=\nonumber \\
 & =-\int_{hy-ha}^{hy-hb}\frac{1}{\pi}\frac{\sin z}{z}f\left(y-\frac{z}{h}\right)\,\diff z=\nonumber \\
 & =\frac{1}{\pi}\left(\left[-\frac{\cos z}{z}f\left(y-\frac{z}{h}\right)\right]_{hy-hb}^{hy-ha}+\int_{hy-hb}^{hy-ha}\cos(z)\frac{\diff{}}{\diff z}\left[\frac{f\left(y-\frac{z}{h}\right)}{z}\right]\,\diff z\right).\label{eq:int_a^bdelta}
\end{align}
The first summand in \eqref{eq:int_a^bdelta} yields
\begin{align}
\left|\left[-\frac{\cos z}{z}f\left(y-\frac{z}{h}\right)\right]_{hy-hb}^{hy-ha}\right| & =\left|-\cos(hy-ha)\frac{f(a)}{hy-ha}+\cos(hy-hb)\frac{f(b)}{hy-hb}\right|\le\nonumber \\
 & \le\frac{|f(a)|}{h}\cdot\frac{1}{|y-a|}+\frac{|f(b)|}{h}\cdot\frac{1}{|y-b|}\label{eq:0}
\end{align}
The second summand in \eqref{eq:int_a^bdelta} yields
\begin{align}
\int_{hy-hb}^{hy-ha}\cos(z)\frac{\diff{}}{\diff z}\left[\frac{f\left(y-\frac{z}{h}\right)}{z}\right]\,\diff z & =-\int_{hy-hb}^{hy-ha}\cos(z)\frac{f'\left(y-\frac{z}{h}\right)}{hz}\ \diff z\label{eq:1}\\
 & \phantom{=}-\int_{hy-hb}^{hy-ha}\cos(z)\frac{f\left(y-\frac{z}{h}\right)}{z^{2}}\ \diff z.\label{eq:2}
\end{align}
If $hy-hb\le z\le hy-ha$, then $-k\le a\le y-\frac{z}{h}\le b\le k$,
so that $\left|f'\left(y-\frac{z}{h}\right)\right|\le\max_{x\in K}\left|f'(x)\right|=:M_{1}\in\rti$
(note that this step would not be so trivial if we had to let $k\to+\infty$).
Thereby, from the mean value theorem applied to \eqref{eq:1}
\[
\left|\int_{hy-hb}^{hy-ha}\cos(z)\frac{f'\left(y-\frac{z}{h}\right)}{hz}\ \diff z\right|\le(b-a)\cdot\frac{M_{1}}{\left|\zeta\right|}
\]
for some $\zeta\in[hy-hb,hy-ha]$ (note that $a\le b$ implies $hy-ha\ge hy-hb$).
In the first case we assumed, i.e.~$y<a\le b\le k$, we have $hy-ha<0$,
so that $|\zeta|=-\zeta\ge ha-hy$, and hence $(b-a)\cdot\frac{M_{1}}{\left|\zeta\right|}\le(b-a)\cdot\frac{M_{1}}{ha-hy}$.
In the second case $-k\le a\le b<y$, we have $|\zeta|=\zeta\ge hy-hb$,
and hence $(b-a)\cdot\frac{M_{1}}{\left|\zeta\right|}\le(b-a)\cdot\frac{M_{1}}{hy-hb}$.
The last term in \eqref{eq:2} can be estimated as
\[
\left|\int_{hy-hb}^{hy-ha}\cos(z)\frac{f\left(y-\frac{z}{h}\right)}{z^{2}}\ \diff z\right|\le M_{0}\int_{hy-hb}^{hy-ha}\frac{\diff z}{z^{2}}=M_{0}\left(\frac{1}{hy-hb}-\frac{1}{hy-ha}\right),
\]
where $M_{0}:=\max_{x\in K}\left|f(x)\right|\in\rti$. Applying these
estimates with $a=-k$ and $b=y-\delta$ (so that the second case
holds), we get
\begin{align*}
\left|\int_{-k}^{y-\delta}\delta_{h}^{n}(y-x)\cdot f(x)\,\diff x\right| & \le\frac{|f(-k)|}{h}\cdot\frac{1}{|y+k|}+\frac{|f(y-\delta)|}{h}\cdot\frac{1}{\delta}+\frac{(y-\delta+k)M_{1}}{h\delta}+\\
 & \phantom{\le}+\frac{M_{0}}{h}\left(\frac{1}{y+k}-\frac{1}{\delta}\right)\le\\
 & \le2\frac{M_{0}}{h}\frac{1}{|y+k|}+\frac{(y-\delta+k)M_{1}}{h\delta}.
\end{align*}
For $h\to+\infty$, this proves the first part of \ref{enu:DirichletTails}.
Once again, note that if $h=k\to+\infty$, in general the term $\frac{|f(-k)|}{k}\not\to0$;
it is hence important that $k$ is fixed and only $h\to+\infty$.
Similarly, we can estimate the other integral in \ref{enu:DirichletTails}
for $a=y+\delta$ and $b=k$ (the first case holds) obtaining
\[
\left|\int_{y+\delta}^{k}\delta_{h}^{n}(y-x)\cdot f(x)\,\diff x\right|\le2\frac{M_{0}}{h}\frac{1}{|y-k|}+\frac{(k-y-\delta)M_{1}}{h\delta}\to0
\]
as $h\to+\infty$.

The claim \ref{enu:DirchletConvAround_y} is proved considering \eqref{eq:FI-convDirichletDelta}
and \ref{enu:DirichletTails}.
\end{proof}
Finally, we have the Fourier inversion theorem:
\begin{thm}
\label{thm:FIT}Let $K\subseteq X\subseteq\rti^{n}$ and $f\in\gsf(X)$.
Then for all sharply interior $y\in\mathring{K}$, we have
\begin{equation}
\lim_{h\to+\infty}\mathcal{F}_{h}^{-1}\left(\mathcal{F}_{k}(f)\right)(y)=\lim_{h\to+\infty}\mathcal{F}_{h}\left(\mathcal{F}_{k}^{-1}(f)\right)(y)=f(y).\label{eq:FIT}
\end{equation}
\end{thm}

\begin{proof}
To link the integrals of the previous Lem.~\ref{lem:DirichletTails}
with $f(y)$, we use the Fermat-Reyes Thm.~\ref{thm:FR-forGSF}:
for any $\delta\in\rti_{>0}$ sufficiently small such that $B_{\delta}(y)\subseteq K\subseteq X$,
we can write $f(x)=f(y)+(y-x)f'[y;y-x]$ for all $x\in[y-\delta,y+\delta]$.
Therefore
\begin{multline*}
\int_{y-\delta}^{y+\delta}\delta_{h}^{1}(y-x)\cdot f(x)\,\diff x=\int_{y-\delta}^{y+\delta}\delta_{h}^{1}(y-x)\cdot\left(f(y)+(y-x)f'[y;y-x]\right)\,\diff x\\
=f(y)\int_{y-\delta}^{y+\delta}\delta_{h}^{1}(y-x)\,\diff x+\frac{1}{\pi}\int_{y-\delta}^{y+\delta}\sin(h(y-x))f'[y;y-x]\,\diff x.
\end{multline*}
Thereby
\[
\left|\int_{y-\delta}^{y+\delta}\delta_{h}^{1}(y-x)\cdot f(x)\,\diff x-f(y)\int_{y-\delta}^{y+\delta}\delta_{h}^{1}(y-x)\,\diff x\right|\le\frac{1}{\pi}2\delta M_{1\delta},
\]
where $M_{1\delta}:=\max_{x\in[y-\delta,y+\delta]}|f'[x;y-x]|\in\rti$.
Considering \eqref{eq:FI-convDirichletDelta}, we have
\begin{align*}
\left|\mathcal{F}_{h}^{-1}\left(\mathcal{F}_{k}(f)\right)(y)-f(y)\right. & \left.\int_{-k}^{k}\delta_{h}^{1}(y-x)\,\diff x\right|\le\frac{1}{\pi}2\delta M_{1\delta}+\\
 & \left|\int_{-k}^{y-\delta}\delta_{h}^{1}(y-x)\cdot f(x)\,\diff x+\int_{y+\delta}^{k}\delta_{h}^{1}(y-x)\cdot f(x)\,\diff x\right.\\
 & -\left.f(y)\int_{-k}^{y-\delta}\delta_{h}^{1}(y-x)\,\diff x-f(y)\int_{y+\delta}^{k}\delta_{h}^{1}(y-x)\,\diff x\right|.
\end{align*}
Take the limit both for $h\to+\infty$ and for $\delta\to0^{+}$ in
this inequality, considering that the left hand side does not depend
on $\delta$. Using Lem.~\ref{lem:DirichletTails}, we obtain
\[
\lim_{h\to+\infty}\left|\mathcal{F}_{h}^{-1}\left(\mathcal{F}_{k}(f)\right)(y)-f(y)\int_{-k}^{k}\delta_{h}^{1}(y-x)\,\diff x\right|=0.
\]
Finally, Lem.~\ref{lem:DirichletMass1} yields $\lim_{h\to+\infty}f(y)\int_{-k}^{k}\delta_{h}^{1}(y-x)\,\diff x=f(y)$
and this proves the first part of the claim. To prove the second equality
in \eqref{eq:FIT}, we use the general equalities
\begin{align}
(2\pi)^{n}\mathcal{F}_{h}^{-1}(g) & =\mathcal{F}_{h}(-1\diamond g)\qquad\forall g\in\gsf(H)\label{eq:3}\\
(2\pi)^{n}\mathcal{F}_{h}^{-1}(f) & =-1\diamond\mathcal{F}_{k}(f)\qquad\forall f\in\gsf(K).\label{eq:4}
\end{align}
Applying \eqref{eq:3} with $g=\mathcal{F}_{k}(f)|_{H}$, we get
\begin{equation}
(2\pi)^{n}\mathcal{F}_{h}^{-1}(\mathcal{F}_{k}(f))=\mathcal{F}_{h}(-1\diamond\mathcal{F}_{k}(f)).\label{eq:5}
\end{equation}
Thereby, using \eqref{eq:4} and \eqref{eq:5}, we get $(2\pi)^{n}\mathcal{F}_{h}(\mathcal{F}_{k}^{-1}(f))=\mathcal{F}_{h}((2\pi)^{n}\mathcal{F}_{k}^{-1}(f))=\mathcal{F}_{h}(-1\diamond\mathcal{F}_{k}(f))=(2\pi)^{n}\mathcal{F}_{h}^{-1}(\mathcal{F}_{k}(f))$.
\end{proof}
\noindent One way to summarize the meaning of this version of the
Fourier inversion theorem is as follows: If for a smooth function
$g\in\mathcal{C}^{\infty}(\R)$, we want to define
\[
\int_{-\infty}^{+\infty}g(\omega)e^{ix\omega}\,\diff\omega:=\lim_{h\to+\infty}\int_{-h}^{h}g(\omega)e^{ix\omega}\,\diff\omega,
\]
we can assume some sufficiently strong behavior of $g$ at $\pm\infty$,
e.g.~that $g$ is rapidly decreasing. On the one hand, this is only
a sufficient condition deeply linked to the limitations of the Lebesgue
integral, as the function $g(x)=\frac{\sin(x)}{x}$ shows. On the
other hand, \eqref{eq:FIT} shows that this type of limit exists (in
the sharp topology) for all the GSF of the form $g=\mathcal{F}_{k}(f)$,
where $f$ is an \emph{arbitrary} GSF and $k\in\rti$ is \emph{any
fixed infinite number}. Indeed, Thm.~\ref{thm:FIT} is strongly related
to the Dirichlet delta, as \eqref{eq:FI-convDirichletDelta} already
shows. In other words, a key idea of the present work is to consider
the HFT $\mathcal{F}_{h}^{-1}(f)$ for arbitrary $f\in\gsf(H)$, and
to take the limit for $h\to+\infty$ only in the Fourier inversion
theorem.

We can also state the Fourier inversion theorem using a strong equivalence
relation instead of a limit. For this aim, we need the following notions:
\begin{defn}
\label{def:RCud}If $\sigma$ is a gauge \emph{smaller} than $\rho$,
and we write $\sigma\le\rho^{*}$, i.e.~if
\[
\exists R\in\R_{>0}\,\forall^{0}\eps:\ \sigma_{\eps}\le\rho_{\eps}^{R},
\]
then we have $\R_{\rho}\subseteq\R_{\sigma}$ and we can hence consider
the set of $\rho$\emph{-moderate numbers in} $\RC{\sigma}$:
\[
\RCud{\rho}{\sigma}:=\left\{ [x_{\eps}]\in\RC{\sigma}\mid(x_{\eps})\in\R_{\rho}\right\} .
\]
Let $\partial\rho:=[\rho_{\eps}]\in\RC{\sigma}$ denotes the generalized
number in $\RC{\sigma}$ defined by the net $(\rho_{\eps})$. If $x$,
$y\in\RC{\sigma}$, we say that $x$ \emph{is equal up to $\rho$
to $y$}, and we write $x=_{\rho}y$, if
\[
\forall q\in\N:\ |x-y|\le\partial\rho^{q}.
\]
We say that $\sigma$ is an \emph{auxiliary gauge of }$\rho$, and
we write $\sigma\ll\rho$, if
\begin{equation}
\exists Q\in\mathbb{N}\,\forall q\in\mathbb{N}\,\forall^{0}\eps:\ \sigma_{\eps}^{Q}\leq\rho_{\eps}^{q}.\label{eq:sigma-ll-rho}
\end{equation}
Finally, we say that $k\in\RC{\sigma}$ is $\rho$\emph{-immoderate},
and we write $k\gg\partial\rho^{-*}$ if
\[
\forall Q\in\N:\ k\ge\partial\rho^{-Q}.
\]
\end{defn}

\begin{rem}
\label{rem:RCud}~
\begin{enumerate}
\item \label{enu:RCud}Clearly, $\RCud{\rho}{\sigma}$ is a subring of $\RC{\sigma}$,
but in general it is not isomorphic to $\rti$ because the notion
of equality $\sim_{\sigma}$ in $\RC{\sigma}$ is generally stronger
than the one $\sim_{\rho}$ in $\rti$. However, if $[x_{\eps}]_{\sigma}\in\RC{\sigma}$
and $[x_{\eps}]_{\rho}\in\rti$ denotes the equivalence classes generated
by the net $(x_{\eps})\in\R_{\rho}$, then the map
\[
\iota:[x_{\eps}]_{\sigma}\in\RCud{\rho}{\sigma}\mapsto[x_{\eps}]\in\rti
\]
 is surjective and ``injective up to $\rho$'', i.e.~$\iota(x)=\iota(y)$
implies $x=_{\rho}y$. Similarly, we can define
\[
f\in\GSFud{\rho}{\sigma}(X,Y):\Leftrightarrow X\subseteq\RCud{\rho}{\sigma},\ Y\subseteq\RCud{\rho}{\sigma},\ \forall x\in X\,\forall\alpha\in\mathbb{N}^{n}:\ \partial^{\alpha}f\left(x\right)\in\RCud{\rho}{\sigma},
\]
and the map $j:\GSFud{\rho}{\sigma}(X,Y)\ra\gsf(\iota(X),\iota(Y))$
defined by $j(f)(\iota(x)):=\iota(f(x))$ is surjective and satisfies
$j(f)=j(g)$ if and only if $f(x)=_{\rho}g(x)$ for all $x\in X$.
\item \label{enu:RCudExamples}$\sigma_{1\eps}:=\rho_{\eps}^{1/\eps}$,
$\sigma_{2\eps}:=\exp\left(-\frac{1}{\rho_{\eps}}\right)$ and $\sigma_{3\eps}:=\exp\left(-\rho_{\eps}^{-1/\eps}\right)$
are all auxiliary gauges of $\rho$, $k_{j}:=\sigma_{j}^{-1}$ and
$-\log k_{3}$ are $\rho$-immoderate numbers. On the other hand,
if $\sigma$ is an arbitrary gauge, and $\rho_{\eps}:=-\log(\sigma_{\eps})^{-1}$,
then $\sigma\ll\rho$.
\end{enumerate}
\end{rem}

\begin{cor}
\label{cor:FIT=00003D_rho}Let $\sigma\ll\rho$ and $k\gg\partial\rho^{-*}$.
Let $f\in\Gsf{\sigma}(X)$, with $K\subseteq X$. Then for all $h\in\RC{\sigma}$
sufficiently large, we have
\[
\mathcal{F}_{h}^{-1}\left(\mathcal{F}_{k}(f)\right)(y)=_{\rho}\mathcal{F}_{h}\left(\mathcal{F}_{k}^{-1}(f)\right)(y)=_{\rho}f(y)
\]
 for all $y\in K\cap\RCud{\rho}{\sigma}$.
\end{cor}

\begin{proof}
From Thm.~\ref{thm:FIT} (with $\sigma$ in the play of $\rho$),
we have that for all $h\in\RC{\sigma}$ sufficiently large
\[
\left|\mathcal{F}_{h}^{-1}\left(\mathcal{F}_{k}(f)\right)(y)-f(y)\right|\le\diff\sigma^{Q}\le\partial\rho^{q}
\]
for all $q\in\N$ from \eqref{eq:sigma-ll-rho}. Note that $y\in K\cap\RCud{\rho}{\sigma}$
and $k\gg\partial\rho^{-*}$ imply $y\in\mathring{K}$ (in the $\sigma$-sharp
topology).
\end{proof}
The following results allows one to have independence from $k$ or
both $k$ and $h$.
\begin{cor}
\label{cor:FITComptSuppGSF}If $f\in\Dgsf(\rti)$, there exists $k\in\rti_{>0}$
infinite such that
\[
\lim_{h\to+\infty}\mathcal{F}_{h}^{-1}\left(\mathcal{F}(f)\right)(y)=\lim_{h\to+\infty}\mathcal{F}_{h}\left(\mathcal{F}^{-1}(f)\right)(y)=f(y)
\]
for all $y\in\mathring{K}$.
\end{cor}

\begin{cor}
\label{cor:FIT-RLlemma}Let $H\fcmp\rti^{n}$ and $f\in\Dgsf(H)$.
Assume that
\[
\exists C,b\in\rti_{>0}\,\forall x\in H\,\forall j\in\N:\ \left|\diff{^{j}f}(x)\right|\le C\cdot b^{j}.
\]
Then
\[
\mathcal{F}^{-1}\left(\mathcal{F}(f)\right)(y)=\mathcal{F}\left(\mathcal{F}^{-1}(f)\right)(y)=f(y)\qquad\forall y\in H.
\]
\end{cor}

\begin{proof}
See the Riemann-Lebesgue Lem.~\ref{lem:Rieman-Lebesgue}, which guarantees
that also $\mathcal{F}(f)$ is compactly supported.
\end{proof}
In the following result, we summarize some properties of Dirichlet
delta function:
\begin{cor}
\label{cor:Dirichlet-delta}Let $K\subseteq X\subseteq\rti^{n}$,
$f\in\gsf(X)$, $h\in\rti$ be a positive infinite number, and $\delta\in(0,1]$.
Then, the following properties hold:
\begin{enumerate}
\item \label{enu:delta1}$\delta_{h}^{n}=\delta_{n}^{1}\cdot\ldots\cdot\delta_{n}^{1}$;
\item \label{enu:mass1}$\lim_{h\to+\infty}\int_{-k}^{k}\delta_{h}^{n}(x)\,\diff x=1$;
\item \label{enu:tails}$\lim_{h\to+\infty}\int_{-k}^{-\delta}\delta_{h}^{n}(x)\,\diff x=\lim_{h\to+\infty}\int_{\delta}^{k}\delta_{h}^{n}(x)\,\diff x=0$;
\item \label{enu:sifting}$\lim_{h\to+\infty}\int_{-k}^{k}\delta_{h}^{n}(x)f(x)\,\diff x=f(0)$
if $0\in\mathring{K}$;
\item \label{enu:conv_k}$\lim_{h\to+\infty}\int_{-k}^{k}\delta_{h}^{n}(y-x)f(x)=:\left(\delta_{h}^{n}*_{k}f\right)(y)=f(y)$
for all $y\in\mathring{K}$;
\item \label{enu:FT1}$\lim_{h\to+\infty}\mathcal{F}_{k}\left(\delta_{h}^{n}\right)=1$
and $\mathcal{F}_{h}^{-1}(1)=\delta_{h}^{n}$;
\item \label{enu:IFT}$\mathcal{F}_{h}^{-1}\left(\mathcal{F}_{k}(f)\right)=\mathcal{F}_{h}\left(\mathcal{F}_{k}^{-1}(f)\right)=\delta_{h}^{n}*_{k}f$
\end{enumerate}
If $\sigma\ll\rho$, and $k\in\RC{\sigma}$, $k\gg\partial\rho^{-*}$,
then in the previous properties we can replace the limits with $=_{\rho}$
and with $h\in\RC{\sigma}$ sufficiently large.

\end{cor}

\begin{proof}
For \ref{enu:delta1}, see \eqref{eq:delta1}. For \ref{enu:mass1},
see Lem.~\ref{lem:DirichletMass1}. For \ref{enu:tails}, see Lem.~\ref{lem:DirichletTails}.
Property \ref{enu:sifting} is exactly Thm.~\ref{thm:FIT} with $y=0$
and considering \eqref{eq:beforeSifting}; the same for \ref{enu:conv_k}.
To prove the first equality of \ref{enu:FT1}, use \ref{enu:conv_k}
with $f(x)=e^{-i\omega x}$; for the second one, simply compute $\mathcal{F}_{h}^{-1}(1)$
and use \eqref{eq:delta1} for $n=1$. Finally, \ref{enu:IFT} can
be proved as in \eqref{eq:FI-convDirichletDelta}.
\end{proof}

\subsection{Parseval's relation, Plancherel's identity and the uncertainty principle}
\begin{thm}
\label{thm:MainPropertiesHFT}Let $h\in\rcrho_{>0}$ be an infinite
number and set $H:=\left[-h,h\right]^{n}$. Let $f\in\gsf\left(K\right)$
and $g\in\gsf\left(H\right)$. Then
\begin{enumerate}
\item \label{enu:IntegralProp}$\intop_{H}\mathcal{F}_{k}\left(f\right)\left(\omega\right)g\left(\omega\right)\,\diff{\omega}=\intop_{K}f\left(x\right)\mathcal{F}_{h}\left(g\right)\left(x\right)\,\diff{x}$.
\item \label{enu:hom}$\mathcal{F}_{k}|_{K}:\gsf(K)\ra\gsf(K)$ is an injective
homeomorphism such that
\begin{equation}
\forall f\in\gsf(K)\,\exists g\in\gsf(K):\ \lim_{h\to+\infty}\mathcal{F}_{h}(g)=f.\label{eq:surj}
\end{equation}
\item \label{enu:FF}$\lim_{h\to+\infty}\mathcal{F}_{h}|_{H}(\mathcal{F}_{k}|_{K}(f))=(2\pi)^{n}\left(-1\diamond f\right)$
and $\lim_{h\to+\infty}\mathcal{F}_{h}^{-1}|_{H}(\mathcal{F}_{k}^{-1}|_{K}(f))=(2\pi)^{-n}\left(-1\diamond f\right)$,
where we recall that $-1\diamond f$ is the reflection of $f$. 
\end{enumerate}
If $H=K$, then
\begin{enumerate}[resume]
\item \label{enu:Parseval's-relation}(Parseval's relation) $(2\pi)^{n}\intop_{K}f\overline{g}=\lim_{h\to+\infty}\intop_{K}\mathcal{F}_{h}\left(f\right)\overline{\mathcal{F}_{k}\left(g\right)}$.
\item \label{enu:Plancherel=002019s-identity}(Plancherel’s identity) $(2\pi)^{n}\intop_{K}\left|f\right|^{2}=\lim_{h\to+\infty}\intop_{K}\left|\mathcal{F}_{h}\left(f\right)\right|^{2}$.
\item \label{enu:five}$\intop_{K}fg=\lim_{h\to+\infty}\intop_{K}\mathcal{F}_{h}\left(f\right)\mathcal{F}_{k}^{-1}\left(g\right)$.
\end{enumerate}
In the assumptions of Cor.~\ref{cor:FIT=00003D_rho}, we can also
write all the relations involving $\lim_{h\to+\infty}$ using $=_{\rho}$
instead. Finally, using Thm.~\ref{thm:contResult}, we can also take
the limit under the integral sign.

\end{thm}

\begin{proof}
\ref{enu:IntegralProp} follows from Def.~\ref{def:HyperfiniteFT}
and Fubini's theorem.

In order to prove \ref{enu:hom}, we assume $\mathcal{F}_{k}(f)=\mathcal{F}_{k}(g)$,
so that $\mathcal{F}_{h}^{-1}\left(\mathcal{F}_{k}(f)\right)=\mathcal{F}_{h}^{-1}\left(\mathcal{F}_{k}(g)\right)$
and hence $f=g$ on $\mathring{K}$ by the inversion Thm.~\ref{thm:FIT}.
The equality on the entire $K$ follows by sharp continuity. If $f\in\gsf(K)$,
set $g:=\mathcal{F}_{k}^{-1}(f)|_{K}\in\gsf(K)$, then \eqref{eq:surj}
follows again from the Fourier inversion theorem.

To prove \ref{enu:FF}, using \eqref{eq:reflection} we have $\mathcal{F}_{h}\left(\mathcal{F}_{k}(f)\right)=\mathcal{F}_{h}\left(\mathcal{F}_{k}(-1\diamond(-1\diamond f))\right)=\mathcal{F}_{h}\left(-1\diamond\mathcal{F}_{k}(-1\diamond f)\right)=(2\pi)^{n}\mathcal{F}_{h}\left(\mathcal{F}_{k}^{-1}(-1\diamond f)\right)\to(2\pi)^{n}(-1\diamond f)$as
$h\to+\infty$.

To prove \ref{enu:Parseval's-relation}, use \ref{enu:IntegralProp}
with $\overline{\mathcal{F}_{k}\left(g\right)}$ instead of $g$,
then Thm.~\ref{thm:thmProperties}.\ref{enu:prop3}, and finally
\ref{enu:FF}.

Plancherel’s identity \ref{enu:Plancherel=002019s-identity} is a
trivial consequence of \ref{enu:Parseval's-relation}.

Finally, \ref{enu:five} follows from \ref{enu:IntegralProp} and
the inversion Thm.~\ref{thm:FIT}.
\end{proof}
We close this section with a proof of the uncertainty principle:
\begin{thm}
\label{thm:uncertainty}If $\psi\in\Dgsf(\rti)$, then
\begin{enumerate}
\item If $\psi\in\Dgsf(H)\cap\Dgsf(K)$, then
\[
\intop_{H}\omega^{2}\left|\mathcal{F}\left(\psi\right)\left(\omega\right)\right|^{2}\,\diff{\omega}=\intop_{K}\omega^{2}\left|\mathcal{F}\left(\psi\right)\left(\omega\right)\right|^{2}\,\diff{\omega}=:\intop\omega^{2}\left|\mathcal{F}\left(\psi\right)\left(\omega\right)\right|^{2}\,\diff{\omega}
\]
\item \label{enu:uncertainty}$\left(\intop x^{2}\left|\psi\left(x\right)\right|^{2}\,\diff{x}\right)\left(\intop\omega^{2}\left|\mathcal{F}\left(\psi\right)\left(\omega\right)\right|^{2}\,\diff{\omega}\right)\geq\frac{1}{4}\Vert\psi\Vert_{2}\Vert\mathcal{F}(\psi)\Vert_{2}$.
\end{enumerate}
\end{thm}

\begin{proof}
Properties \ref{enu:derZeroExt} and \ref{enu:equivCmptSupp} of Thm.~\ref{thm:DerivativeIsZero}
imply that also $\psi'\in\Dgsf(H)$. Therefore, Plancherel's identity
Thm.~\ref{thm:MainPropertiesHFT}.\ref{enu:Plancherel=002019s-identity}
yields
\[
\int_{H}\left|\psi'\right|^{2}=\frac{1}{2\pi}\int_{H}\left|\mathcal{F}(\psi')\right|^{2}.
\]
But $\left|\mathcal{F}(\psi')\right|^{2}=\omega^{2}\left|\mathcal{F}(\psi)\right|^{2}$
from Thm.~\ref{thm:thmProperties}.\ref{enu:prop8} because $\psi$
is compactly supported\emph{ }and hence $\Delta_{1k}\psi=0$. Therefore
\begin{equation}
\int_{H}\left|\psi'\right|^{2}=\frac{1}{2\pi}\int_{H}\omega^{2}\left|\mathcal{F}(\psi)(\omega)\right|^{2}\,\diff{\omega}.\label{eq:varHFT}
\end{equation}
At the same result we arrive considering $K$ instead of $H$. Finally,
we apply Def.~\ref{def:intCmpSupp} of integral of a compactly supported
GSF, which yields independence from the functionally compact integration
domain.

To prove inequality \ref{enu:uncertainty}, we assume that $\psi\in\Dgsf(K)$;
using integration by parts, we get:
\begin{align*}
\int x\overline{\psi(x)}\psi'(x)\,\diff{x} & =\int_{-k}^{k}x\overline{\psi(x)}\psi'(x)\,\diff{x}=\\
 & =\left[x\overline{\psi(x)}\psi(x)\right]_{x=-k}^{x=k}-\int\psi(x)\left(\overline{\psi(x)}+x\overline{\psi'(x)}\right)\,\diff{x}=\\
 & =-\int\left[\left|\psi(x)\right|^{2}+x\psi(x)\overline{\psi'(x)}\right]\,\diff{x}.
\end{align*}
Thereby
\begin{align*}
\int\left|\psi(x)\right|^{2}\,\diff{x} & =-2\text{Re}\left(\int x\psi(x)\overline{\psi'(x)}\,\diff{x}\right)\le\\
 & \le2\left|\text{Re}\left(\int x\psi(x)\overline{\psi'(x)}\,\diff{x}\right)\right|\le\\
 & \le2\int\left|x\psi(x)\overline{\psi'(x)}\right|\,\diff{x}.
\end{align*}
Where we used the triangle inequality for integrals (see Thm.~\ref{thm:muMeasurableAndIntegral}.\ref{enu:existsReprDefInt}).
Using Cauchy-Schwarz inequality (see Thm.~\ref{thm:Holder}), we
hence obtain
\begin{align*}
\left(\int\left|\psi(x)\right|^{2}\,\diff{x}\right)^{2} & \le4\left(\int\left|x\psi(x)\overline{\psi'(x)}\right|\,\diff{x}\right)^{2}\le\\
 & \le4\left(\int x^{2}\left|\psi(x)\right|^{2}\,\diff{x}\right)\left(\int\left|\psi'(x)\right|^{2}\,\diff{x}\right).
\end{align*}
From this, thanks to \eqref{eq:varHFT} and Plancherel's equality,
the claim follows.
\end{proof}
\noindent Note that if $\Vert\psi\Vert_{2}\in\rti$ is invertible,
then also $\Vert\mathcal{F}(\psi)\Vert_{2}$ is invertible by Plancherel's
equality, and we can hence write the uncertainty principle in the
usual normalized form.
\begin{example}
\label{exa:uncertaintyDelta}On the contrary with respect the classical
formulation in $L^{2}(\R)$ of the uncertainty principle, in Thm.~\ref{thm:uncertainty}
we can e.g.~consider $\psi=\delta\in\Dgsf(\rti)$, and we have
\[
\int x^{2}\delta(x)^{2}\,\diff{x}=\left[\int_{-1}^{1}x^{2}b_{\eps}^{2}\psi_{\eps}(b_{\eps}x)^{2}\,\diff{x}\right]
\]
where $\psi(x)=[\psi_{\eps}(x_{\eps})]$ is a Colombeau mollifier
and $b=[b_{\eps}]\in\rti$ is a strong infinite number (see Example
\ref{exa:deltaCompDelta}). Since normalizing the function $\eps\mapsto b_{\eps}^{2}\psi_{\eps}(b_{\eps}x)^{2}$
we get an approximate identity, we have $\lim_{\eps\to0^{+}}\int_{-1}^{1}x^{2}b_{\eps}^{2}\psi_{\eps}(b_{\eps}x)^{2}\,\diff{x}=0$,
and hence $\int x^{2}\delta(x)^{2}\,\diff{x}\approx0$ is an infinitesimal.
The uncertainty principle Thm.~\ref{thm:uncertainty} implies that
it is an invertible infinitesimal. Considering the HFT $\mathbb{1}=\mathcal{F}(\delta)\in\Dgsf(\rti)$,
we have
\[
\int\omega^{2}\mathbb{1}(\omega)^{2}\,\diff{\omega}\ge\int_{-r}^{r}\omega^{2}\,\diff{\omega}=2\frac{r^{3}}{3}\quad\forall r\in\R_{>0}.
\]
Thereby, $\int\omega^{2}\mathbb{1}(\omega)^{2}\,\diff{\omega}$ is
an infinite number.
\end{example}

\section{\label{sec:preservation}Preservation of classical Fourier transform}

It is natural to inquire the relations between classical FT of tempered
distributions and our HFT.

Let us start with a couple of exploring examples:
\begin{enumerate}
\item $\mathcal{F}_{k}(1)(\omega)=\int_{-k}^{k}1\cdot e^{-ix\omega}\,\diff{x}=\int_{-k}^{k}\cos(x\omega)\,\diff{x}$.
If $L\subzero I$ and $\omega|_{L}$ is invertible (see Sec.~\ref{subsec:subpoints}
for the language of subpoints), then $\mathcal{F}_{k}(1)(\omega)=_{L}2\frac{\sin(k\omega)}{\omega}$;
if $\omega=_{L}0$, then $\mathcal{F}_{k}(1)(\omega)=2k$. Classically,
we have $\hat{1}=2\pi\delta$.
\item \label{enu:HFT_H}$\mathcal{F}_{k}(H)(\omega)=\int_{-k}^{k}H(x)e^{-ix\omega}\,\diff{x}$.
Assuming that $\omega|_{L}$ is invertible on $L\subzero I$, we have
$\mathcal{F}_{k}(H)(\omega)=_{L}\frac{i}{\omega}e^{-ik\omega}-\frac{i}{\omega}\mathbb{1}(\omega)$.
Classically, we have $\hat{H}=\pi\delta-\frac{i}{\omega}$. Therefore,
if $k$ is a strong infinite number and $\omega$ is far from the
origin, $|\omega|\ge r\in\R_{>0}$, we have $\mathcal{F}_{k}(H)(\omega)=\iota_{\R}^{b}\left(\hat{H}\right)(\omega)$
(here $\iota_{\R}^{b}$ is an embedding of $\mathcal{D}'(\R)$ into
$\gsf(\csp{\R})$, see Sec.~\ref{sec:preservation}). However, the
latter equality does not hold if $\omega\approx0$.
\end{enumerate}
Since classically we do not have infinite numbers, the former of these
examples leads us to the following idea
\[
\mathcal{F}(1\cdot\mathbb{1})=\mathcal{F}(\mathcal{F}(\delta))=2\pi\left(-1\diamond\delta\right)=2\pi\delta.
\]
Note that if $f\in\gsf(K)$, then $\left(f\cdot\mathbb{1}\right)(\omega)=f(\omega)$
for all finite point $\omega\in K$. We therefore call $f\cdot\mathbb{1}$
the \emph{finite part of $f$}. The same idea works for $e^{iax}$
and hence also for $\sin$, $\cos$. Let us now consider $\delta\cdot\mathbb{1}$:
\[
\mathcal{F}(\delta\cdot\mathbb{1})(\omega)=\int\delta(x)\mathcal{F}(\delta)(x)e^{-ix\omega}\,\diff{x}.
\]
We recall that integrating against $\delta$ yields the evaluation
of the second factor at $0$ only if the latter is bounded by a tame
polynomial at $0$ (see Example \ref{exa:tamePol}.\ref{enu:intAgainstDelta}).
But the function $x\mapsto\mathcal{F}(\delta)(x)e^{-ix\omega}$ is
bounded by a tame polynomial at $x=0$ for all $\omega$, and we get
$\mathcal{F}(\delta\cdot\mathbb{1})(\omega)=1$. Being bounded by
a tame polynomial is, in general, a necessary assumption because
\begin{align*}
\mathcal{F}(H\cdot\mathbb{1})(\omega) & =\int H(x)\cdot\mathcal{F}(\delta)(x)e^{-ix\omega}\,\diff{x}=\\
 & =\int\delta(x)\mathcal{F}_{k}(H\cdot e^{-i(-)\omega})(x)\,\diff{x}=\\
 & =\int\delta(x)\mathcal{F}_{k}(H)(x+\omega)\,\diff{x},
\end{align*}
but $\mathcal{F}_{k}(H)(x+\omega)=\frac{i}{x+\omega}e^{-ik(x+\omega)}-\frac{i}{x+\omega}\mathbb{1}(x+\omega)$
is not bounded by a tame polynomial at $x=0$ if $\omega\approx0$
because of the terms $\frac{1}{\omega}$.

These exploratory examples lead us to the following
\begin{thm}
\label{thm:preservationFinitePart}Let $f\in\gsf(K)$, and assume
that $\mathcal{F}_{k}(f)$ is bounded by a tame polynomial at $\omega\in\rti^{n}$.
Then $\mathcal{F}(f\cdot\mathbb{1})(\omega)=\mathcal{F}_{k}(f)(\omega)$.
\end{thm}

\begin{proof}
It suffices to apply Thm.~\ref{thm:MainPropertiesHFT}.\ref{enu:IntegralProp}:
\begin{align*}
\mathcal{F}(f\cdot\mathbb{1})(\omega) & =\int f(x)\mathcal{F}(\delta)(x)e^{-ix\cdot\omega}\,\diff{x}=\\
 & =\int\delta(x)\mathcal{F}_{k}\left(f\cdot e^{-i(-)\cdot\omega}\right)(x)\,\diff{x}=\\
 & =\int\delta(x)\mathcal{F}_{k}\left(f\right)(x+\omega)\,\diff{x}=\mathcal{F}_{k}\left(f\right)(\omega).
\end{align*}
\end{proof}
Since $\frac{\partial}{\partial x_{j}}\left[\mathcal{F}_{k}(f)\right](\omega)=-i\mathcal{F}_{k}(x_{j}f)(\omega)$,
we have the following sufficient condition for $\mathcal{F}_{k}(f)$
being bounded by a tame polynomial at $\omega\in\rti^{n}$:
\begin{thm}
\label{thm:preservationS}Let $b\in\rti_{>0}$ be a large infinite
number, and let $f\in\gsf(K)$ be uniformly bounded by a tame polynomial
at $K$, i.e.
\begin{equation}
\exists M,c\in\rti\,\forall y\in K\,\forall j\in\N:\ \left|\diff{^{j}}f(y)\right|\le M\cdot c^{j},\quad\frac{b}{c}\text{ is a large infinite number}.\label{eq:unifBTP}
\end{equation}
Then for all $\omega\in\rti^{n}$, the HFT $\mathcal{F}_{k}(f)$ is
bounded by a tame polynomial at $\omega$. In particular, every $f\in\mathcal{S}(\R^{n})$
satisfies condition \eqref{eq:unifBTP}, and hence
\begin{equation}
\mathcal{F}(f)=\mathcal{F}(f\cdot\mathbb{1})=\iota_{\R^{n}}^{b}(\hat{f}),\label{eq:preservRapDecr}
\end{equation}
where $\hat{f}\in\mathcal{S}(\R^{n})$ is the classical FT of $f$.
\end{thm}

\begin{proof}
We have
\begin{align*}
\left|\diff{^{j}}\mathcal{F}_{k}(f)(\omega)\right| & \le\max_{h\le j}\left|\mathcal{F}_{k}(x_{h}f)(\omega)\right|\le\max_{h\le j}\int_{K}\left|x_{h}f(x)\right|\,\diff{x}\le\\
 & \le Mc^{j}\max_{h\le j}\int_{K}|x_{h}|\,\diff{x}=:\bar{M}c^{j}.
\end{align*}
If $f\in\mathcal{S}(\R^{n})$, then $\left|\diff{^{j}}f(y)\right|\in\R$,
so that if $b\ge\diff{\rho}^{-r}$, $r\in\R_{>0}$, it suffices to
take $c=\diff{\rho}^{-r+s}$ where $0<s<r$ to have that \eqref{eq:unifBTP}
holds. The last equality in \eqref{eq:preservRapDecr} is equivalent
to say that $\int_{\R^{n}}f(x)e^{-ix\cdot\omega}\,\diff{x}=\int_{K}f(x)e^{-ix\cdot\omega}\,\diff{x}$,
which can be proved as for the Gaussian, see Lem.~\ref{lem:example}.
\end{proof}
We can now consider the relations between $\iota_{\R^{n}}^{b}(\hat{T})$
and $\mathcal{F}_{k}(\iota_{\R^{n}}^{b}(T))$ when $T\in\mathcal{S}'(\R^{n})$.
A first trivial link is given by the so-called \emph{equality in the
sense of generalized tempered distributions}: For all $\phi\in\mathcal{S}(\R^{n})$,
from \eqref{eq:pairTphiAsInt} we have
\[
\int\iota_{\R^{n}}^{b}(\hat{T})\phi=\langle\hat{T},\phi\rangle=\langle T,\hat{\phi}\rangle=\int\iota_{\R^{n}}^{b}(T)\hat{\phi}.
\]
Using the previous Thm.~\ref{thm:preservationS} we get $\hat{\phi}=\mathcal{F}(\phi)$
(identifying a smooth function with its embedding). Thereby
\begin{equation}
\int\iota_{\R^{n}}^{b}(\hat{T})\phi=\int\iota_{\R^{n}}^{b}(T)\mathcal{F}(\phi)=\int\mathcal{F}\left(\iota_{\R^{n}}^{b}(T)\right)\phi\quad\forall\phi\in\mathcal{S}(\R^{n}).\label{eq:gtd}
\end{equation}
In Colombeau's theory, this relation is usually written $\iota_{\R^{n}}^{b}(\hat{T})=_{\text{g.t.d.}}\mathcal{F}\left(\iota_{\R^{n}}^{b}(T)\right)$.

In the following result, we give a sufficient condition to have a
pointwise equality between the HFT of the finite part of $\iota_{\R^{n}}^{b}(T)$
and $\hat{T}$:
\begin{thm}
\label{thm:preservation}Let $b\in\rti_{>0}$ be a large infinite
number and $T\in\mathcal{S}'(\R^{n})$. Assume that $\mathcal{F}(\iota_{\R^{n}}^{b}(T))$
is bounded by a tame polynomial at $\omega\in\rti^{n}$. Then
\[
\mathcal{F}_{k}(\iota_{\R^{n}}^{b}(T))(\omega)=\mathcal{F}(\iota_{\R^{n}}^{b}(T)\cdot\mathbb{1})(\omega)=\iota_{\R^{n}}^{b}(\hat{T})(\omega).
\]
\end{thm}

\begin{proof}
For simplicity of notation, we use $\iota:=\iota_{\R^{n}}^{b}$. Using
Thm.~\ref{thm:preservationFinitePart}, we have
\[
\mathcal{F}\left(\iota(T)\cdot\mathbb{1}\right)(\omega)=\mathcal{F}_{k}(\iota(T))(\omega).
\]
Let $\psi(x)=[\psi_{\eps}(x_{\eps})]$ \emph{be an $n$-}dimensional
Colombeau mollifier defined by $b$, and set $K_{\eps}:=[-k_{\eps},k_{\eps}]^{n}$;
we have
\begin{align*}
\mathcal{F}_{k}\left(\iota(T)\right)(\omega) & =\left[\int_{K_{\eps}}\langle T(y),\psi_{\eps}(x-y)\rangle e^{-ix\cdot\omega_{\eps}}\,\diff{x}\right]=\\
 & =\left[\langle T(y),\int\psi_{\eps}(x-y)e^{-ix\cdot\omega_{\eps}}\,\diff{x}\rangle\right]=\\
 & =\left[\langle T(y),\widehat{y\oplus\psi_{\eps}}(\omega_{\eps})\rangle\right]=\\
 & =\left[\langle\hat{T}(y),\left(y\oplus\psi_{\eps}\right)(\omega_{\eps})\rangle\right]=\\
 & =\left[\langle\hat{T}(y),\psi_{\eps}(\omega_{\eps}-y)\rangle\right]=\iota(\hat{T})(\omega).
\end{align*}
\end{proof}

\subsection{Fourier transform in the Colombeau's setting}

Only in this section we assume a very basic knowledge of Colombeau's
theory.

Assume that $\Omega\subseteq\mathbb{R}^{n}$ is an open set. The algebra
$\mathcal{G}_{\tau}^{s}\left(\Omega\right)$ of tempered generalized
functions was introduced by J.F.~Colombeau in \cite{C1}, in order
to develop a theory of Fourier transform. Since then, there has been
a rapid development of Fourier analysis, regularity theory and microlocal
analysis in this setting.
\begin{defn}
The $\mathcal{G}_{\tau}^{s}(\Omega)$ algebra of \emph{Colombeau tempered
GF} (trivially generalized by using an arbitrary gauge $\rho$) is
defined as follows:
\begin{align*}
\mathcal{E}_{\tau}^{s}\left(\Omega\right) & :=\left\{ \left(u_{\eps}\right)\in\Coo(\Omega)^{I}\mid\forall\alpha\in\mathbb{N}^{n}\,\exists N\in\mathbb{N}:\right.\\
 & \phantom{:=\qquad}\left.\sup_{x\in\Omega}\left(1+\left|x\right|\right)^{-N}\left|\partial^{\alpha}u_{\eps}\left(x\right)\right|=O(\rho_{\eps}^{-N})\right\} ,
\end{align*}

\begin{align*}
\mathcal{N}_{\tau}^{s}\left(\Omega\right) & :=\left\{ \left(u_{\eps}\right)\in\Coo(\Omega)^{I}\mid\forall\alpha\in\mathbb{N}^{n}\,\exists p\in\N\,\forall m\in\mathbb{N}:\right.\\
 & \phantom{:=\qquad}\left.\sup_{x\in\Omega}\left(1+\left|x\right|\right)^{-p}\left|\partial^{\alpha}u_{\eps}\left(x\right)\right|=O(\rho_{\eps}^{m})\right\} ,
\end{align*}
\[
\mathcal{G}_{\tau}^{s}(\Omega):=\mathcal{E}_{\tau}^{s}(\Omega)/\mathcal{N}_{\tau}^{s}(\Omega).
\]
\end{defn}

\noindent Colombeau tempered GF can be embedded as GSF, at least if
the internal set $[\Omega]$ is sharply bounded. We first define
\begin{defn}
\label{def:gsf_tau}Let $X\subseteq\rti^{n}$, then
\[
\gsf_{\tau}(X):=\left\{ u\in\gsf(X)\mid\forall\alpha\in\mathbb{N}^{n}\,\exists N\in\mathbb{N}\,\forall x\in X:\ \left|\partial^{\alpha}u\left(x\right)\right|\le\frac{\left(1+\left|x\right|\right)^{N}}{\diff{\rho}^{N}}\right\} .
\]
\end{defn}

\begin{thm}
\label{thm:embCTGF}Let $\Omega\subseteq\R^{n}$ be an open set such
that $[\Omega]$ is sharply bounded. A Colombeau tempered GF $u=(u_{\eps})+\ns_{\tau}(\Omega)\in\gs_{\tau}(\Omega)$
defines a GSF $u:[x_{\eps}]\in[\Omega]\longrightarrow[u_{\eps}(x_{\eps})]\in\rccrho$.
This assignment provides an algebra isomorphism $\gs_{\tau}(\Omega)\simeq\gsf_{\tau}([\Omega])$.
\end{thm}

Integration of a CGF $u=[u_{\eps}]\in\gs(\Omega)$ over a standard
$K\Subset\Omega$ can be defined $\eps$-wise as $\int_{K}u\left(x\right)\,\diff{x}:=\left[\int_{K}u_{\varepsilon}\left(x\right)\,\diff{x}\right]\in\rti$.
Similarly we can proceed for $\int_{\Omega}u$ if $u$ is compactly
supported and $\Omega\subseteq\R^{n}$ is an arbitrary open set. On
the other hand, to define the Fourier transform, we have to integrate
tempered CGF on the entire $\R^{n}$. Using this integration of CGF,
this is accomplished by multiplying the generalized function by a
so-called \emph{damping measure} $\phi$, see e.g.~\cite{Hor99}:
\begin{defn}
\label{def:damping}Let $\phi\in\mathcal{S}(\R^{n})$ with $\int_{\R^{n}}\phi=1$,
$\int_{\R^{n}}x^{\alpha}\phi(x)\,\diff{x}=0$ for all $\alpha\in\N^{n}\setminus\{0\}$,
and set $\phi_{\eps}(x):=\rho_{\eps}\odot\phi(x)=\rho_{\eps}^{-n}\phi(\rho_{\eps}^{-1}x)$.
Let $u=[u_{\eps}]\in\mathcal{G}_{\tau}(\R^{n})$, then $u_{\hat{\phi}}:=[u_{\eps}\widehat{\phi_{\eps}}]$,
$\int_{\R^{n}}u(x)\,\diff{}_{\hat{\phi}}x:=\int_{\R^{n}}u_{\hat{\phi}}\,\diff{x}=\left[\int_{\R^{n}}u_{\eps}(x)\widehat{\phi_{\eps}}(x)\,\diff{x}\right]\in\Ctil$,
where $\widehat{\phi_{\eps}}$ denotes the classical FT. In particular,
\begin{align*}
\mathcal{F}_{\hat{\phi}}(u)(\omega) & :=\int_{\R^{n}}e^{-ix\omega}u(x)\,\diff{}_{\hat{\phi}}x=\left[\int_{\R^{n}}e^{-ix\omega}u_{\eps}(x)\widehat{\phi_{\eps}}(x)\,\diff{x}\right]\\
\mathcal{F}_{\hat{\phi}}^{*}(u)(x) & :=(2\pi)^{-n}\int_{\R^{n}}e^{-ix\omega}u(\omega)\,\diff{}_{\hat{\phi}}\omega=\left[(2\pi)^{-n}\int_{\R^{n}}e^{-ix\omega}u_{\eps}(\omega)\widehat{\phi_{\eps}}(\omega)\,\diff{\omega}\right].
\end{align*}
\end{defn}

As a result, although this notion of Fourier transform in the Colombeau
setting shares several properties with the classical one, it lacks
e.g.~the Fourier inversion theorem, which holds only at the level
of equality in the sense of generalized tempered distributions \cite{Col85,Das91,NedPil92},
see also \eqref{eq:gtd}. See also \cite{Sor96} for a Paley-Wiener
like theorem. In other words, we only have e.g.~$\mathcal{F}_{\hat{\phi}}(\partial^{\alpha}u)=_{\text{g.t.d.}}i^{|\alpha|}\omega^{\alpha}\mathcal{F}_{\hat{\phi}}(u)$,
$i^{|\alpha|}\mathcal{F^{*}}_{\hat{\phi}}(\partial^{\alpha}u)=_{\text{g.t.d.}}x^{\alpha}\mathcal{F^{*}}_{\hat{\phi}}(u)$,
$\mathcal{F}_{\hat{\phi}}\mathcal{F^{*}}_{\hat{\phi}}u=_{\text{g.t.d.}}\mathcal{F^{*}}_{\hat{\phi}}\mathcal{F}_{\hat{\phi}}u$,
where $\mathcal{F}_{\hat{\phi}}(u)$ denotes the Fourier transform
with respect to the damping measure. Moreover $\langle\iota_{\R}(\hat{T}),\psi\rangle\approx\langle\mathcal{F}_{\hat{\phi}}\iota_{\R}(T),\psi\rangle$
for all $T\in\mathcal{S}'(\R)$ and all $\psi\in\mathcal{S}(\R)$,
where $\iota_{\R}(T)$ is the embedding of Thm.~\ref{thm:embeddingD'}.
Intuitively, one could say that the use of the multiplicative damping
measure introduces a perturbation of infinitesimal frequencies that
inhibit several results that, on the contrary, hold for the HFT. On
the other hand, HFT lies on a better integration theory that allows
to integrate any GSF on the functionally compact set $K$. The only
possibilities to obtain a strict Fourier inversion theorem in Colombeau's
theory, are the approach used by \cite{Nig16}, where smoothing kernels
are used as index set (instead of the simpler $\eps\in I$) and therefore
the knowledge of infinite dimensional calculus in convenient vector
spaces is needed, or \cite{RaTkRa,BoSa10}, which are based on the
Colombeau space $\mathcal{G}(\mathcal{S}(\R))$, but where the imbedding
of $\mathcal{S}'(\R)$ is more complicated.

Finally, the following result links the HFT with the FT of tempered
CGF as defined above.
\begin{thm}
\label{thm:presFT_TCGF}Let $\Omega\subseteq\R^{n}$ be an open set
such that $[\Omega]$ is sharply bounded, and let $u\in\gsf_{\tau}([\Omega])$
be a tempered CGF (identified with the corresponding GSF). Finally,
let $\phi\in\mathcal{S}(\R^{n})$ be a dumping measure. Then
\[
\mathcal{F}_{\hat{\phi}}(u)=\mathcal{F}\left[u\cdot\hat{\phi}((-)\cdot\diff{\rho})\right]=\mathcal{F}\left[u\cdot\mathcal{F}(\phi)((-)\cdot\diff{\rho})\right].
\]
\end{thm}

\begin{proof}
Def.~\eqref{def:damping} yields
\begin{align*}
\mathcal{F}_{\hat{\phi}}(u)(\omega) & =\int_{\R^{n}}u(x)e^{-ix\cdot\omega}\,\diff{_{\hat{\phi}}x}=\\
 & =\int_{\R^{n}}u(x)e^{-ix\cdot\omega}\widehat{\diff{\rho}\odot\phi}(x)\,\diff{x}=\\
 & =\int_{\R^{n}}u(x)e^{-ix\cdot\omega}\left(\diff{\rho}\diamond\hat{\phi}\right)(x)\,\diff{x}=\\
 & =\int_{\R^{n}}u(x)e^{-ix\cdot\omega}\hat{\phi}(\diff{\rho}\cdot x)\,\diff{x}=\\
 & =\mathcal{F}\left[u\cdot\hat{\phi}((-)\cdot\diff{\rho})\right](\omega)=\\
 & =\mathcal{F}\left[u\cdot\mathcal{F}(\phi)((-)\cdot\diff{\rho})\right](\omega),
\end{align*}
where, in the last equality, we applied \eqref{eq:preservRapDecr}.
\end{proof}

\section{Examples and applications\label{sec:Examples-and-applications}}

In this section we present an initial study of possible applications
of the HFT. Our main aim is to highlight the new potentialities of
the theory. For example, thanks to the possibility of applying the
HFT also to non-tempered GF, the next deductions are fully rigorous,
even if they correspond to the frequently used statement: ``proceeding
formally, we obtain...''. We also note that the FT is often used
to prove necessary conditions: if the solution $y$ satisfies a given
differential equation, then necessarily $y=\ldots$ In the following,
we propose an attempt to also reverse this implication, even if this
depends on a suitable \emph{extensibility property}: If the HFT of
the given differential equation holds on some space, e.g.~$\mathring{K}\times\RC{\rho}_{\ge0}$,
then it also holds on the entire $\RC{\rho}\times\RC{\rho}_{\ge0}$.
We discuss some sufficient conditions for this property to hold, but
a thorough study of this condition is out of the scope of the present
work.

\subsection{Applications of HFT to ordinary differential equations}

\subsubsection*{The simplest ODE}

We first consider the following, apparently trivial but actually meaningful,
example: 
\begin{equation}
y'=y,\ y\left(0\right)=c,\ y\in\gsf\left(\left[-k,k\right]\right),\ c\in\rcrho,\label{eq:first_order_ode}
\end{equation}
where $k=-\log\left(\diff{\rho}\right)$. Since we do not impose limitations
on the initial value $c$, this simple example clearly shows the possibilities
of the HFT to manage non tempered generalized functions. Applying
the HFT to both sides of \eqref{eq:first_order_ode} and using the
derivation formula \eqref{eq:DerRule1}, we get 
\begin{equation}
\mathcal{F}_{k}\left(y\right)=\Delta_{1k}y+i\omega\mathcal{F}_{k}\left(y\right).\label{eq:FT_of_first_order_ode}
\end{equation}
Set for simplicity $\Delta_{y}\left(\omega\right):=\Delta_{1k}y\left(\omega\right)=y(k)e^{-ik\omega}-y(-k)e^{ik\omega}$
and note that the function $\Delta_{y}$ does not depend on the whole
function $y$ but only on the two values $y(\pm k)$. We get $\mathcal{F}_{k}\left(y\right)\left(\omega\right)=\frac{\Delta_{y}\left(\omega\right)}{1-i\omega}$,
and applying the Fourier inversion Thm.~\ref{thm:FIT}, we obtain
\begin{equation}
y(x)=\lim_{h\to+\infty}\mathcal{F}_{h}^{-1}\left(\frac{\Delta_{y}\left(\omega\right)}{1-i\omega}\right)(x)\quad\forall x\in\mathring{K}.\label{eq:expNec}
\end{equation}
Using the initial condition in \eqref{eq:first_order_ode}, we have
\begin{equation}
y\left(0\right)=\lim_{h\to+\infty}\mathcal{F}_{h}^{-1}\left(\frac{\Delta_{y}\left(\omega\right)}{1-i\omega}\right)\left(0\right)=\lim_{h\to+\infty}\int_{-h}^{h}\frac{\Delta_{y}\left(\omega\right)}{1-i\omega}\,\diff{\omega}=c.\label{eq:C}
\end{equation}
Clearly, e.g.~by separation of variables, \eqref{eq:first_order_ode}
necessarily yields $y(x)=ce^{x}$ for all $x\in[-k,k]$. Therefore,
$y\left(k\right)=ce^{-\log\diff{\rho}}=\frac{c}{\diff{\rho}}$, $y\left(-k\right)=ce^{\log\diff{\rho}}=c\diff{\rho}$
and $\Delta_{y}\left(\omega\right)=c\diff{\rho^{i\omega-1}}-c\diff{\rho^{-i\omega+1}}$
because $\diff{\rho^{i\omega}}=e^{-ik\omega}$.

Vice versa, take $a$, $b\in\rti$, and set
\[
\Delta_{a,b}^{k}(\omega):=a\cdot e^{-ik\omega}-b\cdot e^{ik\omega}\quad\forall\omega\in\rti.
\]
We also assume the following compatibility conditions on $a$, $b$:
\begin{equation}
\begin{cases}
a=\lim_{h\to+\infty}\mathcal{F}_{h}^{-1}\left(\frac{\Delta_{a,b}^{k}(\omega)}{1-i\omega}\right)(k)\\
b=\lim_{h\to+\infty}\mathcal{F}_{h}^{-1}\left(\frac{\Delta_{a,b}^{k}(\omega)}{1-i\omega}\right)(-k)
\end{cases}\label{eq:compCondExp}
\end{equation}
We will see in Rem.~\ref{rem:exp}.\ref{enu:expOverdet} that actually
these conditions overdetermine $a$, $b$. Set
\begin{equation}
y(x):=\lim_{h\to+\infty}\mathcal{F}_{h}^{-1}\left(\frac{\Delta_{a,b}^{k}(\omega)}{1-i\omega}\right)(x)\in\rti\quad\forall x\in K,\label{eq:def_y_exp}
\end{equation}
where we also assumed that the limit in \eqref{eq:def_y_exp} exists.
Thereby, \eqref{eq:def_y_exp} and \eqref{eq:compCondExp} imply $\Delta_{a,b}^{k}(\omega)=\Delta_{1k}y(\omega)$.
Now, apply $\mathcal{F}_{k}$ to both sides of \eqref{eq:def_y_exp}
and use Thm.~\ref{thm:contResult}, Thm.~\ref{thm:FIT} to get
\begin{align}
\mathcal{F}_{k}(y)(\omega) & =\mathcal{F}_{k}\left(\lim_{h\to+\infty}\mathcal{F}_{h}^{-1}\left(\frac{\Delta_{1k}y(\omega)}{1-i\omega}\right)\right)(\omega)\\
 & =\lim_{h\to+\infty}\mathcal{F}_{k}\left(\mathcal{F}_{h}^{-1}\left(\frac{\Delta_{1k}y(\omega)}{1-i\omega}\right)\right)(\omega)\\
 & =\lim_{h\to+\infty}\mathcal{F}_{h}^{-1}\left(\mathcal{F}_{k}\left(\frac{\Delta_{1k}y(\omega)}{1-i\omega}\right)\right)(\omega)\label{eq:exchange}\\
 & =\frac{\Delta_{1k}y(\omega)}{1-i\omega}\quad\forall\omega\in\mathring{K}.\label{eq:antecExtExp0}
\end{align}
Note that in \eqref{eq:exchange} we used Fubini's theorem to exchange
$\mathcal{F}_{h}^{-1}$ with $\mathcal{F}_{k}$. We can now reverse
all the calculations leading us to the necessary condition \eqref{eq:expNec}
to obtain
\begin{equation}
\mathcal{F}_{k}(y)(\omega)=\mathcal{F}_{k}(y')(\omega)\quad\forall\omega\in\mathring{K}.\label{eq:antecExtExp}
\end{equation}
We would like to apply $\mathcal{F}_{h}$ to both sides of \eqref{eq:antecExtExp}
and then take $\lim_{h\to+\infty}$. However, to make this step, we
need equality \eqref{eq:antecExtExp} to hold for all $\omega\in\rti$
because $h\to+\infty$, $h\in\rti$. We therefore assume that \eqref{eq:antecExtExp}
can be extended from $\mathring{K}$ to the entire $\rti$ (note that
both sides of \eqref{eq:antecExtExp} are GSF defined on $\rti$):
\begin{equation}
\mathcal{F}_{k}(y)|_{\mathring{K}}=\mathcal{F}_{k}(y')|_{\mathring{K}}\Rightarrow\ \mathcal{F}_{k}(y)=\mathcal{F}_{k}(y')\text{ on }\rti.\label{eq:extExp}
\end{equation}
This is the aforementioned \emph{extensibility property for the ODE}
$y'=y$. Under this assumption, we can apply the Fourier inversion
Thm.~\ref{thm:FIT} to obtain $y(x)=y'(x)$ for all $x\in\mathring{K}$,
and hence $y=y'$ by continuity. We can simply, and more generally,
state the extensibility property saying: \emph{if the HFT of the differential
equation holds in $\mathring{K}$ for the function $y$, then it holds
everywhere}.
\begin{rem}
\label{rem:exp}\ 
\begin{enumerate}
\item \label{enu:expOverdet}Since $y(x)=y(0)e^{x}$, compatibility conditions
\eqref{eq:compCondExp} imply $a=y(0)e^{k}$ and $b=y(0)e^{-k}$.
If $y(0)$ is invertible, we obtain $a=be^{2k}$. Therefore, \eqref{eq:compCondExp}
overdetermine the constants $a$, $b$.
\item Using the notion of hyperseries we already mentioned in Sec.~\ref{subsec:The-heuristic-motivation}
(see \cite{TiwGio21}), we can say that if both $\mathcal{F}_{k}(y)$
and $\mathcal{F}_{k}(y')$ can be expanded in Taylor hyperseries (we
can say that they are \emph{real hyper analytic}), i.e.~if for some
$\bar{\omega}\in\mathring{K}$ and for all $\omega\in\rti$, we have
\begin{align*}
\mathcal{F}_{k}(y)(\omega) & =\hypersum{\rho}{\sigma}\frac{\mathcal{F}_{k}(y)^{(n)}(\bar{\omega})}{n!}\cdot(\omega-\bar{\omega})^{n}\\
\mathcal{F}_{k}(y')(\omega) & =\hypersum{\rho}{\sigma}\frac{\mathcal{F}_{k}(y')^{(n)}(\bar{\omega})}{n!}\cdot(\omega-\bar{\omega})^{n},
\end{align*}
then \eqref{eq:extExp} holds, because \eqref{eq:antecExtExp} yields
$\mathcal{F}_{k}(y)^{(n)}(\bar{\omega})=\mathcal{F}_{k}(y')^{(n)}(\bar{\omega})$
for all $n\in\N$.
\end{enumerate}
\end{rem}

Even if this ODE is the simplest one, we want to underline our deductions
with the following statements:
\begin{thm}
\label{thm:HFTExp}If $k=-\log\diff\rho$, $y\in\gsf(K)$, $\Delta_{y}:=\Delta_{1k}y$
and $y'=y$, then
\[
y(x)=\lim_{h\to+\infty}\mathcal{F}_{h}^{-1}\left(\frac{\Delta_{y}\left(\omega\right)}{1-i\omega}\right)(x)\quad\forall x\in\mathring{K}.
\]
\end{thm}

\noindent The sufficient condition deduction corresponds to the following
\begin{thm}
\label{thm:ExpSuff}Let $a$, $b\in\rti$, and set
\begin{align*}
 & \Delta_{a,b}^{k}(\omega):=a\cdot e^{-ik\omega}-b\cdot e^{ik\omega}\quad\forall\omega\in\rti\\
 & \forall x\in K\,\exists\lim_{h\to+\infty}\mathcal{F}_{h}^{-1}\left(\frac{\Delta_{a,b}^{k}(\omega)}{1-i\omega}\right)(x)=:y(x)\in\rti.
\end{align*}
Assume the compatibility conditions \eqref{eq:compCondExp} and the
extensibility property for the ODE $y'=y$, i.e.~\eqref{eq:extExp}.
Then
\[
\begin{cases}
y'=y\text{ on }K\\
y(0)=\lim_{h\to+\infty}\int_{-h}^{h}\frac{\Delta_{a,b}^{k}(\omega)}{1-i\omega}\,\diff\omega.
\end{cases}
\]
\end{thm}

We finally underscore that:
\begin{enumerate}[label=(\alph*)]
\item In the classical theory, the lacking of the term $\Delta_{1k}y(\omega)$
does not allow to obtain the non-tempered solution for $c\ne0$: in
other words, if $\Delta_{1k}y=0$, then \eqref{eq:C} implies that
necessarily $c=0$.
\item In the previous deduction, it is clearly important that the HFT can
be applied to all the GF of the space $\gsf(K)$.
\item Compare \eqref{eq:expNec} with Example \ref{exa:exp} to note that
if $c\ge r\in\R_{>0}$, then in \eqref{eq:expNec} we are considering
the inverse HFT of a GSF which always takes infinite values for all
finite $\omega$. Clearly, this strongly motivates the use of a non-Archimedean
framework for this type of problems.
\item All our results, in particular the inversion Thm.~\ref{thm:FIT},
hold for an arbitrary infinite number $k$. In this particular case,
we considered $k$ of logarithmic type to get moderateness of the
exponential function.
\end{enumerate}

\subsubsection*{General constant coefficient ODE}

Let us consider an arbitrary $n$-th order constant (generalized)
coefficient ODE
\begin{equation}
a_{n}y^{\left(n\right)}+\ldots a_{1}y^{\left(1\right)}+a_{0}y=g,\quad y,g\in\gsf(\left[-k,k\right]),\ a_{j}\in\rcrho,\ n\in\mathbb{N}_{\geq1}.\label{eq:non_homog_ode}
\end{equation}
Note that simply assuming to have a solution $y$ defined on the infinite
interval $[-k,k]$, already yields an implicit limitation on the coefficients
$a_{j}\in\rti$. In fact, the equation $y'-\frac{1}{\diff{\rho}}y=0$
has solution $y(x)=y(0)e^{x/\diff{\rho}}$, which is defined only
if $x\le-N\diff{\rho}\log\diff{\rho}\approx0$ for some $N\in\N$.
Thereby, its domain will never be of the type $[-k,k]$ unless $y(0)=0$.
By applying the HFT to both sides of equation \eqref{eq:non_homog_ode},
the differential equation is converted into the algebraic equation
\begin{equation}
P\cdot\mathcal{F}_{k}\left(y\right)+\Delta_{y}=\mathcal{F}_{k}\left(g\right),\label{eq:hft_applied_to_nonhomog_ode}
\end{equation}
where 
\[
P\left(\omega\right):=\sum_{j=0}^{n}a_{j}\left(i\omega\right)^{j},
\]
and $\Delta_{y}\left(\omega\right)$ is the sum of all the extra terms
in Thm.~\ref{thm:thmProperties}.\ref{enu:prop8}, which in this
case becomes
\[
\Delta_{y}(\omega):=\sum_{j=1}^{n}a_{j}\cdot\sum_{p=1}^{j}(i\omega)^{j-p}\Delta_{1k}y^{(p-1)}(\omega)\quad\forall\omega\in\rti.
\]
Note that the function $\Delta_{y}$ depends on the points $y^{(p)}(\pm k)$
for $p=0,\ldots,n-1$. Assuming that $P(\omega)$ is invertible for
all $\omega\in K$, from \eqref{eq:hft_applied_to_nonhomog_ode} and
the inversion Thm\@.~\ref{thm:FIT}, we get
\begin{equation}
y(x)=\lim_{h\to+\infty}\mathcal{F}_{h}^{-1}\left(\frac{\mathcal{F}_{k}(g)-\Delta_{y}}{P}\right)(x)\quad\forall x\in\mathring{K}.\label{eq:finalODE}
\end{equation}
Proceeding as in the previous example, i.e.~using again the inversion
Thm\@.~\ref{thm:FIT}, the differentiation formula \eqref{eq:DerRule1}
and assuming suitable compatibility and extensibility conditions,
we can actually prove that \eqref{eq:finalODE} yields a solution
of \eqref{eq:non_homog_ode}. For a generalization to GSF of the usual
results about $n$-th order constant generalized coefficient ODE,
see \cite{LuGi16a}.

\subsubsection*{Airy equation}

A simple example of non-constant coefficient linear ODE is given by
the Airy equation 
\begin{gather}
u''(x)-x\cdot u(x)=0,\quad u\in\gsf(\left[-k,k\right],\rcrho).\label{eq:airy}
\end{gather}
By applying the derivative formulas Thm.~\ref{thm:thmProperties}.\ref{enu:prop8}
and Thm.~\ref{thm:thmProperties}.\ref{enu:prop9}, we get 
\[
i\omega^{2}\mathcal{F}_{k}\left(u\right)+\omega\Delta_{1k}u-i\Delta_{1k}u'-\mathcal{F}'_{k}\left(u\right)=0.
\]
Let us now set $\Delta_{u}\left(\omega\right):=\omega\Delta_{1k}u\left(\omega\right)-i\Delta_{1k}u'\left(\omega\right)$,
$\forall\omega\in\rti$. Once again, the function $\Delta_{u}$ depends
on the points $u\left(\pm k\right)$ and $u'\left(\pm k\right)$.
\begin{equation}
\mathcal{F}'_{k}\left(u\right)-i\omega^{2}\mathcal{F}_{k}\left(u\right)=\Delta_{u}.\label{eq:non_CC_ode}
\end{equation}
Equation \eqref{eq:non_CC_ode} is a first order non-constant coefficient,
non-homogeneous generalized ODE with respect to the variable $\omega$.
We can solve it e.g.~by considering the integrating factor $\mu\left(\omega\right):=e^{\int_{0}^{\omega}-iz^{2}\,\diff{z}}=e^{-i\frac{\omega^{3}}{3}}$.
Then, the solution of \eqref{eq:non_CC_ode} is given by 
\[
\mathcal{F}_{k}\left(u\right)\left(\omega\right)=\frac{\intop_{0}^{\omega}\mu\left(z\right)\Delta_{u}\left(z\right)\,\diff z+c}{\mu\left(\omega\right)}=\frac{\intop_{0}^{\omega}e^{-\frac{iz^{3}}{3}}\Delta_{u}\left(z\right)\,\diff z+c}{e^{-\frac{i\omega^{3}}{3}}}\quad\forall\omega\in\rti,
\]
where $c:=\mathcal{F}_{k}(u)(0)\in\rcrho$. Finally, we apply the
inversion Thm.~\ref{thm:FIT} and substitute $\Delta_{u}\left(\omega\right)$
to recover the original function
\begin{align*}
u(x) & =\lim_{h\to+\infty}\mathcal{F}_{h}^{-1}\left(\frac{\intop_{0}^{\omega}e^{-\frac{iz^{3}}{3}}\Delta_{u}\left(z\right)\diff z+c}{e^{-\frac{i\omega^{3}}{3}}}\right)(x)=\\
 & =\lim_{h\to+\infty}\mathcal{F}_{h}^{-1}\left(\frac{\intop_{0}^{\omega}e^{-\frac{iz^{3}}{3}}\Delta_{u}\left(z\right)\diff z}{e^{-\frac{i\omega^{3}}{3}}}\right)(x)+\lim_{h\to+\infty}\frac{c}{\pi}\int_{0}^{h}\cos\left(\frac{\omega^{3}}{3}+\omega x\right)\,\diff\omega\\
 & =\lim_{h\to+\infty}\frac{1}{\pi}\int_{0}^{h}\cos\left(\omega x+\frac{\omega^{3}}{3}\right)\int_{0}^{\omega}e^{-i\left(hz+\frac{z^{3}}{3}\right)}\left(zu\left(k\right)-iu'\left(k\right)\right)\diff z\,\diff\omega\\
 & \phantom{=}-\lim_{h\to+\infty}\frac{1}{\pi}\int_{0}^{h}\cos\left(\omega x+\frac{\omega^{3}}{3}\right)\int_{0}^{\omega}e^{-i\left(-hz+\frac{z^{3}}{3}\right)}\left(zu\left(-k\right)-iu'\left(-k\right)\right)\diff z\,\diff\omega\\
 & \phantom{=}+\lim_{h\to+\infty}\frac{c}{\pi}\int_{0}^{h}\cos\left(\frac{\omega^{3}}{3}+\omega x\right)\,\diff\omega\quad\forall x\in K.
\end{align*}
If we assume that $u(\pm k)\approx0$ and $u'\left(\pm k\right)\approx0$,
then we get the first Airy function up to infinitesimals $u(x)\approx c\cdot\text{Ai}(x)$.
Therefore, if $u(\pm k)\not\approx0$ or $u'\left(\pm k\right)\not\approx0$
and $c=0$, we must get, up to infinitesimals, a multiple of the second
Airy function (see e.g.~\cite{AbrSte72})
\[
\exists d\in\rti:\ u(x)\approx\text{Bi}(x)=\frac{d}{\pi}\int_{0}^{+\infty}\left\{ \exp\left(-\frac{t^{3}}{3}+xt\right)+\sin\left(\frac{t^{3}}{3}+xt\right)\right\} \diff t.
\]
We explicitly recall that $\text{Bi}(x)$ is of exponential order
as $x\to+\infty$ and hence it is not a tempered distribution, so
that classically we miss this solution.

\subsection{Applications of HFT to partial differential equations}

\subsubsection*{\label{subsec:The-wave-equation}The wave equation}

Let us consider the one dimensional (generalized) wave equation\textit{
\begin{equation}
\frac{\partial^{2}u}{\partial t^{2}}=c^{2}\frac{\partial^{2}u}{\partial x^{2}},\quad c\in\RC{\rho},\ u\in\Gsf{\rho}(\left[-k,k\right]\times\RC{\rho}_{\geq0}),\label{eq:wave}
\end{equation}
}where $c$ is invertible, and subject to the initial conditions at
$t=0$
\begin{align}
u(-,0) & =f,\quad\partial_{t}u(-,0)=g,\label{eq:ICwave}
\end{align}
Where $f$, $g\in\Gsf{\sigma}(\rti)$. As usual, we directly apply
the HFT $\mathcal{F}_{k}$ with respect to the variable $x$ to both
sides and then apply the derivation formula Thm.~\ref{thm:thmProperties}.\ref{enu:prop8}
to the right hand side
\[
\mathcal{F}_{k}\left(\frac{\partial^{2}u}{\partial t^{2}}\right)=c^{2}\mathcal{F}_{k}\left(\frac{\partial^{2}u}{\partial x^{2}}\right),
\]
\[
\frac{\partial^{2}\mathcal{F}_{k}\left(u\right)}{\partial t^{2}}=-c^{2}\omega^{2}\mathcal{F}_{k}\left(u\right)+i\omega\Delta_{1k}u+\Delta_{1k}\left(\partial_{x}u\right).
\]
Note that also the $\Delta_{1k}$-terms refer to the variable $x$,
but the result is a function of $t$. More precisely, set
\begin{equation}
\Delta_{u}\left(\omega,t\right):=i\omega\Delta_{1k}\left(u(-,t)\right)(\omega)+\Delta_{1k}\left(\partial_{x}u(-,t)\right)(\omega).\label{eq:Cwave}
\end{equation}
The function $\Delta_{u}$ does not depend on the whole functions
$u$ and $\partial_{x}u$ but only on its boundary values: $u\left(\pm k,-\right)$
and $\partial_{x}u\left(\pm k,-\right)$, which are functions of $t$.
Hence, we get 
\begin{equation}
\frac{\partial^{2}\mathcal{F}_{k}\left(u\right)}{\partial t^{2}}(\omega,-)+c^{2}\omega^{2}\mathcal{F}_{k}\left(u\right)(\omega,-)=\Delta_{u}(\omega,-)\quad\forall\omega\in\rti.\label{eq:SOnHODE}
\end{equation}
We obtain, for each fixed $\omega$, a constant (generalized) coefficient,
\emph{non}-ho\-mo\-ge\-ne\-ous, second order ODE in the unknown
$\mathcal{F}_{k}\left(u\right)\left(\omega,-\right)$. Clearly, \eqref{eq:SOnHODE}
already highlights a difference with the classical FT, where $\Delta_{u}=0$.
To solve equation \eqref{eq:SOnHODE}, we can use the standard method
of variation of parameters to get
\begin{align}
\mathcal{F}_{k}(u)(\omega,t) & =d_{2}(\omega)tS(c\omega t)+d_{1}(\omega)\cos(c\omega t)+\nonumber \\
 & \phantom{=}+tS(c\omega t)\int_{1}^{t}\Delta_{u}(\omega,s)\cos(c\omega s)\,\diff s-\\
 & \phantom{=}-\cos(c\omega t)\int_{1}^{t}s\Delta_{u}(\omega,s)S(c\omega s)\,\diff s,\label{eq:Fwave}\\
S(z) & :=\frac{1}{2}\int_{-1}^{1}\cos(zt)\,\diff t.\label{eq:sinx/x}
\end{align}
More precisely, in the previous step we applied the general theory
of linear constant generalized coefficient, non-homogeneous ODE developed
in \cite{LuGi16a}, which generalizes the classical theory proving
that the space of all the solutions is a $2$-dimensional $\rti$-module,
generated in this case by $tS(c\omega t)$ and $\cos(c\omega t)$,
and translated by a particular solution of \eqref{eq:SOnHODE}. Explicitly
note that every functions in \eqref{eq:Fwave} is a smooth function
or a GSF. We also note that the functions $d_{1}$, $d_{2}$ satisfy
\begin{align*}
d_{1}(\omega) & =\mathcal{F}_{k}(f)(\omega)-\int_{0}^{1}s\Delta_{u}(\omega,s)S(c\omega s)\,\diff s\\
d_{2}(\omega) & =\mathcal{F}_{k}(g)(\omega)+\int_{0}^{1}\Delta_{u}(\omega,s)\cos(c\omega s)\,\diff s.
\end{align*}
They hence depend on the functions $f$, $g$ of the initial conditions
\eqref{eq:ICwave}, but also on the unknown function $u$ because
of \eqref{eq:Cwave}. Finally, applying the inversion Thm.~\ref{thm:FIT},
for all the interior points $x\in\mathring{K}$ and all $t\in\RC{\rho}_{\ge0}$,
we get 
\begin{align*}
u(x,t) & =\lim_{h\to+\infty}\left\{ \phantom{\int_{1}^{t}}\!\!\!\!\!\!\!\!\mathcal{F}_{h}^{-1}\left(d_{2}(\omega)tS(c\omega t)+d_{1}(\omega)\cos(c\omega t)\right)(x,t)+\right.\\
 & \phantom{=}+\mathcal{F}_{h}^{-1}\left(tS(c\omega t)\int_{1}^{t}\Delta_{u}(\omega,s)\cos(c\omega s)\,\diff s\right)(x,t)-\\
 & \phantom{=}-\left.\mathcal{F}_{h}^{-1}\left(\cos(c\omega t)\int_{1}^{t}s\Delta_{u}(\omega,s)S(c\omega s)\,\diff s\right)(x,t)\right\} .
\end{align*}
Following the usual calculations, the first summand yields the following
generalizations of the d'Alembert formula
\begin{align}
u(x,t) & =\frac{1}{2}\left[f(x-ct)+f(x+ct)\right]+\frac{1}{2c}\int_{x-ct}^{x+ct}g(x')\,\diff x'+\nonumber \\
 & \phantom{=}+\lim_{h\to+\infty}\left\{ \phantom{\int_{1}^{t}}\!\!\!\!\!\!\!\!\right.\mathcal{F}_{h}^{-1}\left(tS(c\omega t)\int_{1}^{t}\Delta_{u}(\omega,s)\cos(c\omega s)\,\diff s\right)(x,t)-\nonumber \\
 & \phantom{=}-\left.\mathcal{F}_{h}^{-1}\left(\cos(c\omega t)\int_{1}^{t}s\Delta_{u}(\omega,s)S(c\omega s)\,\diff s\right)(x,t)\right\} \label{eq:waveSol2}
\end{align}
for all the  interior points $x\in\mathring{K}$ and all $t\in\RC{\sigma}_{\ge0}$.
Note explicitly that \eqref{eq:waveSol2} does not yield a uniqueness
result because $\Delta_{u}$ depends on $u\left(\pm k,-\right)$ and
$\partial_{x}u\left(\pm k,-\right)$ (see \eqref{eq:Cwave}). This
proves the following
\begin{thm}
\label{thm:waveGlobal}Let $f$, $g\in\Gsf{\rho}(\left[-k,k\right])$
and assume that $u\in\Gsf{\rho}(\left[-k,k\right]\times\RC{\rho}_{\geq0})$
is a solution of the wave equation \eqref{eq:wave} subject to the
initial conditions \eqref{eq:ICwave}. Then necessarily $u(x,t)$
satisfies relation \eqref{eq:waveSol2} at all  interior points $x\in\mathring{K}$
and all $t\in\RC{\rho}_{\ge0}$. In particular, if we also assume
that $u\left(\pm k,-\right)=0=\partial_{x}u\left(\pm k,-\right)$,
we get the usual d'Alembert solution, and if in addition we take $f=0$,
$g=\delta$, we get the wave kernel $u(x,t)=\frac{1}{2c}\left[H(x+ct)-H(x-ct)\right]$.
\end{thm}

Now, we want to see how to revert the previous steps to obtain a sufficient
condition. Given GSF $F_{+}$, $F_{-}$, $G_{+}$, $G_{-}\in\Gsf{\rho}(\RC{\rho}_{\ge0})$,
set for all $\omega\in\RC{\rho}$ and all $t\in\RC{\rho}_{\ge0}$
\begin{equation}
\Delta_{\pm}^{k}(\omega,t):=i\omega\left(F_{+}(t)e^{-ik\cdot\omega}-F_{-}(t)e^{ik\cdot\omega}\right)+\left(G_{+}(t)e^{-ik\cdot\omega}-G_{-}(t)e^{ik\cdot\omega}\right)\label{eq:C_pmWave}
\end{equation}
Let $W_{\pm}^{k}(x,t)$ be the function defined by the right hand
side of \eqref{eq:waveSol2} with $\Delta_{\pm}^{k}$ instead of $\Delta_{u}$,
i.e. for all $(x,t)\in K\times\rti_{\ge0}$
\begin{align}
W_{\pm}^{k}(x,t) & =\frac{1}{2}\left[f(x-ct)+f(x+ct)\right]+\frac{1}{2c}\int_{x-ct}^{x+ct}g(x')\,\diff x'+\nonumber \\
 & \phantom{=}+\lim_{h\to+\infty}\left\{ \phantom{\int_{1}^{t}}\!\!\!\!\!\!\!\!\right.\mathcal{F}_{h}^{-1}\left(tS(c\omega t)\int_{1}^{t}\Delta_{\pm}^{k}(\omega,s)\cos(c\omega s)\,\diff s\right)(x,t)-\nonumber \\
 & \phantom{=}-\left.\mathcal{F}_{h}^{-1}\left(\cos(c\omega t)\int_{1}^{t}s\Delta_{\pm}^{k}(\omega,s)S(c\omega s)\,\diff s\right)(x,t)\right\} \label{eq:Wfun}
\end{align}
(we are clearly implicitly assuming that such limit exists, which
is, on the other hand a necessary condition of the previous deduction).
We assume the compatibility conditions
\begin{align}
F_{+}(t) & =W_{\pm}^{k}(k,t),\quad F_{-}(t)=W_{\pm}^{k}(-k,t)\nonumber \\
G_{+}(t) & =\partial_{x}W_{\pm}^{k}(k,t),\quad G_{-}(t)=\partial_{x}W_{\pm}^{k}(-k,t)\label{eq:compCondWave}
\end{align}
(as usual, we will see that they are redundant because having a solution
of the DE imply further restrictions on these functions). Finally,
set $u(x,t):=W_{\pm}^{k}(x,t)\in\RC{\rho}$ for all $(x,t)\in K\times\RC{\rho}_{\ge0}$.
Conditions \eqref{eq:compCondWave} and \eqref{eq:C_pmWave} imply
\begin{equation}
\Delta_{\pm}^{k}(\omega,t)=i\omega\Delta_{1k}\left(u(-,t)\right)(\omega)+\Delta_{1k}\left(\partial_{x}u(-,t)\right)(\omega),\label{eq:C_pm_delta}
\end{equation}
which is an important equality to reverse all the previous steps.
In fact, applying $\mathcal{F}_{k}$ to both side of the equality
$u(x,t)=W_{\pm}^{k}(x,t)$ we get
\begin{align}
\mathcal{F}_{k}\left(u\right)(\omega,t) & =\mathcal{F}_{k}\left(\frac{1}{2}\left[f(x-ct)+f(x+ct)\right]+\frac{1}{2c}\int_{x-ct}^{x+ct}g(x')\,\diff x'\right)(\omega,t)+\nonumber \\
 & \phantom{=}+\mathcal{F}_{k}\left(\lim_{h\to+\infty}\left\{ \phantom{\int_{1}^{t}}\!\!\!\!\!\!\!\!\right.\mathcal{F}_{h}^{-1}\left(tS(c\omega t)\int_{1}^{t}\Delta_{u}(\omega,s)\cos(c\omega s)\,\diff s\right)\right.-\nonumber \\
 & \phantom{=}-\left.\left.\mathcal{F}_{h}^{-1}\left(\cos(c\omega t)\int_{1}^{t}s\Delta_{u}(\omega,s)S(c\omega s)\,\diff s\right)(x,t)\right\} \right)(\omega,t).\label{eq:HFTWaveSol1}
\end{align}
The first summand can be written as
\begin{multline*}
\mathcal{F}_{k}\left(\frac{1}{2}\left[f(x-ct)+f(x+ct)\right]+\frac{1}{2c}\int_{x-ct}^{x+ct}g(x')\,\diff x'\right)(\omega,t)=\\
=\mathcal{F}_{k}\left(\lim_{h\to+\infty}\mathcal{F}_{h}^{-1}\left(d_{2}(\omega)tS(c\omega t)+d_{1}(\omega)\cos(c\omega t)\right)\right)(\omega,t).
\end{multline*}
As above, we can exchange $\mathcal{F}_{k}$ and $\lim_{h\to+\infty}$
because of Thm.~\ref{thm:contResult}. Thereby, applying the Fourier
inversion Thm.~\ref{thm:FIT}, we obtain
\begin{align}
\mathcal{F}_{k}(u)(\omega,t) & =d_{2}(\omega)tS(c\omega t)+d_{1}(\omega)\cos(c\omega t)+\nonumber \\
 & \phantom{=}+tS(c\omega t)\int_{1}^{t}\Delta_{\pm}^{k}(\omega,s)\cos(c\omega s)\,\diff s-\\
 & \phantom{=}-\cos(c\omega t)\int_{1}^{t}s\Delta_{\pm}^{k}(\omega,s)S(c\omega s)\,\diff s,\nonumber 
\end{align}
which holds for all interior point $\omega\in K$ and for all $t\in\RC{\rho}$.
Reversing the previous calculations, we arrive at
\[
\frac{\partial^{2}\mathcal{F}_{k}\left(u\right)}{\partial t^{2}}(\omega,t)+c^{2}\omega^{2}\mathcal{F}_{k}\left(u\right)(\omega,t)=\Delta_{\pm}^{k}(\omega,t).
\]
Now, we can substitute \eqref{eq:C_pm_delta} and use the derivation
formula Thm.~\ref{thm:thmProperties}.\ref{enu:prop8}, to get
\[
\mathcal{F}_{k}\left(\frac{\partial^{2}u}{\partial t^{2}}\right)(\omega,t)=c^{2}\mathcal{F}_{k}\left(\frac{\partial^{2}u}{\partial x^{2}}\right)(\omega,t)\qquad\forall\omega\in\mathring{K}\,\forall t\in\RC{\rho}_{\ge0}.
\]
We finally assume the extensibility property for the wave equation:
If the HFT of the wave equation holds on $\mathring{K}\times\RC{\rho}_{\ge0}$,
then it also holds on $\RC{\rho}\times\RC{\rho}_{\ge0}$. This allows
us to apply $\mathcal{F}_{h}^{-1}$ to both sides and use again the
Fourier inversion theorem:
\begin{align*}
\mathcal{F}_{h}^{-1}\left(\mathcal{F}_{k}\left(\frac{\partial^{2}u}{\partial t^{2}}\right)\right) & =c^{2}\mathcal{F}_{h}^{-1}\left(\mathcal{F}_{k}\left(\frac{\partial^{2}u}{\partial x^{2}}\right)\right)\qquad\forall h\in\RC{\rho},\\
\lim_{h\to+\infty}\mathcal{F}_{h}^{-1}\left(\mathcal{F}_{k}\left(\frac{\partial^{2}u}{\partial t^{2}}\right)\right) & (x,t)=c^{2}\lim_{h\to+\infty}\mathcal{F}_{h}^{-1}\left(\mathcal{F}_{k}\left(\frac{\partial^{2}u}{\partial x^{2}}\right)\right)(x,t),
\end{align*}
\begin{equation}
\frac{\partial^{2}u}{\partial t^{2}}(x,t)=c^{2}\frac{\partial^{2}u}{\partial x^{2}}(x,t)\label{eq:rhoWave}
\end{equation}
for all $x\in\mathring{K}$ and all $t$, and hence also for all $x\in K$
by continuity. It is important to note that equation \eqref{eq:rhoWave}
implies the usual compatibility conditions (see e.g.~\cite{Joh91}
for similar calculations):
\begin{align*}
 & F_{+}(0)=f(h),\quad F'_{+}(0)=g(h),\quad F''_{+}(0)=c^{2}f''(h)\\
 & F_{-}(0)=f(-h),\quad F'_{-}(0)=g(-h),\quad F''_{-}(0)=c^{2}f''(-h)\\
 & g(k-ct)-f'(k-ct)=F'_{+}(t)-G_{+}(t)\\
 & g(-k-ct)+f'(-k-ct)=F'_{+}(t)+G_{+}(t)
\end{align*}
Therefore, conditions \eqref{eq:compCondWave} over-determine the
functions $F_{\pm}$ and $G_{\pm}$ which hence cannot be freely chosen.
In particular, if $f$, $g\in\Dgsf(\RC{\rho})$, and we take $F_{\pm}=G_{\pm}=0$,
then $u(\pm k,-)=0=\partial_{x}u(\pm k,-)$ for some $k\in\RC{\rho}$
sufficiently large, and hence the compatibility conditions \eqref{eq:ICwave}
always hold.

This proves the following
\begin{thm}
\label{thm:waveBack}Let $f$, $g\in\Gsf{\rho}(\left[-k,k\right])$,
$F_{+}$, $F_{-}$, $G_{+}$, $G_{-}\in\Gsf{\rho}(\RC{\rho}_{\ge0})$.
Define $\Delta_{\pm}^{k}$ as in \eqref{eq:C_pmWave}, $W_{\pm}^{k}$
as in \eqref{eq:Wfun} (assuming that the corresponding $\lim_{h\to+\infty}$
exists) and $u(x,t):=W_{\pm}^{k}(x,t)$ for all $(x,t)\in K\times\RC{\rho}_{\ge0}$.
Assume the compatibility conditions \eqref{eq:compCondWave}, and
the extensibility property: If $\mathcal{F}_{k}\left(\frac{\partial^{2}u}{\partial t^{2}}\right)=c^{2}\mathcal{F}_{k}\left(\frac{\partial^{2}u}{\partial x^{2}}\right)$
holds in $\mathring{K}\times\RC{\rho}_{\ge0}$, then it also holds
on $\RC{\rho}\times\RC{\rho}_{\ge0}$. Then $u$ satisfies the wave
equation on $K\times\RC{\rho}_{\ge0}$. In particular, if $F_{\pm}=G_{\pm}=0$,
then $u$ also satisfies the initial conditions \eqref{eq:ICwave}.
Finally, if $f$, $g\in\DGsf{\rho}(\RC{\rho})$, then the compatibility
conditions \eqref{eq:ICwave} always hold for some $k\in\RC{\rho}$
sufficiently large, and $u$ is given by the usual d'Alembert formula.
\end{thm}

\noindent Explicitly note that even the last case, $F_{\pm}=G_{\pm}=0$
and $f$, $g\in\Dgsf(\RC{\rho})$, includes for $f$ and $g$ a large
class of GSF, e.g.~non-linear operations $F\left(\left(\delta^{(p)}\right)^{a}\right)_{\substack{0<a\le A\\
0\le p\le P
}
}$ of Dirac delta and its derivatives, such that $F\left(0\right)=0$.

\subsubsection*{The Heat equation}

Let us consider the one dimensional (generalized) heat equation\textit{
\begin{equation}
\frac{\partial u}{\partial t}=a^{2}\frac{\partial^{2}u}{\partial x^{2}},\quad u\in\gsf(\left[-k,k\right]\times\rcrho_{\geq0}),\label{eq:Heat}
\end{equation}
}(where $a\in\rcrho_{>0}$, $t\leq-\frac{N}{a^{2}k^{2}}\log\left(\diff\rho\right)$,
$N\in\mathbb{N}_{>0}$) and subject to the initial conditions at $t=0$
\begin{equation}
u(-,0)=f,\label{eq:ICheat}
\end{equation}
where $f\in\gsf(\left[-k,k\right])$. Applying, as usual, the HFT
with respect to the variable $x$ to both sides of \eqref{eq:Heat}
and Thm.~\ref{thm:thmProperties}.\ref{enu:prop8} we get
\[
\frac{\partial\mathcal{F}_{k}\left(u\right)}{\partial t}=-a^{2}\omega^{2}\mathcal{F}_{k}\left(u\right)+i\omega\Delta_{1k}u+\Delta_{1k}\left(\partial_{x}u\right).
\]
For all $\omega\in\rti$, set
\begin{align*}
\Delta_{u}\left(\omega,t\right) & :=i\omega\Delta_{1k}\left(u(-,t)\right)+\Delta_{1k}\left(\partial_{x}u(-,t)\right).
\end{align*}
Therefore, we get 
\begin{equation}
\frac{\partial\mathcal{F}_{k}\left(u\right)}{\partial t}\left(\omega,-\right)+a^{2}\omega^{2}\mathcal{F}_{k}\left(u\right)\left(\omega,-\right)=\Delta_{u}\left(\omega,-\right)\quad\forall\omega\in\rcrho.\label{eq:FOnHODE}
\end{equation}
Solving \eqref{eq:FOnHODE} with the integrating factor $\mu\left(t\right):=e^{a^{2}\omega^{2}\intop_{0}^{t}\diff t}=e^{a^{2}\omega^{2}t}$
(which is well-defined if $\omega\in K$ because we assumed that $t\leq-\frac{N}{a^{2}k^{2}}\log\left(\diff\rho\right)$),
we have 
\[
\mathcal{F}_{k}\left(u\right)\left(\omega,t\right)=\frac{\intop_{0}^{t}e^{a^{2}\omega^{2}t}\Delta_{u}\left(\omega,t\right)\diff t+c(\omega)}{e^{a^{2}\omega^{2}t}},
\]
where $c(\omega):=\mathcal{F}_{k}\left(u\right)\left(\omega,0\right)=\mathcal{F}_{k}(f)(\omega)\in\rcrho$,
so that
\begin{align*}
\mathcal{F}_{k}\left(u\right)\left(\omega,t\right) & =e^{-a^{2}\omega^{2}t}\intop_{0}^{t}e^{a^{2}\omega^{2}t}\Delta_{u}\left(\omega,t\right)\diff t+\mathcal{F}_{k}(f)(\omega)e^{-a^{2}\omega^{2}t}\\
 & =e^{-a^{2}\omega^{2}t}\intop_{0}^{t}e^{a^{2}\omega^{2}t}\Delta_{u}\left(\omega,t\right)\diff t+\mathcal{F}_{k}(f)(\omega)\mathcal{F}\left(\frac{1}{2a\sqrt{\pi t}}e^{-\frac{x^{2}}{4a^{2}t}}\right)(\omega,t)\\
 & =:e^{-a^{2}\omega^{2}t}\intop_{0}^{t}e^{a^{2}\omega^{2}t}\Delta_{u}\left(\omega,t\right)\diff t+\mathcal{F}_{k}(f)(\omega)\mathcal{F}\left(H_{t}^{a}(x)\right)(\omega,t)\\
 & =e^{-a^{2}\omega^{2}t}\intop_{0}^{t}e^{a^{2}\omega^{2}t}\Delta_{u}\left(\omega,t\right)\diff t+\mathcal{F}_{k}\left(f*H_{t}^{a}\right)(\omega,t),
\end{align*}
where $H_{t}^{a}(x):=\frac{1}{2a\sqrt{\pi t}}e^{-\frac{x^{2}}{4a^{2}t}}$
is the heat kernel (which, in our setting, is a compactly supported
GSF). Finally, applying the inversion Thm.\ref{thm:FIT} and the convolution
formula Thm.~\ref{thm:thmProperties}.\ref{enu:prop10} we get
\begin{equation}
u(x,t)=(f*H_{t}^{a})(x,t)+\mathcal{F}_{k}^{-1}\left(e^{-a^{2}\omega^{2}t}\intop_{0}^{t}e^{a^{2}\omega^{2}t}\Delta_{u}\left(\omega,t\right)\diff t\right)(x,t).\label{eq:heatNec}
\end{equation}
As usual, if $\Delta_{u}\left(\omega,t\right)$ equals zero, we obtain
the classical solution. This proves the following
\begin{thm}
\label{thm:waveGlobal-1}Let $f\in\gsf(\left[-k,k\right])$, and assume
that $u\in\gsf(\left[-k,k\right]\times\rcrho_{\geq0})$ is a solution
of the heat equation \eqref{eq:Heat}, where $a\in\rcrho_{>0}$, $t\leq-\frac{N}{a^{2}k^{2}}\log\left(\diff\rho\right)$,
$N\in\mathbb{N}_{>0}$, subject to initial condition \eqref{eq:ICheat}.
Then necessarily $u(x,t)$ satisfies \eqref{eq:heatNec} at all interior
points $x\in\mathring{K}$ and all $t\in\rti_{\ge0}$. In particular,
if we also assume that $u(\pm k,-)=0=\partial_{x}u(\pm k,-)$, we
get the usual solution, and if in addition we take $f=\delta$, then
we get the heat kernel $u(x,t)=H_{t}^{a}(x)=\frac{1}{2a\sqrt{\pi t}}e^{-\frac{x^{2}}{4a^{2}t}}$.
\end{thm}

It is now clear how one can proceed to obtain a sufficient condition
for the heat equation similar to Thm.~\ref{thm:waveBack}, and for
this reason, we omit it here.

\subsubsection*{Laplace's equation}

Actually, we show this example only for the sake of completeness,
but we present here only a preliminary study. Let us consider the
one dimensional Laplace's equation \textit{
\begin{equation}
\frac{\partial^{2}u}{\partial x^{2}}+\frac{\partial^{2}u}{\partial y^{2}}=0\ ,u\in\gsf([-k,k]\times[0,-\frac{N}{k}\log\diff\rho]),\label{eq:laplace}
\end{equation}
}where $N\in\N_{>0}$, and subject to the boundary conditions at $y=0$
\begin{align}
u(-,0) & =f,\quad\partial_{y}u(-,0)=0,\label{eq:IClaplace}\\
u(\pm k,-) & =0,\quad\partial_{x}u(\pm k,-)=0,\label{eq:ICLaplace0}
\end{align}
where $f\in\gsf(\left[-k,k\right])$. Set $Y:=[0,-\frac{N}{k}\log\diff\rho]\subseteq\rti$.
By applying the HFT with respect to $x$ and Thm.~\ref{thm:thmProperties}.\ref{enu:prop8},
the problem is converted into 
\begin{equation}
\frac{\partial^{2}\mathcal{F}_{k}\left(u\right)}{\partial y^{2}}=\omega^{2}\mathcal{F}_{k}\left(u\right)\label{eq:FTLaplace}
\end{equation}
because of \eqref{eq:ICLaplace0}. The general solution of \eqref{eq:FTLaplace}
is
\begin{align*}
\mathcal{F}_{k}(u)(\omega,y) & =d_{1}(\omega)e^{\omega y}+d_{2}(\omega)e^{-\omega y},
\end{align*}
where the functions \textbf{$d_{1}$}, $d_{2}$ satisfy $\mathcal{F}_{k}(f)(\omega)=d_{1}(\omega)+d_{2}(\omega)$
and $\partial_{y}\mathcal{F}_{k}(u)(\omega,0)=\mathcal{F}_{k}\left(\partial_{y}u(-,0)\right)(\omega)=0=\omega d_{1}(\omega)-\omega d_{2}(\omega)$
because $\partial_{y}u(-,0)=0$. Since the set of invertible numbers
in $\RC{\rho}$ is dense in the sharp topology, we hence have
\[
d_{1}(\omega)=d_{2}(\omega)=\frac{1}{2}\mathcal{F}_{k}(f)(\omega).
\]
Note that $e^{\pm\omega y}$ is well defined for all $\omega\in K$
and all $y\in Y=[0,-\frac{N}{k}\log\diff\rho]$. Finally, applying
the inversion Thm.~\ref{thm:FIT}, we get
\begin{align}
u(x,y) & =\lim_{h\to+\infty}\mathcal{F}_{h}^{-1}\left(\mathcal{F}_{k}(f)\cdot\cosh(\omega y)\right)(x,y)\label{eq:LaplaceSol2}
\end{align}
for all $(x,y)\in\mathring{K}\times Y$. Note that a term of the type
$\mathcal{F}_{h}^{-1}\left(\cosh(\omega y)\right)$ cannot be considered
if $h\in\RC{\rho}$ is sufficiently large and $y$ is invertible because
$\cosh(\pm hy)$ would yield a non $\rho$-moderate number. Thereby,
we cannot transform the product in \eqref{eq:LaplaceSol2} into a
convolution.
\begin{thm}
Let $f\in\gsf(K)$, $N\in\N_{>0}$, and set $Y:=[0,-\frac{N}{k}\log\diff\rho]$.
Assume that $u\in\gsf(K\times Y)$ is a solution of the Laplace's
equation subject to the boundary conditions \eqref{eq:IClaplace}.
Then necessarily $u(x,y)$ satisfies relation \eqref{eq:LaplaceSol2}
for all $(x,y)\in\mathring{K}\times Y$. In particular, if $f=\delta$
then instead of $\mathcal{F}_{h}^{-1}\left(\mathcal{F}_{k}(f)\cdot\cosh(\omega y)\right)$
in \eqref{eq:LaplaceSol2} we can take $\mathcal{F}_{h}^{-1}\left(\mathbb{1}\cdot\cosh(\omega y)\right)$.
\end{thm}

\noindent It is well-known that if $f\in\mathcal{C}^{\infty}$ is
a classical smooth function, then $f$ is necessarily an analytic
function (see \cite{Had02} and e.g.~\cite{HaVaGo92}). Assuming
that the classical Hadamard result \cite{Had02} can be extended to
GSF, this would not exclude the case $f=\delta$, which can be proved
to be an analytic GSF.

To reverse the previous steps, assume that
\begin{align}
 & f\in\DGsf{\rho}(H),\quad H\fcmp\RC{\rho}\nonumber \\
 & \exists C,B\in\RC{\rho}\,\forall x\in H\,\forall j\in\N:\ \left|f^{(j)}(x)\right|\le C\cdot B^{j}.\label{eq:analytic}
\end{align}
Note that if $f\in\mathcal{C}^{\infty}$ is an ordinary smooth function
and $C$, $B\in\R$, assumption \eqref{eq:analytic} implies $f\in\mathcal{C}^{\omega}(H\cap\R)$,
i.e.~$f$ is real analytic. Moreover, for the sake of clarity, finally
note that $f\in\mathcal{C}^{\omega}(H\cap\R)\cap\DGsf{\rho}(H)$ implies
only $f(x_{\eps})\sim_{\sigma}0$ for all $[x_{\eps}]\notin H$, but
it does not imply $f=0$.\\
The previous assumptions and the Riemann-Lebesgue Lem.~\ref{lem:Rieman-Lebesgue}
yield that also $\mathcal{F}(f)$ is compactly supported in $\RC{\rho}$.
Define
\begin{equation}
u(x,y):=\lim_{h\to+\infty}\mathcal{F}_{h}^{-1}\left(\mathcal{F}(f)\cdot\cosh(\omega y)\right)(x,y)\qquad\forall x\in K\,\forall y\in Y.\label{eq:def_u_Laplace}
\end{equation}
Since $\mathcal{F}(f)\cdot\cosh(\omega y)$ is compactly supported
in $\RC{\rho}$ and satisfies the assumptions of Riemann-Lebesgue
Lem.~\ref{lem:Rieman-Lebesgue}, we have that also $u(-,y)$ is compactly
supported in $\RC{\rho}$. We hence assume to have considered $k$
sufficiently large so that
\begin{equation}
u(\pm k,-)=0=\partial_{x}u(\pm k,-).\label{eq:LapBound}
\end{equation}
We now proceed in the usual way:
\begin{align}
\mathcal{F}_{k}\left(u\right)(\omega) & =\mathcal{F}_{k}\left(\lim_{h\to+\infty}\mathcal{F}_{h}^{-1}\left(\mathcal{F}(f)\cdot\cosh(\omega y)\right)\right)(\omega)\nonumber \\
 & =\lim_{h\to+\infty}\mathcal{F}_{k}\left(\mathcal{F}_{h}^{-1}\left(\mathcal{F}(f)\cdot\cosh(\omega y)\right)\right)(\omega)\nonumber \\
 & =\lim_{h\to+\infty}\mathcal{F}_{h}^{-1}\left(\mathcal{F}_{k}\left(\mathcal{F}(f)\cdot\cosh(\omega y)\right)\right)(\omega)\nonumber \\
 & =\mathcal{F}(f)(\omega)\cdot\cosh(\omega y)\label{eq:LaplaceBeforeRev}
\end{align}
for all $\omega\in\mathring{K}$ and all $y\in Y$. Since $i\omega\Delta_{1k}\left(u(-,y)\right)+\Delta_{1k}\left(\partial_{x}u(-,y)\right)=0$,
from \eqref{eq:LaplaceBeforeRev} we can revert the previous calculations
to obtain
\begin{equation}
\mathcal{F}_{k}\left(\frac{\partial^{2}u}{\partial x^{2}}\right)(\omega,y)+\mathcal{F}_{k}\left(\frac{\partial^{2}u}{\partial y^{2}}\right)(\omega,y)=0\qquad\forall(\omega,y)\in\mathring{K}\times Y.\label{eq:antecedExtProp}
\end{equation}
Once again, we assume that our solution $u$ satisfies the extensibility
property, i.e.~that \eqref{eq:antecedExtProp} implies that the same
equation holds on $\RC{\rho}\times Y$. A final application of the
Fourier inversion Thm.~\ref{thm:FIT} yields that $u$ satisfies
\eqref{eq:laplace}. Setting $y=0$ in \eqref{eq:def_u_Laplace} we
obtain the first boundary condition in \eqref{eq:IClaplace}. Finally,
\[
\partial_{y}\left(\mathcal{F}_{h}^{-1}\left(\mathcal{F}(f)\cdot\cosh(\omega y)\right)\right)=\mathcal{F}_{h}^{-1}\left(\mathcal{F}(f)\cdot\omega\sinh(\omega y)\right)
\]
converges for $h\to+\infty$ for all fixed $y\in Y$ because $\mathcal{F}(f)\cdot\omega\sinh(\omega y)$
is compactly supported (see \eqref{eq:intCmpSuppLimit}). Since $Y$
is functionally compact, Thm.~\ref{thm:contResult} implies that
the convergence of these partial derivatives is actually uniform on
$Y$. Thereby $\partial_{y}u(x,0)=\left.\lim_{h\to+\infty}\partial_{y}\left(\mathcal{F}_{h}^{-1}\left(\mathcal{F}(f)\cdot\cosh(\omega y)\right)\right)\right|_{y=0}=0$.
\begin{thm}
Assume that $f$ satisfies \eqref{eq:analytic} and that for all $x\in K$,
$y\in Y:=[0,-\frac{N}{k}\log\diff\rho]$
\[
\exists\lim_{h\to+\infty}\mathcal{F}_{h}^{-1}\left(\mathcal{F}(f)\cdot\cosh(\omega y)\right)(x,y)=:u(x,y).
\]
Finally, assume that $u$ satisfies the extensibility property on
$\mathring{K}\times Y$ for the Laplace's equation and $k$ is sufficiently
large so that \eqref{eq:LapBound} holds. Then $u$ satisfies the
Laplace's equations \eqref{eq:laplace} and the boundary conditions
\eqref{eq:IClaplace} and \eqref{eq:ICLaplace0}.
\end{thm}

\subsection{Applications to convolution equations.}

Consider the following convolution equation in $y$
\begin{equation}
g=f*y,\label{eq:integralEquation}
\end{equation}
where we assume that $y$, $g\in\gsf(K)$ and $f\in\Dgsf(\rti)$.
As in the classical theory, we apply the convolution Thm.~\ref{thm:thmProperties}.\ref{enu:prop10}
to get 
\[
\mathcal{F}_{k}\left(g\right)=\mathcal{F}\left(f\right)\mathcal{F}_{k}\left(y\right).
\]
Assuming that $\mathcal{F}\left(f\right)(\omega)$ is invertible for
all $\omega\in K$, the inversion Thm.~\ref{thm:FIT} yields 
\[
y(t)=\lim_{h\to+\infty}\mathcal{F}_{h}^{-1}\left(\frac{\mathcal{F}_{k}\left(g\right)}{\mathcal{F}\left(f\right)}\right)(t),\quad\forall t\in\mathring{K}.
\]

\noindent For example, to highlight the differences with the classical
theory, let us consider the convolution equation $(\delta'+\delta)*y=\delta$
with $y(-1)=0$. We have $g=\delta$, and $f=\delta'+\delta$ so that
$\mathcal{F}(f)=i\omega\mathbb{1}+\mathbb{1}$, where as usual $\mathbb{1}=\mathcal{F}_{k}\left(\delta\right)$.
Since $\mathbb{1}(\omega)\in\rti$, the quantity $i\omega\mathbb{1}(\omega)+\mathbb{1}(\omega)$
is always invertible, and hence we obtain 
\[
y(t)=\mathcal{F}^{-1}\left(\frac{\mathbb{1}}{i\omega\mathbb{1}+\mathbb{1}}\right)(t),\quad\forall t\in\mathring{K}.
\]
It is easy to prove that $y(t)+y'(t)=\mathcal{F}^{-1}\left(1|_{K}\right)\left(t\right)=\frac{1}{2\pi}\int_{-k}^{k}e^{i\omega t}\,\diff t=\frac{k}{\pi}S(kt)$
(see \eqref{eq:sinx/x}) and hence $y(t)=e^{-t}\frac{k}{\pi}\int_{-1}^{t}S(kx)e^{x}\,\diff x$
e.g.~for all $\log(\diff\rho)\le t\le-\log(\diff\rho)$. Therefore
\[
y(t)=e^{-t}\int_{-1}^{t}\mathcal{F}^{-1}\left(1|_{K}\right)(s)e^{s}\,\diff s\approx e^{-t}\int_{-1}^{t}\delta(s)e^{s}\,\diff s=e^{-t}H(t),
\]
for all $t\in\mathring{K}$ which are far from the origin, i.e.~such
that $|t|\ge r\in\R_{>0}$ for some $r$.

\section{Conclusions}

In the introduction of this article, we motivated the natural attempts
of several authors to extend the domain of some kind of Fourier transform.
The HFT presented in this paper can be applied to the entire space
of all the GSF defined in the infinite interval $[-k,k]^{n}$. These
clearly include all tempered Schwartz distributions, all tempered
Colombeau GF, but also a large class of non-tempered GF, such as the
exponential functions, or non-linear examples like $\delta^{a}\circ\delta^{b}$,
$\delta^{a}\circ H^{b}$, $a$, $b\in\N$, etc.

We want to close by listing some features of the theory that allow
some of the main results we saw:
\begin{enumerate}
\item The power of a non-Archimedean language permeates the whole theory
since the beginning (e.g.~by defining GF as set-theoretical maps
with infinite values derivatives or in the use of sharp continuity).
This power turned out to be important also for the HFT: see the heuristic
motivation of the FT in Sec.~\ref{subsec:The-heuristic-motivation},
Example \ref{exa:uncertaintyDelta} about application of the uncertainty
principle to a delta distribution, or the HFT of exponential functions
in Example \ref{exa:exp} and in Sec.~\ref{sec:Examples-and-applications}.
\item The results presented here are deeply founded on a strong and flexible
theory of multidimensional integration of GSF on functionally compact
sets: the possibility to exchange hyperlimits and integration has
been used several times in the present work; the possibility to compute
$\eps$-wise integrals on intervals is another feature used in several
theorems and a key step in defining integration of compactly supported
GSF.
\item It can also be worth explicitly mentioning that the definition of
HFT is based on the classical formulas used only for rapidly decreasing
smooth functions and not on duality pairing. In our opinion, this
is a strong simplification that even underscores more the strict analogies
between ordinary smooth functions and GSF. All this in spite of the
fact that the ring of scalars $\rti$ is not a field and is not totally
ordered.
\item Important differences with respect to the classical theory result
from the Riemann-Lebesgue Lem.~\ref{lem:Rieman-Lebesgue} and the
differentiation formula \eqref{eq:DerRule1}. In the former case,
we explained these differences as a general consequence of integration
by part formula, i.e.~of the non-linear framework we are working
in, see Thm.~\ref{thm:R-Limp}. The compact support of the HFT $\mathbb{1}$
of Dirac's delta reveals to be very important in stating and proving
the preservation properties of HFT, see Sec.~\ref{sec:preservation}.
Surprisingly (the classical formula dates back at least to 1822),
in Sec.~\ref{sec:Examples-and-applications} we showed that the new
differentiation formula is very important to get out of the constrained
world of tempered solutions.
\item Finally, Example \ref{exa:uncertaintyDelta} of application of the
uncertainty principle, further suggests that the space $\gsf(K)$
may be a useful framework for quantum mechanics, so as to have both
GF and smooth functions in a space sharing several properties with
the classical $L^{2}(\R^{n})$ (but which, on the other hand, is a
\emph{graded} Hilbert space).
\end{enumerate}
\textbf{Acknowledgement.} We are grateful to Michael Oberguggenberger
and Sanja Konjik for their careful reading of our work. In particular,
for the stimulating discussions with Michael, which considerably improved
this manuscript.

\end{document}